\documentclass[10pt]{amsart}

\usepackage{amsmath, amsthm, amssymb, mathrsfs}
\usepackage{epsfig, graphicx,subfig}
\usepackage{verbatim,setspace}
\usepackage[colorlinks=true, citecolor=blue]{hyperref}

\newtheorem*{stahl}{Theorem S}
\newtheorem*{cond1}{Condition~GP}

\newtheorem{theorem}{Theorem}
\newtheorem{corollary}[theorem]{Corollary}
\newtheorem{proposition}[theorem]{Proposition}
\newtheorem{lemma}{Lemma}
\newtheorem{definition}{Definition}

\theoremstyle{remark}
\newtheorem{remark}{Remark}[theorem]

\numberwithin{equation}{section}

\newcommand{\D}     {\mathbb{D}}

\newcommand{\R}     {\mathbb{R}}
\newcommand{\C}     {\mathbb{C}}
\newcommand{\N}     {\mathbb{N}}
\newcommand{\Z}     {\mathbb{Z}}

\newcommand{\W}     {\mathcal{W}_\Delta}
\newcommand{\RS}        {\mathfrak{R}}

\newcommand{\map}   {\Phi}
\newcommand{\Arg}   {\mathrm{Arg}}

\newcommand{\Ai}        {\mathrm{Ai}}
\newcommand{\z} {\mathbf{z}}
\newcommand{\w} {\mathbf{w}}
\newcommand{\tr}    {\mathbf{t}}

\newcommand{\n}     {\mathbf{n}}
\newcommand{\mdp}        {\mathrm{mod~periods~}}
\newcommand{\cp}        {\mathrm{cp}}
\newcommand{\clos}  {\mathrm{clos}}
\newcommand{\const} {\mathrm{const.}}

\newcommand{\dist}  {\mathrm{dist}}
\newcommand{\diam}  {\mathrm{diam}}

\newcommand{\im}        {\mathrm{Im}}
\newcommand{\re}        {\mathrm{Re}}

\newcommand{\cws}   {\stackrel{*}{\to}}
\newcommand{\rhy}   {\textnormal{RHP}-$\mathcal{Y}$}
\newcommand{\rht}   {\textnormal{RHP}-$\mathcal{T}$}
\newcommand{\rhs}   {\textnormal{RHP}-$\mathcal{S}$}
\newcommand{\rhn}   {\textnormal{RHP}-$\mathcal{N}$}

\newcommand{\rhpa}   {\textnormal{RHP}-$\mathcal{P}_a$}
\newcommand{\rhpb}   {\textnormal{RHP}-$\mathcal{P}_b$}
\newcommand{\rhpsi} {\textnormal{RHP}-$\Psi$}
\newcommand{\rhups} {\textnormal{RHP}-$\Upsilon$}
\newcommand{\rhr}   {\textnormal{RHP}-$\mathcal{R}$}

\begin{document}

\title[Strong Asymptotics of Nuttall-Stahl Polynomials] {{P}ad\'e approximants for functions with branch points -- Strong Asymptotics of Nuttall-Stahl Polynomials}

\author[A.I. Aptekarev]{Alexander I. Aptekarev}

\address{Keldysh Institute of Applied Mathematics, Russian Academy of Science, Moscow, RF}

\email{aptekaa@keldysh.ru}

\author[M.L. Yattselev]{Maxim L. Yattselev}

\address{Corresponding author, Department of Mathematics, University of Oregon, Eugene, OR, 97403, USA}

\email{maximy@uoregon.edu}

\begin{abstract}
Let $f $ be a germ of an analytic function at infinity that can be analytically
continued along any path in the complex plane deprived of a finite set of
points, $f \in\mathcal{A}(\overline{\C} \setminus A), \quad \sharp A <
\infty $. J.~Nuttall has put forward the important relation between the
\textit{maximal domain}  of $f$ where the function has a
single-valued branch and the \emph{domain of convergence} of the diagonal Pad\'e
approximants for $f$. The Pad\'e approximants, which are rational functions and
thus single-valued, approximate a  holomorphic branch of $f$ in
the domain of their convergence. At the same time most of their poles tend to
the boundary of the domain of convergence and the support of their limiting
distribution models the system of cuts that makes the function $f$
single-valued. Nuttall has conjectured (and proved for many important special
cases) that this system of cuts has \emph{minimal logarithmic capacity} among
all other systems converting the function $f$ to a single-valued branch. Thus
the domain of convergence corresponds to the \textit{maximal} (in the sense of
\textit{minimal} boundary) domain of single-valued holomorphy for the analytic
function $f \in\mathcal{A}(\overline{\C} \setminus A)$. The complete
proof of Nuttall's conjecture (even in a more general setting where the set $A$
has logarithmic capacity $0$) was obtained by H.~Stahl. In this work, we derive
\textit{strong asymptotics} for the denominators of the diagonal Pad\'e approximants for
this problem in a rather general setting. We assume that $A$ is a finite set of
branch points of $f$ which have the \emph{algebro-logarithmic character} and
which are placed in a \emph{generic position}. The last restriction means that
we exclude from our consideration some degenerated ``constellations'' of the
branch points.

\end{abstract}



\maketitle

\section{Introduction}

Let $f$ be a function holomorphic at infinity. Then $f$ can be represented as a power series
\begin{equation}
\label{f}
f(z) = \sum_{k=0}^\infty \frac{f_k}{z^k}.
\end{equation}
A \emph{diagonal Pad\'e approximant} to $f$ is a rational function $[n/n]_f=p_n/q_n$ of type $(n,n)$ (i.e., $\deg(p_n),\deg(q_n)\leq n$) that has maximal order of contact with $f$ at infinity \cite{Pade92,BakerGravesMorris}. It is obtained from the solutions of the linear system
\begin{equation}
\label{linsys}
R_n(z):= q_n(z)f(z)-p_n(z) = \mathcal{O}\left(1/z^{n+1}\right) \quad \mbox{as} \quad z\to\infty
\end{equation}
whose coefficients are the moments $f_k$ in \eqref{f}. System \eqref{linsys} is always solvable and no solution of it can be such that $q_n\equiv0$ (we may thus assume that $q_n$ is monic). In general, a solution is not unique, but yields exactly the same rational function $[n/n]_f$. Thus, each solution of \eqref{linsys} is of the form $(lp_n,lq_n)$, where $(p_n,q_n)$ is the unique solution of minimal degree. Hereafter, $(p_n,q_n)$ will always stand for this unique pair of polynomials.

Pad\'e approximant $[n/n]_f$ as well as the index $n$ are called \emph{normal} if $\deg(q_n)=n$ \cite[Sec.~2.3]{NikishinSorokin}. The occurrence of non-normal indices is a consequence of overinterpolation. That is, if $n$ is normal index and\footnote{We say that $a(z)\sim b(z)$ if $0<\liminf_{z\to\infty}|a(z)/b(z)|\leq\limsup_{z\to\infty}|a(z)/b(z)|<\infty$.}
\[
f(z)-[n/n]_f(z) \sim z^{-(2n+l+1)} \quad \mbox{as} \quad z\to\infty
\]
for some $l\geq0$, then $[n/n]_f=[n+j/n+j]_f$ for $j\in\{0,\ldots,l\}$, and $n+l+1$ is normal.

Assume now that the germ \eqref{f} is analytically continuable along any path in $\C\setminus A$ for some fixed set $A$. Suppose further that this continuation is multi-valued in $\overline\C\setminus A$, i.e., $f$ has branch-type singularities at some points in $A$. For brevity, we denote this by
\begin{equation}
\label{analytE}
 f \in \mathcal{A}(\overline\C\setminus A)\,.
\end{equation}
The  theory of Pad\'e approximants to functions with branch points has been initiated by J.~Nuttall. In the  pioneering paper \cite{Nut77}
 he considered a class of functions \eqref{analytE}  with an even number of
 branch points (forming the set $A$) and principal singularities of the
square root type. Convergence in \emph{logarithmic capacity}
\cite{Ransford,SaffTotik} of Pad\'e approximants, i.e.,
\begin{equation}
\label{convCAP} \forall\,\,\varepsilon\,>\,0\,, \qquad \lim_{n \to \infty}
\cp\left(\{z\in K :~|f(z)-[n/n]_f(z)|>\varepsilon\}\right) \,=\,0\,,
\end{equation}
was proven uniformly on compact subsets of $\overline\C \setminus\Delta$, where $\Delta$ is a system of arcs which is completely determined by the location of the branch points. Nuttall characterized this
system of arcs as a system that has minimal logarithmic capacity among all
other systems of cuts making the function $f$  single-valued in their complement. That is,
\begin{equation}
\label{minCAP}
\cp ( \Delta ) = \min_{\partial D: D \in \mathcal{D}_f}
\cp ( \partial D ) \,,
\end{equation}
where we denoted  by $\mathcal{D}_f$ the collection of all connected domains
 containing the point at infinity in which $f$ is holomorphic and single-valued.

 In that paper he
has  conjectured  that \textit{for any function $f$ in $\mathcal{A}(\overline\C\setminus A)$ with any
finite number of branch points that are arbitrarily positioned in the complex plane, i.e.,}
\begin{equation}
\label{Efinite}  \sharp \, A < \infty \quad \mbox{and} \quad A\subset\C,
\end{equation}
\textit{and with an arbitrary type of branching singularities at those points, the diagonal Pad\'e approximants converge to $f$ in logarithmic capacity away from the system of cuts $\Delta$ characterized by the property of minimal logarithmic capacity.}

Thus, Nuttall in his conjecture has put forward the important relation between
the \textit{maximal domain}   where the multi-valued function $f$ has
single-valued branch and the \emph{domain of convergence} of the diagonal Pad\'e
approximants to $f$ constructed solely based on the series representation
\eqref{f}. The Pad\'e approximants, which are rational functions and thus
single-valued, approximate a single-valued holomorphic branch of $f$ in the
domain of their convergence. At the same time most of their poles tend to the
boundary of the domain of convergence and the support of their limiting
distribution models the system of cuts that makes the function $f$
single-valued (see also \cite{Nut80}).

The complete proof of Nuttall's conjecture (even in a more general setting) was taken up by H.~Stahl. In a series of  fundamental papers \cite{St85,St85b,St86,St89,St97} for a multi-valued function $f \in \mathcal{A}(\overline\C\setminus A)$ with $\cp( A ) \,=\,0$ (no more restrictions!) he proved: the \emph{existence} of a domain $D^*\in \mathcal{D}_f$ such that the boundary $\Delta =\partial D^*$satisfies \eqref{minCAP};  \textit{weak ($n$-th root) asymptotics} for the denominators of the Pad\'e approximants \eqref{linsys}
\begin{equation}
\label{nthAs}
\lim_{n \to \infty} \frac{1}{n}\log |q_n(z)| = - V^{\omega_\Delta}(z),
\qquad z \in D^*,
\end{equation}
where $V^{\omega_\Delta}:=-\int \log |z-t|\, d\omega_\Delta(t)$ is the logarithmic potential of the equilibrium measure $\omega_\Delta$, minimizing the energy functional $I(\mu):=\int V^{\mu}(z)\, d\mu(z)$ among all probability measures $\mu$ on $\Delta$, i.e., $I(\omega_\Delta):= \min_{\mu(\Delta)=1}I(\mu)$; \emph{convergence theorem} \eqref{convCAP}.

The aim of the present paper is to established the strong  (or Szeg\H{o} type, see   \cite[Ch.~XII]{Szego}) asymptotics of \textit{the Nuttall-Stahl polynomials} $q_n$. In other words, to identify the limit
\[
\lim_{n \to \infty}\, \frac{q_n}{\map^n} \,= \,? \quad \mbox{in} \quad D^*,
\]
where the polynomials $q_n$ are the denominators of the diagonal Pad\'e approximants \eqref{linsys} to functions \eqref{analytE} satisfying \eqref{Efinite} and $\map$ is a properly chosen normalizing function. 

Interest in the strong asymptotics comes, for example, from the problem of
uniform convergence of the diagonal Pad\'e approximants. Indeed, the weak type of convergence such as the
convergence in capacity in Nuttall's conjecture and Stahl theorem is not a mere technical shortcoming. Indeed, even though most of the poles (full measure) of the approximants approach the system of the extremal cuts $\Delta$, a small number of them (measure zero) may cluster away from $\Delta$ and impede the uniform convergence. Such poles are
called \textit{spurious or wandering}. Clearly, controlling these
poles is the key for understanding the uniform convergence.

There are many  special cases of the Nuttall-Stahl polynomials that have been studied in detail including their strong asymptotics. Perhaps the most famous examples are the Pad\'e approximants to functions $1/\sqrt{z^2-1}$ and $\sqrt{z^2-1}-z$ (the simplest meromorphic functions on a two sheeted Riemann surface of genus zero) where the Nuttall-Stahl polynomials $q_n$ turn out to be the classical Chebysh\"ev polynomials of the first and second kind, respectively. The study of the diagonal Pad\'e approximants for functions meromorphic on certain Riemann surfaces of genus one by means of elliptic functions was initiated in the works of S.~Duma \cite{uDum} and N.I.~Akhiezer \cite{Akhiezer}, see also \cite{Nik76} by E.M.~Nikishin. Supporting his conjecture, Nuttall considered two important classes of functions with branch points for which he obtained strong asymptotics of the diagonal Pad\'e approximants. In a joint paper with S.R.~Singh \cite{NutS80}, a generalization of the class of functions considered in \cite{Nut77} (even number of quadratic type branch points) was studied. Peculiarity of this class as well as its prototype from
\cite{Nut77} is that $\Delta$, the system  of extremal cuts \eqref{minCAP},
consists of non-intersecting analytic arcs, see Figure~\ref{figure01}A.
\begin{figure}[!ht]
\centering
\subfloat[]{\includegraphics[scale=.5]{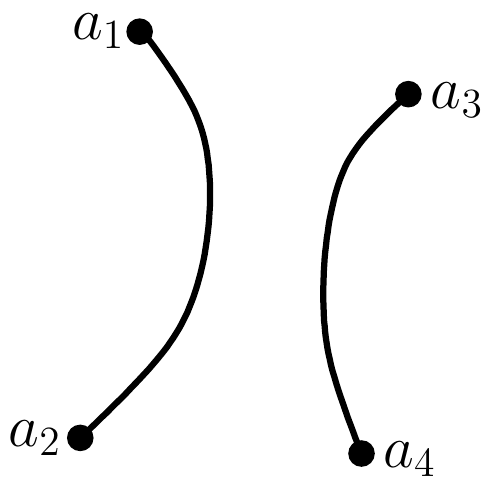}}\quad\quad\quad
\subfloat[]{\includegraphics[scale=.55]{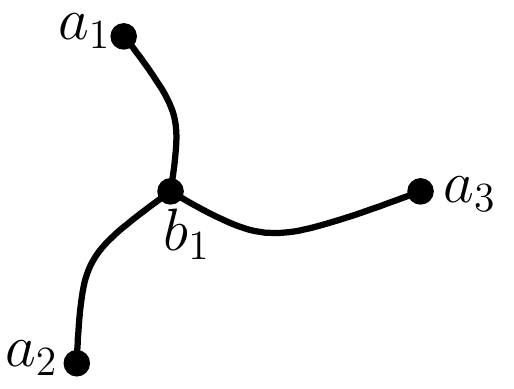}}
\caption{\small The left-hand figure depicts the case where $\Delta$ is comprised of non-intersecting analytic arcs, and the right-hand figure illustrates the case where three arcs share a common endpoint.}
\label{figure01}
\end{figure}

\noindent
In the paper \cite{Nut86} (see also \cite{Nut82}) Nuttall investigated the behavior of the Pad\'e approximants for functions
with three non-collinear branch points. Namely, functions of the form
\begin{equation}
\label{3p} f(z):= \prod_{j=1}^3 (z-e_j)^{\alpha_j}\,,
\qquad \alpha_j \in \mathbb{C}\,:\quad \sum_{j=1}^3
\alpha_j = 0\,.
\end{equation}
Analytic arcs of the system of extremal cuts for these functions (contrary to the functions from the previous
class) share a common endpoint, see Figure~\ref{figure01}B. In order to shed some light on the behavior of the spurious poles, Stahl studied strong asymptotics of the diagonal Pad\'e approximants for hyperelliptic functions \cite{St96}, see also \cite{St97}.

An important feature of the diagonal Pad\'e approximants, which plays a key
role in the study of their asymptotics, is the orthogonality of their denominators. It is quite simple to see that
\eqref{linsys} and the Cauchy theorem (taking into account the definition of
$\mathcal{D}_f$ in \eqref{minCAP}) lead to
\[
\int\limits_{\partial D: D \in \mathcal{D}_f} q_n(z)\,z^j\,f(z)\,dz\,=\,0\,,
\qquad j\in\{0,\ldots, n-1\},
\]
where the integral is taken along the orientated boundary $\partial D$ of a domain
$D$ from $\mathcal{D}_f$. For the extremal domain $D^*\in\mathcal{D}_f$ of $f$ satisfying \eqref{analytE} and
\eqref{Efinite},  the boundary $\Delta=\partial D^*$ consists of a finite union of
analytic Jordan arcs. Hence, choosing an orientation of $\Delta$
(as a set of Jordan arcs), we can introduce in general complex-valued weight function
\begin{equation}
\label{cw}
\rho(t) =  (f^+-f^-)(t), \qquad t\in \Delta,
\end{equation}
which turns $q_n$ into \textit{non-Hermitian orthogonal polynomials}. That is,
\begin{equation}
\label{ortho}
\int\limits_{\Delta} q_n(t)\,t^j\,\rho(t)\,dt\,=\,0,
\qquad j\in\{0,\ldots, n-1\}.
\end{equation}
Asymptotic analysis of the non-Hermitian orthogonal polynomials is a difficult problem substantially different from the study of the asymptotics of the polynomials orthogonal with respect to a \textit{Hermitian inner product}, i.e., the case where $\rho$ is real-valued and $\Delta \subset\R$.

In \cite{St86}, Stahl  developed a new method of study of the weak ($n$-th root)
asymptotics \eqref{nthAs} of the polynomials orthogonal with respect to
\textit{complex-valued weights}. As discussed above, this resulted in the proof of
Nuttall's conjecture. The method of Stahl was later extended by A.A.~Gonchar and E.A.~Rakhmanov
in \cite{GRakh87} to include the weak ($n$-th root)
asymptotics of the polynomials orthogonal with respect to \textit{varying
complex-valued weights}, i.e., to include the case where the weight function $\rho := \rho_n$ in \eqref{ortho} depends
on $n$, the degree of the polynomial $q_n$. Orthogonal polynomials with varying weights
play an important role in analysis of multipoint Pad\'e approximants, best
(Chebysh\"ev) rational approximants, see \cite{GL78, Gon78}, and in many
other applications (for example in description of the eigenvalue distribution of random
matrices \cite{Deift}).

The methods of obtaining strong asymptotics of the polynomials orthogonal
 with respect to a complex weight are based on a certain boundary-value
problem for analytic functions (Riemann--Hilbert problem). Namely,
\begin{equation}
\label{BC}
R_n^+-R_n^-= q_n\rho \quad \text{on}\quad \Delta,
\end{equation}
where $R_n$, defined in \eqref{linsys}, are the  \textit{reminder functions}  for Pad\'e approximants (or \textit{functions of the second kind} for polynomials \eqref{ortho}), which also can be expressed as
\begin{equation}
\label{secondkind}
R_n(z)=\int_\Delta\frac{q_n(t)\rho(t)}{t-z}\,\frac{dt}{2\pi i}, \qquad
R_n(z)=\mathcal{O}\biggl(\frac1{z^{n+1}}\biggr) \quad \text{as}\quad
n\to\infty.
\end{equation}
The boundary-value problem \eqref{BC} naturally follows from \eqref{linsys} and the Sokhotski\u\i--Plemelj formulae.
This approach appeared in the works of Nuttall in connection with the
study of the strong asymptotics of the Hermite--Pad\'e polynomials, see the review~\cite{Nut84}. In~\cite{Nut90} Nuttall transformed the boundary
condition~\eqref{BC} into a singular integral equation
and on this basis obtained the formulae of strong asymptotics for
polynomials~\eqref{ortho} orthogonal on the interval
$\Delta:= [-1,1] $ with respect to a holomorphic complex-valued weight
\[
\rho(x):=\frac{\widetilde \rho(x)}{\sqrt{1-x^2}} ,
\qquad
\widetilde \rho\in H(\Delta), \quad
\widetilde \rho\ne0 \enskip \text{on}\quad \Delta:= [-1,1],
\]
where $H(\Delta)$ is a class of functions holomorphic in some neighborhood of $\Delta$. Here~$\widetilde \rho$ can also be
a complex-valued non-vanishing Dini-continuous function on $[-1,1]$ \cite{BY09c}. The most general known extension of this class of  orthogonal polynomials is due to S.P.~Suetin \cite{Suet00,Suet03} who considered the convergence domain $D^*\in \mathcal{D}_f$ for the function $1/\sqrt{(t-e_1)\cdots(t-e_{2g+2})}$, $e_j\in\C$, when the boundary $\Delta = \partial D^*$ consists of $g+1$ disjoint Jordan arcs (like in \cite{NutS80}, see Figure~\ref{figure01}A). Elaborating on the singular integral method of  Nuttall, he derived strong asymptotics for polynomials~\eqref{ortho} orthogonal on $\Delta$ with respect to the complex
weight
\[
\rho(x):=\frac{\widetilde \rho(x)}{\sqrt{(t-e_1)\cdots(t-e_{2g+2})}} ,
\]
where $\widetilde\rho$ is a H\"older continuous and non-vanishing function on $\Delta$. In \cite{uBY2}, L.~Baratchart and the second author have studied strong asymptotics for polynomials \eqref{ortho} via the singular integral method in the \emph{elliptic} case $g=1$, but under the assumptions that $\rho$ is Dini-continuous and non-vanishing on $\Delta$, while the latter is connected and consists of three arcs that meet at one of the branch points (exactly the same set up as in \cite{Nut86}, see Figure~\ref{figure01}B). The strong asymptotics of Nuttall-Stahl polynomials arising from the function~\eqref{3p} was derived in the recent work \cite{M-FRakhSuet12} in three different ways, including singular integral equation method of Nuttall and the matrix Riemann-Hilbert method.

The latter approach facilitated substantial progress in proving new results for the strong asymptotics of orthogonal polynomials and is based on a matrix-valued Riemann-Hilbert boundary value problem. The core of the method lies in formulating  a Riemann-Hilbert problem for $2\times2$ matrices (due to Fokas, Its, and Kitaev \cite{FIK91,FIK92}) whose entries are orthogonal polynomials (\ref{ortho}) and functions of the second kind (\ref{secondkind}) to which the steepest descent analysis (due to Deift and Zhou \cite{DZ93}) is applied as $n\to \infty$. This method was initially designed to study the asymptotics of the integrable PDEs and was later applied to prove asymptotic results for polynomials orthogonal on the real axis with respect to real-valued analytic weights, including varying weights (depending on $n$) \cite{DKMLVZ99a,DKMLVZ99b,KMcL99,KMcLVAV04} and related questions from random matrix theory. It also has been noticed \cite{BaikDMLMilZ01,Ap02,KamvissisMcLaughlinMiller,BY10} (see also recent paper \cite{BerMo09}) that the method works for the non-Hermitian orthogonality in the complex plane with respect to complex-valued weights.

In the present paper we apply the matrix Riemann-Hilbert method to obtain
strong asymptotics of Pad\'e approximants for functions with branch points (i.e., we obtain strong asymptotics of Nuttall-Stahl polynomials). To capture the geometry of multi-connected domains we use the Riemann theta functions as it was done in \cite{DKMLVZ99a}, but keep our presentation in the spirit of \cite{Ap08,ApLy10}.

This paper is structured as follows. In the next section we introduce necessary notation and state our main result. In Sections~\ref{s:ED} and \ref{S:RS} we describe in greater detail the geometry of the problem. Namely, Section~\ref{s:ED} is devoted to the existence and properties of the extremal domain $D^*$ for the functions of the form \eqref{analytE} and \eqref{Efinite}. Here, for
completeness of the presentation, we present some results and their proofs from the unpublished manuscript \cite{uPerevRakh}. Section~\ref{S:RS} is designed to highlight main properties of the Riemann surface of the derivative of the complex Green function
of the extremal domain $D^*$. Sections~\ref{s:bvp} and \ref{s:s} are devoted to a solution of a certain boundary value problem on $\RS$ and are auxiliary to our main results. In the last three sections we carry out the matrix Riemann-Hilbert analysis. In Section~\ref{s:5} we state the corresponding matrix Riemann-Hilbert problem, renormalize it, and perform some identical transformations that simplify the forthcoming analysis. In Section~\ref{s:6} we deduce the asymptotic  (as $n \rightarrow \infty$) solution of the initial Riemann-Hilbert problem and finally in Section~\ref{s:7} we derive the strong asymptotics
 of Nuttall-Stahl polynomials.

\bigskip

\section{Main Results}
\label{sec:main}

The main objective of this work is to describe the asymptotics of the diagonal Pad\'e approximants to algebraic functions. We restrict our attention to those functions that have finitely many branch points, all of which are of an integrable order, with no poles and whose contour of minimal capacity satisfies some generic conditions, Section~\ref{ss:AF}. Such functions can be written as Cauchy integrals of their jumps across the corresponding minimal capacity contours and therefore we enlarge the considered class of functions to Cauchy integrals of densities that behave like the non-vanishing jumps of algebraic functions, Section~\ref{ss:OP}. It turns out that the asymptotics of Pad\'e approximants is described by solutions of a specific boundary value problem on a Riemann surface corresponding to the minimal capacity contour. This surface and its connection to the contour are described in Section~\ref{ss:RS}, while the boundary value problem as well as its solution are stated in Section~\ref{ss:auxBVP}. The main results of this paper are presented in Section~\ref{ss:MTh}.

\subsection{Functions with Branch Points}
\label{ss:AF}

Let $f$ be a function holomorphic at infinity that extends analytically, but in a multi-valued fashion, along any path in the extended complex plan that omits finite number of points. That is,
\begin{equation}
\label{pointsa}
 f \in \mathcal{A}(\overline\C\setminus A)\,,\qquad A\,:=\, \{a_k\},
 \quad 2\leq \sharp A < \infty.
\end{equation}
Without loss of generality we may assume that  $f(\infty)=0$ since subtracting a constant from $f$ changes the Pad\'e approximant in a trivial manner. 

We impose two general restrictions on the functions \eqref{pointsa}. The first restriction is related to the character of singularities at the branch points. Namely, we assume that the branch points are \emph{algebro-logarithmic}. It means that in a small enough neighborhood of each $a_k$ the function $f$ has a representation
\begin{equation}
\label{AlgLog}
 f(z)\,=\, h_{1}(z)\psi(z) + h_{2}(z), \quad \psi(z) = \left\{\begin{array}{ll}(z-a_k)^{\alpha(a_k)}\\\log(z-a_k)\end{array}\right.,
  \end{equation}
where $-1<\alpha(a_k)<0$ and $h_1,h_2$ are holomorphic around $a_k$. The second restriction is related to the disposition of the branch points. Denote by $D^*$ the extremal domain for $f$ in the sense of Stahl. It is known, Proposition~\ref{prop:structure}, that
\begin{equation}
\label{Delta}
\Delta = \overline\C\setminus D^* = E\cup\bigcup \Delta_k,
\end{equation}
where $\bigcup \Delta_k$ is a finite union of open analytic Jordan arcs and $E$ is a finite set of points such that each element of $E$ is an endpoint for at least one arc $\Delta_k$, Figure~\ref{fig:0}.
\begin{figure}[!ht]
\centering
\includegraphics[scale=.6]{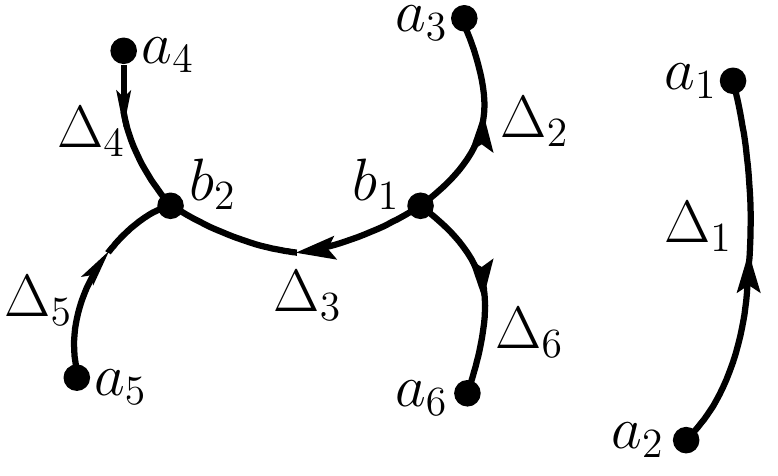}
\caption{\small Schematic example of $\Delta$, depicting the set $E=\{a_1,\ldots,a_6\}\cup\{b_1,b_2\}$, the arcs $\Delta_k$, and their orientation.}
\label{fig:0}
\end{figure}
In what follows, we suppose that the points forming $A$ are in a \emph{Generic Position} (GP).
\begin{cond1}
\label{cond1}
We assume that
\begin{itemize}
\item[(i)] each point in $E\bigcap A$ is incident with exactly one arc from $\bigcup\Delta_k$;
\item[(ii)] each point in $E\setminus A$ is incident with exactly three arcs from $\bigcup\Delta_k$.
\end{itemize}
\end{cond1}
The above condition describes a generic case for  the set $A$. Meaning that if the set $A$ does not satisfy this condition, then there is a small perturbation of the position of these points such that new set $A$ obeys Condition~\hyperref[cond1]{GP}.

Denote by $g_\Delta$ the \emph{Green function} for $D^*$ with a pole at infinity, Section~\ref{s:ED}. That is, $g_\Delta$ is the unique function harmonic in $D^*\setminus\{\infty\}$ having zero boundary values on $\Delta$ that diverges to infinity like $\log|z|$ as $|z|\to\infty$. It is known \cite[Thm.~5.2.1]{Ransford} that the logarithm of the logarithmic capacity of $\Delta$ is equal to
\[
\log\cp(\Delta) = \lim_{z\to\infty}\big(\log|z|-g_\Delta(z)\big).
\]
As shown in \cite{uPerevRakh}, see also \eqref{greendelta} further below, it holds that
\begin{equation}
\label{functionh}
h(z) := (2\partial_zg_\Delta)(z) = \frac1z + \cdots = \sqrt\frac{B(z)}{A(z)},
\end{equation}
where $2\partial_z:=\partial_x-i\partial_y$, $B$ is a monic polynomial of degree $\deg(A)-2$, and
\[
A(z) := \prod_{k=1}^p(z-a_k), \quad \{a_1, \ldots, a_p\}:= A \cap E.
\]
Since $g_\Delta\equiv0$ on $\Delta$, so is its tangential derivative at each smooth point of $\Delta^+\cup\Delta^-$. Hence, $h(t)\tau_t$, $t\in\Delta^\pm$,  is purely imaginary, where $\tau_t$ is the complex number corresponding to the tangent vector at $t$ to $\Delta$. In particular, the integral of $h(t)dt$ along $\Delta^\pm$ is purely imaginary.

Condition~\hyperref[cond1]{GP} has the following implications on $B$ \cite[Section~8]{Pommerenke2}. Let $m$ be the number of the connected components of $\Delta$. Then $B$ has $p-2m$ simple zeros that we denote by $b_1,\ldots,b_{p-2m}$ and all the other zeros are of even multiplicities. In particular,
\begin{equation}
\label{setE}
E:=\{a_1,\ldots,a_p\}\cup\{b_1,\ldots,b_{p-2m}\}.
\end{equation}
If we set $g:=p-m-1$, then $|E|=2g+2$. Moreover, we can write
\begin{equation}
\label{pointsb}
B(z) = \prod_{j=1}^{p-2m}(z-b_j)\prod_{j=p-2m+1}^g(z-b_j)^2
\end{equation}
where the elements of $\{b_{p-2m+1},\ldots,b_g\}$ are the zeros of $B$ of even multiplicities $m_{b_j}$, each listed $m_{b_j}/2$ times.

\subsection{Cauchy-type Integrals}
\label{ss:OP}

Let $f$ be a function of the form \eqref{pointsa}--\eqref{AlgLog} with contour $\Delta$ in \eqref{Delta} satisfying Condition~\hyperref[cond1]{GP}.  We orient the arcs $\Delta_k$ comprising $\Delta$ so that the arcs sharing a common endpoint either are all oriented
towards this endpoint or away from it, see Figure~\ref{fig:0}. As the complement of $\Delta$ is connected, i.e., $\Delta$ forms no loop, such an orientation is always feasible and, in fact, there are only two such choices which are inverse to each other. According to the chosen orientation we distinguish the left ($+$) and the right ($-$) sides of each arc. Then
\begin{equation}
\label{fInt}
f(z) = \int_\Delta\frac{(f^+-f^-)(t)}{t-z}\frac{dt}{2\pi i}, \quad z\in D^*,
\end{equation}
where the integration on $\Delta$ is taking place according to the chosen orientation.

Let $e\in E\setminus A$. Then $e$ is incident with exactly three arcs, which we denote for convenience by $\Delta_{e,j}$, $j\in\{1,2,3\}$. Since $e$ is not a point of branching for $f$, the jumps $(f^+-f^-)_{|\Delta_{e,k}}$ are holomorphic around $e$ and enjoy the property
\begin{equation}
\label{fSum}
(f^+-f^-)_{|\Delta_{e,1}}+(f^+-f^-)_{|\Delta_{e,2}}+(f^+-f^-)_{|\Delta_{e,3}}\equiv 0,
\end{equation}
where we also used the fact that the arcs $\Delta_{e,j}$ have similar orientation as viewed from $e$.

Set $\alpha(e):=0$ for each $e\in E\setminus A$ and fix $a,b\in E$ that are adjacent to each other by an arc $\Delta_k\subset\Delta$. Then the jump of $f$ across $\Delta_k$ can be written as
\begin{equation}
\label{fJump}
(f^+-f^-)(z) = w_k(f;z)(z-a)^{\alpha(a)}(z-b)^{\alpha(b)}, \quad z\in \Delta_k,
\end{equation}
where  we fix branches of $(z-a)^{\alpha(a)}$ and $(z-b)^{\alpha(b)}$ that are holomorphic across~$\Delta_k$ and $w_k(f;\cdot)$ is a holomorphic and non-vanishing function in some neighborhood of $\Delta_k$.

Keeping in mind \eqref{fSum} and \eqref{fJump}, we introduce the following class of weights on $\Delta$. 
\begin{definition}
A weight function $\rho$ on $\Delta$ belongs to the class $\W$ if 
\begin{equation}
\label{rhoSum}
\rho_{|\Delta_{e,1}}+\rho_{|\Delta_{e,2}}+\rho_{|\Delta_{e,3}} \equiv 0
\end{equation}
in a neighborhood of each $e\in E\setminus A$, where $\Delta_{e,j}$, $j\in\{1,2,3\}$, are the arcs incident with~$e$; and
\begin{equation}
\label{rho}
\rho_{|\Delta_k}(z) = w_k(z)(z-a)^{\alpha_a}(z-b)^{\alpha_b}
\end{equation}
on $\Delta_k$, incident with $a,b$, where $w_k$ is holomorphic and non-vanishing in some neighborhood of $\Delta_k$ and $\{(z-e)^{\alpha_e}\}_{e\in E}$ is a collection of functions holomorphic in some neighborhood of  $\Delta\setminus\{e\}$, $\alpha_e>-1$ and $\alpha_e=0$ for $e\in E\setminus A$.
\end{definition}

For a weight $\rho\in\W$, we set
\begin{equation}
\label{hatrho}
\widehat\rho(z) := \int_\Delta\frac{\rho(t)}{t-z}\frac{dt}{2\pi i}, \quad z\in D^*.
\end{equation}
It follows from \eqref{fInt}--\eqref{fJump} that a function $f$ satisfying \eqref{pointsa}--\eqref{AlgLog} and Condition~\hyperref[cond1]{GP}, and whose jump $f^+-f^-$ is non-vanishing on $\Delta\setminus\{a_1,\ldots,a_p\}$ can be written in the form \eqref{hatrho} for a $\W$-weight.

\subsection{Riemann Surface}
\label{ss:RS}

Denote by $\RS$ the Riemann surface of $h$ defined in \eqref{functionh}. We represent $\RS$ as a two-sheeted ramified cover of $\overline\C$ constructed in the following manner. Two copies of $\overline\C$ are cut along every arc $\Delta_k$. These copies are joined at each point of $E$ and along the cuts in such a manner that the right (resp. left) side of the arc $\Delta_k$ belonging to the first copy, say $\RS^{(0)}$, is joined with the left (resp. right) side of the same arc $\Delta_k$ only belonging to the second copy, $\RS^{(1)}$. It can be readily verified that $\RS$ is a hyperelliptic Riemann surface of genus $g$.

According to our construction, each arc $\Delta_k$ together with its endpoints corresponds to a cycle, say $L_k$, on $\RS$. We set $L:=\bigcup_kL_k$, denote by $\pi$ the canonical projection $\pi:\RS\to\overline\C$, and define
\[
D^{(k)}:=\RS^{(k)}\cap \pi^{-1}(D^*) \quad \mbox{and} \quad z^{(k)}:=D^{(k)}\cap\pi^{-1}(z)
\]
for $k\in\{0,1\}$ and $z\in D^*$. We orient each $L_k$ in such a manner that $D^{(0)}$ remains on the left when the cycle is traversed in the positive direction. For future use, we set $\{b_j^{(1)}\}_{j=1}^g$ to be such point on $\RS$ that
\begin{equation}
\label{bj1}
\pi\left(b_j^{(1)}\right)=b_j \quad \mbox{and} \quad b_j^{(1)}\in D^{(1)}\cup L , \quad j\in\{1,\ldots,g\}.
\end{equation}

For any (sectionally) meromorphic function $r$ on $\RS$ we keep denoting by $r$ the pull-back function from $\RS^{(0)}$ onto $\overline\C$ and we denote by $r^*$ the pull-back function from $\RS^{(1)}$ onto $\overline\C$. We also consider any function on $\overline\C$ naturally defined on $\RS^{(0)}$. In particular, $h$ is a rational function over $\RS$ such that $h^*=-h$ (as usual, a function is rational over $\RS$ if the only singularities of this function on $\RS$ are polar).

Denote by $\left\{\mathbf{a}_k\right\}_{k=1}^g$ and $\left\{\mathbf{b}_k\right\}_{k=1}^g$ the following homology basis for $\RS$. Let $C_j(\Delta)$, $j\in\{1,\ldots,m\}$, be the connected components of $\Delta$. Set $p_j:=|\{a_1,\ldots,a_p\}\cap C_j(\Delta)|$. Clearly, $p=\sum_{j=1}^mp_j$. Relabel, if necessary, the points $\{a_1,\ldots,a_p\}$ in such a manner that $\{a_{p_0+\cdots+p_{j-1}+1},\ldots,a_{p_0+\cdots+p_j}\}\subset C_j(\Delta)$, where $p_0:=0$. Then for each $j\in\{2,\ldots,m\}$ we choose $p_j-1$ $\mathbf{b}$-cycles as those cycles $L_k$ that contain the points $a_{p_0+\cdots+p_{j-1}+2},\ldots,a_{p_0+\cdots+p_j}$, and we choose $p_1-2$ $\mathbf{b}$-cycle as those cycles $L_k$ that contain $a_3,\ldots,a_{p_1}$.
\begin{figure}[!ht]
\centering
\includegraphics[scale=.6]{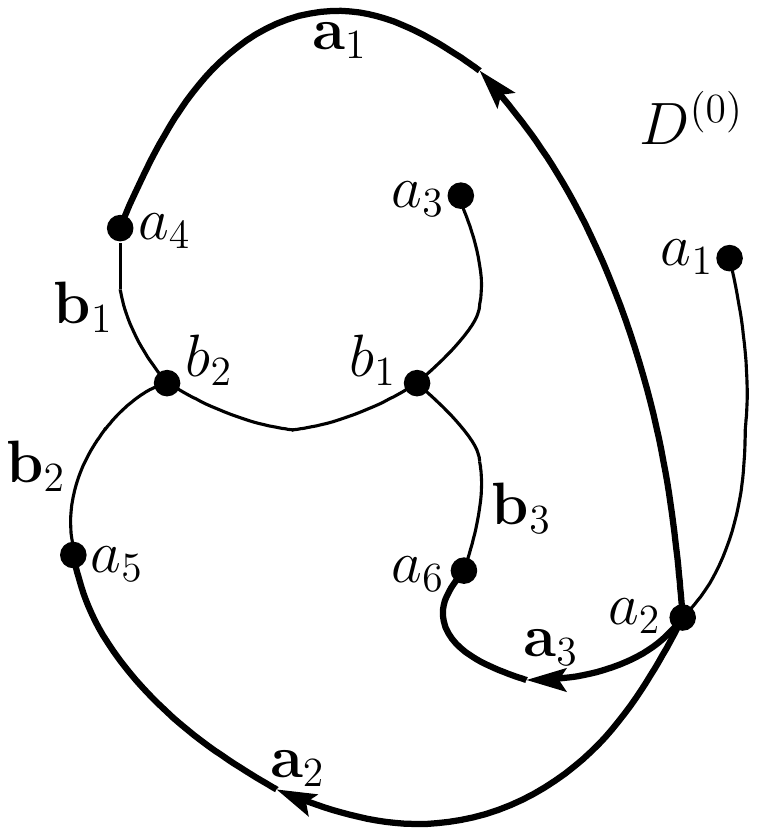}
\caption{\small The choice of the $\mathbf{b}$-cycles and the parts of the $\mathbf{a}$-cycles belonging to $D^{(0)}$ (thicker lines necessarily oriented towards the corresponding $\mathbf{b}$-cycles).}
\label{fig:1}
\end{figure}
We assume that the orientation of the $\mathbf{b}$-cycles is induced by the orientation of the corresponding cycles $L_k$. The $\mathbf{a}$-cycles are chosen to be mutually disjoint except at $a_2$, which belongs to all of them. It is assumed that each cycle $\mathbf{a}_k$ intersect the corresponding cycle $\mathbf{b}_k$ only at one point, the one that belongs to $\{a_3,\ldots,a_p\}$, and that
\[
\Delta_k^a:=\pi(\mathbf{a}_k\cap D^{(0)})=\pi(\mathbf{a}_k\cap D^{(1)}).
\]
The $\mathbf{a}$-cycles are orientated in such a manner that the tangent vectors to $\mathbf{a}_k,\mathbf{b}_k$ form the right pair at the point of their intersection. We also assume that each arc $\Delta_k^a$ naturally inherits the orientation of $\mathbf{a}_k\cap D^{(0)}$. In particular, the $+$ side of $\mathbf{a}_k\cap D^{(0)}$ and the $-$ side of $\mathbf{a}_k\cap D^{(1)}$ project onto the $+$ side of $\Delta_k^a$, see Figure~\ref{fig:1}. We set
\[
\widetilde\RS:=\RS\setminus\bigcup_{k=1}^g(\mathbf{a}_k\cup\mathbf{b}_k) \quad \mbox{and} \quad \widehat\RS:=\RS\setminus\bigcup_{k=1}^g\mathbf{a}_k
\]
(observe that $\widetilde\RS$ is a simply connected subdomain of $\RS$).

Define (see also Section~\ref{ss:gd})
\begin{equation}
\label{map}
\map(\z) := \exp\left\{\int_{a_1}^\z h(t)dt\right\} \quad \mbox{for} \quad \z\in\widetilde\RS.
\end{equation}
Then $\map$ is a holomorphic and non-vanishing function on $\widetilde\RS$ except for a simple pole at $\infty^{(0)}$ and a simple zero at $\infty^{(1)}$ whose pull-back functions are reciprocals of each other, i.e.,
\begin{equation}
\label{Da}
\map\map^* \equiv 1 \quad \mbox{in} \quad D_a := D^*\setminus\bigcup_{k=1}^g \Delta_k^a.
\end{equation}
Furthermore, $\map$ possesses continuous traces on both sides of each $\mathbf{a}$- and $\mathbf{b}$-cycle that satisfy
\begin{equation}
\label{jumpPhiperoids}
\frac{\map^+}{\map^-} = \left\{
\begin{array}{ll}
\exp\big\{2\pi i\omega_k\big\} & \mbox{on} \quad \mathbf{a}_k, \smallskip \\
\exp\big\{2\pi i\tau_k\big\} & \mbox{on} \quad \mathbf{b}_k,
\end{array}
\right.
\end{equation}
where the constants $\omega_k$ and $\tau_k$ are real and can be expressed as
\begin{equation}
\label{GreenPeriods}
\omega_k :=  -\frac{1}{2\pi i}\oint_{\mathbf{b}_k}h(t)dt \quad \mbox{and} \quad \tau_k:=\frac{1}{2\pi i}\oint_{\mathbf{a}_k}h(t)dt,
\end{equation}
$k\in\{1,\ldots,g\}$. In fact, it holds that $\omega_k=\omega_\Delta(\pi(L_k))$, where $\omega_\Delta$ is the \emph{equilibrium measure} of $\Delta$ \cite{Ransford}. Moreover, it is true that
\begin{equation}
\label{a11}
\map(z) = \frac{z}{\xi\cp(\Delta)} + \mathcal{O}(1) \quad \mbox{as} \quad z\to\infty, \quad |\xi|=1.
\end{equation}

In what follows, we shall assume without loss of generality that $\xi=1$. Indeed, if $\xi\neq1$, set $\Delta_\xi:=\{\bar\xi z:z\in\Delta\}$ and $\rho_\xi(z):=\rho(\xi z)$, $z\in\Delta_\xi$, where $\rho$ is a function defined in \eqref{rho}. Then
\[
\widehat\rho_\xi(z):=\widehat\rho(\xi z) \quad \mbox{and} \quad [n/n]_{\widehat\rho_\xi}(z)=[n/n]_{\widehat\rho}(\xi z), \quad z\in \overline\C\setminus\Delta_\xi.
\]
Thus, the asymptotic behavior of $[n/n]_{\widehat\rho}$ is entirely determined by the asymptotic behavior of $[n/n]_{\widehat\rho_\xi}$. Moreover, it holds that $\map_{\Delta_\xi}(z)=\map_\Delta(\xi z)$ and therefore $\map_{\Delta_\xi}(z)=z/\cp(\Delta)+\mathcal{O}(1)$. That is, we always can rotate the initial set up of the problem so that \eqref{a11} holds with $\xi=1$ without altering the asymptotic behavior.

Recall that a Riemann surface of genus $g$ has exactly $g$ linearly independent holomorphic differentials (see Section~\ref{ss:intro}).  We denote by
\[
d\vec\Omega:=\left(d\Omega_1,\ldots,d\Omega_g\right)^T
\]
the column vector of $g$ linearly independent holomorphic differentials normalized so that
\begin{equation}
\label{periodsa}
\oint_{\mathbf{a}_k}d\vec\Omega = \vec e_k \quad \mbox{for each} \quad k\in\{1,\ldots,g\},
\end{equation}
where $\left\{\vec e_k\right\}_{k=1}^g$ is the standard basis for $\R^g$ and $\vec e^T$ is the transpose of $\vec e$. Further, we set
\begin{equation}
\label{periodsb}
\mathcal{B}_\Omega := \left[ \oint_{\mathbf{b}_j}d\Omega_k\right]_{j,k=1}^g.
\end{equation}
It is known that $\mathcal{B}_\Omega$ is symmetric and has positive definite imaginary part.

\subsection{Auxiliary Boundary Value Problem}
\label{ss:auxBVP}

Let $\rho\in\W$. Define
\begin{equation}
\label{vectors}
\left\{
\begin{array}{lll}
\vec\omega &:=& \big(\omega_1,\ldots,\omega_g\big)^T, \smallskip \\
\vec\tau &:=& \big(\tau_1,\ldots,\tau_g\big)^T, \smallskip \\
\vec c_\rho &:=& \frac{1}{2\pi i}\oint_L\log(\rho/h^+) d\vec\Omega
\end{array}
\right.
\end{equation}
for some fixed determination of $\log(\rho/h^+)$ continuous on $\Delta\setminus E$, where the constants $\omega_j$ and $\tau_j$ were defined in \eqref{GreenPeriods} and we understand that $\log(\rho/h^+)$ on $L$ is the lift $\log(\rho/h^+)\circ\pi$. 

Further, let $\{\tr_j\}$ be an arbitrary finite collection of points on $\RS$. An integral divisor corresponding to this collection is defined as a formal symbol $\sum\tr_j$. We call a divisor $\sum\tr_j$ \emph{special} if it contains at least one pair of involution-symmetric points; that is, if there exist $\tr_j\neq\tr_k$ such that $\pi(\tr_j)=\pi(\tr_k)$ or multiple copies of points from $E$ (with a slight abuse of notation, we keep using $E$ for $\pi^{-1}(E)$).

Given constants \eqref{vectors} and points \eqref{bj1}, there exist divisors $\sum_{j=1}^g\tr_{n,j}$, see Sections~\ref{ss:jip} and~\ref{ss:ni} further below, such that
\begin{equation}
\label{main-jip}
\sum_{j=1}^g\int_{b_j^{(1)}}^{\tr_{n,j}}d\vec\Omega \, \equiv \, \vec c_\rho + n\big(\vec\omega+\mathcal{B}_\Omega\vec\tau\big) \quad \left(\mdp d\vec\Omega\right),
\end{equation}
where the path of integration belongs $\widetilde\RS$ (for definiteness, we shall consider each endpoint of integration belonging to the boundary of $\widetilde\RS$ as a point on the positive side of the corresponding $\mathbf{a}$- or $\mathbf{b}$-cycle) and the equivalence of two vectors $\vec c,\vec e\in\C^g$ is defined by $\vec c \equiv \vec e$ $\left(\mdp d\vec\Omega\right)$ if and only if $\vec c - \vec e = \vec j + \mathcal{B}_\Omega\vec m$ for some $\vec j,\vec m\in\Z^g$. 

\begin{proposition}
\label{prop:ni}
Solutions of \eqref{main-jip} are either unique or special. If \eqref{main-jip} is not uniquely solvable for some index $n$, then all the solutions for this index assume the form
\[
\sum_{j=1}^{g-2k}\tr_j+\sum_{j=1}^k\left(z_j^{(0)}+z_j^{(1)}\right),
\]
where the divisor $\sum_{j=1}^{g-2k}\tr_j$ is fixed and non-special and $\{z_j\}_{j=1}^k$ are arbitrary points in $\overline\C$.

If for some index $n$ a divisor solving \eqref{main-jip} has the form
\[
\sum_{i=1}^{g-l}\tr_i+k\infty^{(0)}+(l-k)\infty^{(1)}
\]
with $l>0$, $k\in\{0,\ldots,l\}$, and non-special $\sum_{i=1}^{g-l}\tr_i$ such that $\big|\pi(\tr_i)\big|<\infty$, then
\[
\sum_{i=1}^{g-l}\tr_i+(k+j)\infty^{(0)}+(l-k-j)\infty^{(1)}
\]
solves \eqref{main-jip} for the index $n+j$ for each $j\in\{-k,\ldots,l-k\}$. In particular, \eqref{main-jip} is uniquely solvable for the indices $n-k$ and $n+l-k$.

If $\sum_{j=1}^g\tr_{n,j}$ uniquely solves \eqref{main-jip} and does not contain $\infty^{(k)}$, $k\in\{0,1\}$, then \eqref{main-jip} is uniquely solvable for the index $n-(-1)^k$ and $\left\{\tr_{n,j}\right\}_{j=1}^g\cap\left\{\tr_{n-(-1)^k,j}\right\}_{j=1}^g=\varnothing$.
\end{proposition}

\begin{remark}
Propositions~\ref{prop:ni} says that the non-unique solutions of \eqref{main-jip} occur in blocks. The last unique solution before such a block consists of a non-special finite divisor and multiple copies of $\infty^{(1)}$. Trading one point $\infty^{(1)}$ for $\infty^{(0)}$ and leaving the rest of the points unchanged produces a solution of \eqref{main-jip} (necessarily non-unique as it contains an involution-symmetric pair $\infty^{(1)}+\infty^{(0)}$) for the subsequent index. Proceeding in this manner, a solution with the same non-special finite divisor and all the remaining points being $\infty^{(0)}$ is produced, which starts a block of unique solutions. In particular, there cannot be more than $g-1$ non-unique solutions in a row.
\end{remark}

\begin{definition}
\label{def:2}
In what follows, we always understand under $\sum_{j=1}^g\tr_{n,j}$ either the unique solution of \eqref{main-jip} or the solution where all the involution-symmetric pairs are taken to be $\infty^{(1)}+\infty^{(0)}$. Under this convention, given $\varepsilon>0$, we say that an index $n$ belongs to $\N_\varepsilon\subseteq\N$ if and only if
\begin{itemize}
\item[(i)] the divisor $\sum_{j=1}^g\tr_{n,j}$ satisfies $\big|\pi(\tr_{n,j})\big|\leq1/\varepsilon$ for all $\tr_{n,j}\in \RS^{(0)}$;
\item[(ii)] the divisor $\sum_{j=1}^g\tr_{n-1,j}$ satisfies $\big|\pi(\tr_{n-1,j})\big|\leq1/\varepsilon$ for all $\tr_{n-1,j}\in \RS^{(1)}$.
\end{itemize}
\end{definition}

To show that Definition~\ref{def:2} is meaningful we need to discuss limit points of $\{\sum_{j=1}^g\tr_{n,j}\}_{n\in\N^\prime}$, $\N^\prime\subset\N$, where convergence is understood in the topology of $\RS^g/\Sigma_g$, $\RS^g$ quotient by the symmetric group $\Sigma_g$. The following proposition shows that these limiting divisors posses the same block structure as the divisors themselves.

\begin{proposition}
\label{prop:sequence}
Let $\N^\prime$ be such that all the limit points of $\big\{\sum_{i=1}^g\tr_{n,i}\big\}_{n\in\N^\prime}$ assume the form
\begin{equation}
\label{limdiv1}
\sum_{i=1}^{g-2k-l_0-l_1}\tr_i+ \sum_{i=1}^k\left(z_i^{(0)}+z_i^{(1)}\right) + l_0\infty^{(0)} + l_1\infty^{(1)}
\end{equation}
for a fixed non-special divisor $\sum_{i=1}^{g-2k-l_0-l_1}\tr_i$, $\big|\pi(\tr_i)\big|<\infty$, and arbitrary $\{z_i\}_{i=1}^k\subset\C$. Then all the limit points of the sequence $\big\{\sum_{i=1}^g\tr_{n+j,i}\big\}_{n\in\N^\prime}$, $j\in\{-l_0-k,\ldots,l_1+k\}$, assume the form
\begin{equation}
\label{limdiv2}
\sum_{i=1}^{g-2k-l_0-l_1}\tr_i +  \sum_{i=1}^{k^\prime}\left(w_i^{(0)}+w_i^{(1)}\right) + \big(l_0+j+k-k^\prime\big)\infty^{(0)} + \big(l_1-j+k-k^\prime\big)\infty^{(1)},
\end{equation}
where $0\leq k^\prime\leq\min\left\{l_0+k+j,l_1+k-j\right\}$ and $\{w_i\}_{i=1}^{k^\prime}\subset\C$.

If $\big\{\sum_{i=1}^g\tr_{n,i}\big\}_{n\in\N^\prime}$ converges to a non-special divisor $\sum_{j=1}^g\tr_j$ that does not contain $\infty^{(k)}$, $k\in\{0,1\}$, then the sequence $\big\{\sum_{i=1}^g\tr_{n-(-1)^k,i}\big\}_{n\in\N^\prime}$ also converges, say to $\sum_{j=1}^g\w_j$, which is non-special, and $\left\{\tr_j\right\}_{j=1}^g\cap\left\{\w_j\right\}_{j=1}^g=\varnothing$.
\end{proposition}

\begin{remark}
Proposition~\ref{prop:sequence} shows that the sets $\N_\varepsilon$ are well-defined for all $\varepsilon$ small enough. Indeed, let $\big\{\sum_{i=1}^g\tr_{n,i}\big\}_{n\in\N^\prime}$ be a subsequence that converges in $\RS^g/\Sigma_g$ (it exists by compactness of $\RS$). Naturally, the limiting divisor can be written in the form \eqref{limdiv1}. Then it follows \eqref{limdiv2} that the sequence $\big\{\sum_{i=1}^g\tr_{n-l_0-k,i}\big\}_{n\in\N^\prime}$ converges to $\sum_{i=1}^{g-2k-l_0-l_1}\tr_i + \left(l_1+l_0+2k\right)\infty^{(1)}$. Further, by the second part of the proposition, the sequence $\big\{\sum_{i=1}^g\tr_{n-l_0-k-1,i}\big\}_{n\in\N^\prime}$ also converges and the limit, say $\sum_{j=1}^g\w_j$, does not contain $\infty^{(1)}$. Thus, $\{n-l_0-k:~n\in\N^\prime\}\subset\N_\varepsilon$ for any $\varepsilon$ satisfying $\varepsilon|\pi(\tr_i)|\leq1$ if $\tr_i\in\RS^{(0)}$ and $\varepsilon|\pi(\w_i)|\leq1$ if $\w_i\in\RS^{(1)}$.
\end{remark}

Equipped with the solutions of \eqref{main-jip}, we can construct the \emph{Szeg\H{o}} functions of $\rho$ on $\RS$, which are the solutions of a sequence of boundary value problems on $L$.

\begin{proposition}
\label{prop:Sn}
For each $n\in\N$ there exists a function, say $S_n$, with continuous traces on both sides of $\left(L\cup\bigcup_{k=1}^g\mathbf{a}_k\right)\setminus E$ such that $S_n\map^n$ is meromorphic in $\RS\setminus L$ and
\begin{equation}
\label{jumpSn}
(S_n\map^n)^- = (\rho/h^+)(S_n\map^n)^+ \quad \mbox{on} \quad L\setminus E.
\end{equation}
If we let $m(t)$ to be the number of times, possibly zero, $t$ appears in $\left\{\tr_{n,j}\right\}_{j=1}^g$, then $S_n$ is non-vanishing and finite except for
\begin{equation}
\label{SnE}
\left\{
\begin{array}{ll}
|S_n(z^{(k)})| \sim |z-a|^{m(a)/2-(-1)^k(1+2\alpha_a)/4} & \mbox{as} \quad z^{(k)}\to a\in\big\{a_j\big\}_{j=1}^p, \smallskip\\
|S_n(z^{(k)})| \sim |z-b|^{m(b)/2-1/2+(-1)^k/4} & \mbox{as} \quad z^{(k)}\to b\in\big\{b_j\big\}_{j=1}^{p-2m}, \smallskip\\
|S_n(z^{(1)})| \sim |z-b|^{m(b)-m_b/2} & \mbox{as} \quad z^{(1)}\to b\in\big\{b_j^{(1)}\big\}_{j=p-2m+1}^g,
\end{array}
\right.
\end{equation}
and has a zero of multiplicity $m(t)$ at each $t\in\left\{\tr_{n,j}\right\}_{j=1}^g\setminus\big(\left\{a_j\right\}_{j=1}^p\cup\big\{b_j^{(1)}\big\}_{j=1}^g\big)$ and $m_b/2$ is the multiplicity of $b$ in $\big\{b_j^{(1)}\big\}_{j=p-2m+1}^g$. 

Conversely, if for given $n\in\N$ there exists a function $S$ with continuous traces on $\left(L\cup\bigcup_{k=1}^g\mathbf{a}_k\right)\setminus E$ such that $S\map^n$ is meromorphic in $\RS\setminus L$ and $S$ satisfies \eqref{jumpSn} and \eqref{SnE} with $\sum_{j=1}^g\tr_{n,j}$ replaced by some divisor $\sum\tr_j$, then $\sum\tr_j$ solves \eqref{main-jip} for the index $n$ and $S=pS_n$ for a polynomial $p$ such that $(p\circ\pi)=\sum\tr_j-\sum\tr_{n,j}$.

Finally, given $\epsilon>0$ and $\varepsilon>0$, there exists constant  $C_{\epsilon,\varepsilon}<\infty$ such that
\begin{equation}
\label{SnNormalized}
\left|\frac{S_{n-1}(\z)}{S_n(\z)}\frac{S_n(\infty^{(0)})}{S_{n-1}(\infty^{(1)})}\right| \leq C_{\epsilon,\varepsilon}
\end{equation}
for $n\in\N_\varepsilon$ and $\z\in\RS_{n,\epsilon}:=\RS\setminus\cup_{j=1}^n N_\epsilon(\tr_{n,j})$, where $N_\epsilon(\tr_{n,j})$ is a connected neighborhood of $\tr_{n,j}$ such that $\pi(N_\epsilon(\tr_{n,j}))$ is the $\epsilon$-ball centered at $t_{n,j}$ in the spherical metric\footnote{That is, $\dist(z_1,z_2)=2|z_1-z_2|(1+|z_1|^2)^{-1/2}(1+|z_2|^2)^{-1/2}$ if $|z_1|,|z_2|<\infty$ and $\dist(z,\infty)=2(1+|z|^2)^{-1/2}$.}.
\end{proposition}

\begin{remark}
The integers $m(e)$, $e\in E$, in the first two lines of \eqref{SnE} are either 0 or 1 as otherwise $\sum_{j=1}^g\tr_{n,j}$ would be special. 
\end{remark}

\begin{remark}
The estimate in \eqref{SnNormalized} cannot be improved in a sense that if for some subsequence of indices $\varepsilon_n\to0$, $1/\varepsilon_n:=\max\big\{\max_{\tr_{n-1,j}\in\RS^{(1)}}\big|\pi(\tr_{n-1,j})\big|,\max_{\tr_{n,j}\in\RS^{(0)}}\big|\pi(\tr_{n,j})\big|\big\}$, then $C_{\epsilon,\varepsilon_n}\to\infty$. 
\end{remark}

\begin{remark}
We would like to stress that $S_n$ is unique for $n\in\N_\varepsilon$, $\varepsilon>0$, as \eqref{main-jip} is uniquely solvable for all such indices.
\end{remark}

Propositions~\ref{prop:ni}--\ref{prop:Sn} are proved in Section~\ref{s:s}. 

\subsection{Main Theorem}
\label{ss:MTh}

Let $\{[n/n]_{\widehat\rho}\}_{n\in\N}$ be the sequence of diagonal Pad\'e approximants to the function $\widehat\rho$. As before, denote by  $q_n$ the denominator polynomial of $[n/n]_{\widehat\rho}$ (Nuttall-Stahl orthogonal polynomial \eqref{ortho}) and by $R_n$ the reminder function of $[n/n]_{\widehat\rho}$ \eqref{linsys} (the function of the second kind \eqref{secondkind} for $q_n$). Recall that by $S_n$ and $S_n^*$ we denote the pull-back functions of $S_n$ on $\RS$ from $D^{(0)}$ and $D^{(1)}$ to $D^*$, respectively. Then the following theorem holds.

\begin{theorem}
\label{thm:SA}
Let $\Delta$ be a minimal capacity contour as constructed in Section~\ref{ss:AF} subject to Condition~\hyperref[cond1]{GP} and assumption $\xi=1$ in \eqref{a11}. Further, let $\widehat\rho$ be given by \eqref{hatrho} with $\rho\in\W$, $\N_\varepsilon$ be as in Definition~\ref{def:2} for fixed $\varepsilon>0$, and $S_n$ be given by Proposition~\ref{prop:Sn}. Then for all $n\in\N_\varepsilon$ it holds that
\begin{equation}
\label{SA1}
\left\{
\begin{array}{lll}
q_n &=& \displaystyle \left(1+\upsilon_{n1}\right) \gamma_nS_n\map^n + \upsilon_{n2}\gamma_n^* S_{n-1}\map^{n-1},\bigskip \\
R_n &=& \displaystyle \left(1+\upsilon_{n1}\right)\gamma_n \frac{hS_n^*}{\map^n}+\upsilon_{n2}\gamma_n^*\frac{hS_{n-1}^*}{\map^{n-1}},
\end{array}
\right.
\end{equation}
locally uniformly in $D^*$, where $|\upsilon_{nj}|\leq c(\varepsilon)/n$ in $\overline\C$ while $\upsilon_{nj}(\infty)=0$ and
\[
\gamma_n:=\frac{\cp(\Delta)^n}{S_n(\infty)} \quad \mbox{and} \quad \gamma_n^*:=\frac{\cp(\Delta)^{n+1}}{S_{n-1}^*(\infty)}.
\]
Moreover, it holds locally uniformly in $\Delta\setminus E$ that
\begin{equation}
\label{SA2}
\left\{
\begin{array}{lll}
q_n &=& \displaystyle \left(1+\upsilon_{n1}\right)\gamma_n\left(\left(S_n\map^n\right)^++\left(S_n\map^n\right)^-\right) + \upsilon_{n2}\gamma_n^*\left(\left(S_{n-1}\map^{n-1}\right)^++\left(S_{n-1}\map^{n-1}\right)^-\right), \bigskip \\
R_n^\pm &=& \displaystyle \left(1+\upsilon_{n1}\right) \gamma_n \left(\frac{hS_n^*}{\map^n}\right)^\pm+\upsilon_{n2} \gamma_n^* \left(\frac{hS_{n-1}^*}{\map^{n-1}}\right)^\pm.
\end{array}
\right.
\end{equation}
\end{theorem}

Before proceeding, we would like to make several remarks regarding the statement of Theorem~\ref{thm:SA}.

\begin{remark}
If the set $A$ consists of two points, then $\Delta$ is an interval joining them. In this case the conclusion of Theorem~\ref{thm:SA} is contained in \cite{Nut84,ApVA04,KMcLVAV04,BY10}. Moreover, the Riemann surface $\RS$ has genus zero and therefore $\map$ is simply the conformal map of $D^*$ onto $\{|z|>1\}$ mapping infinity into infinity and having positive derivative there, while $S_n=S_\rho$ is the classical Szeg\H{o} function.
\end{remark}

\begin{remark}
Notice that both pull-back functions $S_n$ and $hS_n^*$ are holomorphic $D^*$. Moreover, $S_n$ has exactly $g$ zeros on $\RS$ that \emph{do} depend on $n$. It can be deduced from \eqref{SA1} that $q_n$ has a zero in the vicinity of each zero of $S_n$ that belongs to $D^{(0)}$. These zeros are called spurious or wandering as their location is determined by the geometry of $\RS$ and, in general, they do not approach $\Delta$ with $n$ while the rest of the zeros of $q_n$ do. On the other hand, those zeros of $S_n$ that belong to $D^{(1)}$ are the zeros of the pull-back function $S_n^*$ and therefore describe locations of the zeros of $R_n$ (points of overinterpolation).
\end{remark}

\begin{remark}
Even though our analysis allows us to treat only normal indices that are also asymptotically normal, formulae \eqref{SA1} illuminate what happens in the degenerate cases. If for an index $n$ the solution of \eqref{main-jip} is unique and contains $l$ copies of $\infty^{(1)}$, the function $S_n^*$ vanishes at infinity with order $l$. The latter combined with the second line of \eqref{SA1} shows that $[n/n]_{\widehat\rho}$ is geometrically close to overinterpolating $\widehat\rho$ at infinity with order $l$. Then it is feasible that there exists a small perturbation of $\rho$ (which leaves the vector $\vec c_\rho$ unaltered) that turns the index $n$ into a last normal index before a block of size $l$ of non-normal indices, which corresponds to the fact that solutions of \eqref{main-jip} are special for the next $l-1$ indices and the solution for the index $n+l$ contains $l$ copies of $\infty^{(0)}$.
\end{remark}

Observe that $\widehat\rho-[n/n]_{\widehat\rho} = R_n/q_n$ by \eqref{linsys} applied with $f:=\widehat\rho$. Thus, the following result on uniform convergence is a consequence of Theorem~\ref{thm:SA}.

\begin{corollary}
\label{cor:SA}
Under the conditions of Theorem~\ref{thm:SA}, it holds for $n\in\N_\varepsilon$ that
\begin{equation}
\label{errorasyspun}
\widehat\rho-[n/n]_{\widehat\rho} = \big[1+\mathcal{O}(1/n)\big]\frac{S_n^*}{S_n}\frac{h}{\map^{2n}}
\end{equation}
in $D^*\cap\pi(\RS_{n,\epsilon})$, where $\mathcal{O}(1/n)$ is uniform for each fixed $\epsilon>0$.
\end{corollary}

\section{Extremal Domains}
\label{s:ED}

In this section we discuss existence and properties of the extremal domain $D^*\in\mathcal{D}_f$ for the function $f$, holomorphic at infinity that can be continued as a multi-valued function to the whole complex plane deprived of a polar set $A$, see  \eqref{analytE}. Recall that the compact set $\Delta:=\partial D^*$ defined in \eqref{minCAP} makes $f$ single-valued in its complement and has minimal logarithmic capacity among all such compacta.

As mentioned in the introduction, the question of existence and characteristic properties of $\Delta$ was settled by Stahl in the most general settings.  Namely, he showed that the following theorem holds \cite[Theorems~1 and~2]{St85} and \cite[Theorem~1]{St85b}.

\begin{stahl}
\label{thm:S}
Let $f \in \mathcal{A}(\overline\C\setminus A)$ with $\cp(A)=0$. Then there exists unique $D^*\in\mathcal{D}_f$, $\Delta=\overline\C\setminus D^*$, the extremal domain for $f$, such that
\[
\cp(\Delta) \leq \cp(\partial D) \quad \mbox{for any} \quad D\in\mathcal{D}_f,
\]
and if $\cp(\Delta)=\cp(\partial D)$ for some $D\in\mathcal{D}_f$, then $D\subset D^*$ and
$\cp(D^*\setminus D)=0$. Moreover, $\Delta:=E_0\cup E_1\cup\bigcup \Delta_k$,
where $E_0\subseteq A$, $E_1$ is a finite set of points, and $\Delta_k$ are open analytic Jordan arcs. Furthermore, it holds that
\begin{equation}
\label{symmetry}
\frac{\partial g_\Delta}{\partial\n_+}=\frac{\partial g_\Delta}{\partial\n_-} \quad \mbox{on} \quad \bigcup \Delta_k,
\end{equation}
where $g_\Delta$ is the Green function for $D^*$ and $\n^\pm$ are the one-sided normals on each $\Delta_k$.
\end{stahl}

Let now $f$ and $A$ be as in \eqref{pointsa}. Denote by $\mathcal{K}$ the collection of all compact sets $K$ such that $K$ is a union of a finite number of disjoint continua each of which contains at least two point from $A$ and $\overline\C\setminus K\in\mathcal{D}_f$. That is,
\[
\mathcal{K}:=\left\{ K:~K= \bigcup_{j=1}^{q<\infty} K_j,~\sharp(A\cap K_j)\ge 2; ~ K_j \setminus \partial K_j = \varnothing; ~ K_j \cap K_i = \varnothing, ~ i\ne j; ~ \overline\C\setminus K\in\mathcal{D}_f\right\}.
\]
Observe that the inclusion $\{\overline\C\setminus K:~K\in\mathcal{K}\}\subset D_f$ is proper.
However, it can be shown using the monodromy theorem (see, for example, \cite[Lemma~8]{BStY12})
that $\Delta\in\mathcal{K}$. Considering only functions with finitely many branch points and
sets in $\mathcal{K}$  allows significantly alter and simplify the proof of
Theorem~\hyperref[thm:S]{S}, \cite[Theorems~2 and~3]{uPerevRakh}.
Although \cite{uPerevRakh} has never been published, generalizations of the method proposed
there were  used to prove extensions of Theorem~\hyperref[thm:S]{S}
for  classes of weighted capacities, see  \cite{KamRakh05}, \cite{M-FRakh11} and \cite{BStY12}.
Below, in a sequence of propositions, we state the simplified version of
Theorem~\hyperref[thm:S]{S} and adduce its proof as devised in \cite{uPerevRakh}
solely for the completeness of the exposition.

\begin{proposition}
\label{prop:existence}
There exists $\Delta\in\mathcal{K}$ such that $\cp(\Delta)\leq\cp(K)$ for any $K\in\mathcal{K}$.
\end{proposition}
\begin{proof}
Let $\{K_n\}$ be a sequence in $\mathcal{K}$ such that
\[
\cp(K_n)\to \inf_{K\in\mathcal{K}}\cp(K)=:c \quad \mbox{as} \quad n\to\infty.
\]
Then there exists $R>0$ such that $K_n\subset\D_R:=\{z:~|z|<R\}$ for all $n$ large enough. Indeed, it is known \cite[Theorem~5.3.2]{Ransford} that $\cp(\Delta)\geq\cp(\gamma)\geq\frac14\diam(\gamma)$, where $\gamma$ is any continuum in $K_n$ and $\diam(\gamma)$ is the diameter of $\gamma$. As $\gamma$ contains at least two points from $A$, the claim follows.

For any $K\in\mathcal{K}_R:=\mathcal{K}\cap\overline\D_R$ and $\epsilon>0$, set $(K)_\epsilon:=\{z:\dist(z,K)<\epsilon\}$.  We endow $\mathcal{K}_R$ with the Hausdorff metric, i.e.,
\[
d_H(K_1,K_2) := \inf\{\epsilon:K_1\subset(K_2)_\epsilon,K_2\subset(K_1)_\epsilon\}.
\]
By standard properties of the Hausdorff distance \cite[Section 3.16]{Dieudonne}, $\clos_{d_H}(\mathcal{K}_R)$, the closure of $\mathcal{K}_R$ in the $d_H$-metric, is a compact metric space. Notice that a compact set which is the $d_H$-limit of a sequence of continua is itself a continuum. Observe also that the process of taking the $d_H$-limit cannot increase the number of the connected components since the $\epsilon$-neighborhoods of the components of the limiting set will become disjoint as $\epsilon\to0$. Thus, each element of $\clos_{d_H}(\mathcal{K}_R)$ still consists of a finite number of continua each containing at least two points from $A$ but possibly with multiply connected complement. However, the polynomial convex hull of such a set, that is, the union of the set with the bounded components of its complement, again belongs to $\mathcal{K}_R$ and has the same logarithmic capacity \cite[Theorem~5.2.3]{Ransford}.

Let $\Delta^*\in\clos_{d_H}(\mathcal{K}_R)$ be a limit point of $\{K_n\}$. In other words, $d_H(\Delta^*,K_n)\to0$ as $n\to\infty$, $n\in\N_1\subseteq\N$. We shall show that
\begin{equation}
\label{cpc}
\cp(\Delta^*)=c.
\end{equation}
To this end, denote by $K^\epsilon:=\{z:~g_K(z)\leq\log(1+\epsilon)\}$, $\epsilon>0$, where $g_K$ is the Green function with pole at infinity for the complement of $K$. It can be easily shown \cite[Theorem~5.2.1]{Ransford} that
\begin{equation}
\label{cpepsilon}
\cp(K^\epsilon)=(1+\epsilon)\cp(K).
\end{equation}
Put $c_0:=\inf\{\cp(\gamma)\}$, where the infimum is taken over all connected components $\gamma$ of $K_n$ and all $n\in\N_1$. Recall that each component $\gamma$ of any $K_n$ contains at least two points from $A$. Thus, it holds that $c_0>0$ since $\cp(\gamma)\geq\frac14\diam(\gamma)$.

We claim that for any $\epsilon\in(0,1)$ and $\delta<\epsilon^2c_0/2$ we have that
\begin{equation}
\label{claimincl}
(K_n)_\delta\subset K_n^\epsilon
\end{equation}
for all $n$ large enough. Granted the claim, it holds by \eqref{cpepsilon} that
\begin{equation}
\label{conseqclaim}
\cp(\Delta^*)\leq(1+\epsilon)\cp(K_n)
\end{equation}
since $\Delta^*\subset(K_n)_\delta\subset K_n^\epsilon$. Thus, by taking the limit as $n$ tends to infinity in \eqref{conseqclaim}, we get that
\begin{equation}
\label{conclclaim}
c\leq\cp(\Delta^*)\leq(1+\epsilon)c,
\end{equation}
where the lower bound follows from the very definition of $c$ since the polynomial convex hull of $\Delta^*$, say $\Delta$, belongs to $\mathcal{K}$ and has the same capacity as $\Delta^*$. As $\epsilon$ was arbitrary, \eqref{conclclaim} yields \eqref{cpc} with $\Delta$ as above.

It only remains to prove \eqref{claimincl}. We show first that for any continuum $\gamma$ with at least two points, it holds that
\begin{equation}
\label{rakhper}
\dist(\gamma,\widetilde\gamma^\epsilon) \geq \frac{\epsilon^2}{2}\cp(\gamma),
\end{equation}
where $\widetilde\gamma^\epsilon:=\{z:~g_{\gamma}(z)=\log(1+\epsilon)\}$. Let $\Psi$ be a conformal map of $\{z:~|z|>1\}$ onto $\overline\C\setminus\gamma$, $\Psi(\infty)=\infty$. It can be readily verified that $|\Psi(z)z^{-1}|\to\cp(\gamma)$ as $z\to\infty$ and that $g_\gamma=\log|\Psi^{-1}|$, where $\Psi^{-1}$ is the inverse of $\Psi$. Then it follows from \cite[Theorem~IV.2.1]{Goluzin} that
\begin{equation}
\label{goluzin}
|\Psi^\prime(z)| \geq \cp(\gamma)\left(1-\frac{1}{|z|^2}\right), \quad |z|>1.
\end{equation}
Let $z_1\in\gamma$ and $z_2\in\widetilde\gamma^\epsilon$ be such that $\dist(\gamma,\widetilde\gamma^\epsilon)=|z_1-z_2|$. Denote by $[z_1,z_2]$ the segment joining $z_1$ and $z_2$.  Observe that $\Psi^{-1}$ maps the annular domain bounded by $\gamma$ and $\widetilde\gamma^\epsilon$ onto the annulus $\{z:1<|z|<1+\epsilon\}$. Denote by $S$ the intersection of $\Psi^{-1}((z_1,z_2))$ with this annulus. Clearly, the angular projection of $S$ onto the real line is equal to $(1,1+\epsilon)$. Then
\begin{eqnarray}
\dist(\gamma,\widetilde\gamma^\epsilon) &=& \int_{(z_1,z_2)}|dz| = \int_{\Psi^{-1}((z_1,z_2))}|\Psi^\prime(z)||dz| \geq \cp(\gamma)\int_{\Phi^{-1}((z_1,z_2))}\left(1-\frac{1}{|z|^2}\right)|dz| \nonumber \\
{} &\geq& \cp(\gamma)\int_{S}\left(1-\frac{1}{|z|^2}\right)|dz| \geq \cp(\gamma)\int_{(1,1+\epsilon)}\left(1-\frac{1}{|z|^2}\right)|dz| = \frac{\epsilon^2\cp(\gamma)}{1+\epsilon}, \nonumber
\end{eqnarray}
where we used \eqref{goluzin}. This proves \eqref{rakhper} since it is assumed that $\epsilon\leq1$.

Now, let $\gamma_n$ be a connected component of $K_n$ such that $\dist(K_n,\widetilde K_n^\epsilon)=\dist(\gamma_n,\widetilde K_n^\epsilon)$. By the maximal principle for harmonic functions, it holds that $g_{\gamma_n}>g_{K_n}$ for $z\notin K_n$, and therefore, $\gamma_n^\epsilon\subset K_n^\epsilon$. Thus,
\[
\dist(K_n,\widetilde K_n^\epsilon)\geq\dist(\gamma_n,\widetilde\gamma_n^\epsilon)\geq \frac{\epsilon^2c_0}{2}
\]
by \eqref{rakhper} and the definition of $c_0$. This finishes the proof of the proposition.
\end{proof}

Let $\Delta$ be as in Proposition~\ref{prop:existence}. Observe right away that $\Delta$ has no interior as otherwise there would exist $\Delta^\prime\subset\Delta$ with smaller logarithmic capacity which still belongs to $\mathcal{K}$. It turns out that $g_\Delta$ has a rather special structure that we describe in the following proposition which was initially proven in this form in \cite[Theorem~3]{uPerevRakh} (the method of proof in a more general form was also used in \cite{M-FRakh11}).

\begin{proposition}
\label{prop:greenfun}
Let $\Delta$ be as in Proposition~\ref{prop:existence}. Then
\begin{equation}
\label{greendelta}
g_\Delta(z) = \re\left(\int_{a_1}^z\sqrt\frac{B(\tau)}{A(\tau)}d\tau\right),
\end{equation}
where $A$ was defined in \eqref{functionh}, $B$ is a monic polynomial of degree $p-2$, and the root is chosen so that $z\sqrt{A(z)/B(z)}\to1$ as $z\to\infty$.
\end{proposition}
\begin{proof}
Denote by $\omega_K$ the equilibrium measure of a compact set $K$ and by $I[\mu]$ the logarithmic energy of a compactly supported measure $\mu$, i.e.,
\[
I[\mu] = -\iint\log|z-\tau|d\mu(z)d\mu(\tau).
\]
Then it is known that
\[
g_\Delta(z) = I[\omega_\Delta] +\int\log|z-\tau|d\omega_\Delta(\tau),
\]
which immediately implies that
\begin{equation}
\label{partialzg}
(\partial_zg_\Delta)(z) = \frac12\int\frac{d\omega_\Delta(\tau)}{z-\tau},
\end{equation}
where $\partial_z:=(\partial_x-i\partial_y)/2$. Since $g_\Delta\equiv0$ on $\Delta$, it holds that
\[
g_\Delta(z) = \re\left(2\int_{a_1}^z(\partial_zg_\Delta)(\tau)d\tau\right)
\]
for any the path of integration in $D$. Thus, to prove \eqref{greendelta}, we need to show that
\begin{equation}
\label{BoverA}
\frac{B(z)}{A(z)} = \left(\int\frac{d\omega_\Delta(\tau)}{z-\tau}\right)^2
\end{equation}
for some monic polynomial $B$, $\deg(B)=p-2$.

Let $O$ be a neighborhood of $\Delta$. Define
\[
\delta(z) := \frac{A(z)}{z-u}, \quad u\notin\overline O.
\]
Then $\delta$ generates a local variation of $O$ according to the rule $z\mapsto z^t:=z+t\delta(z)$, where $t$ is a complex parameter. Since
\begin{equation}
\label{subs}
\left|w^t-z^t\right|= |w-z|\left|1+t\frac{\delta(w)-\delta(z)}{w-z}\right|,
\end{equation}
this transformation is injective for all $|t|\leq t_0<M$, where
\begin{equation}
\label{M}
M:=\max_{w,z\in\overline O}|(\delta(w)-\delta(z))/(w-z)|<\infty.
\end{equation}
Moreover, the transformation $\delta$ naturally induces variation of sets in $\overline O$, $E\mapsto E^t=\{z^t:z\in E\}$, and measures supported in $\overline O$, $\mu\mapsto\mu^t$, $\mu^t(E^t)=\mu(E)$.

Let $\mu$ be a positive measure supported in $\overline O$ with finite logarithmic energy $I[\mu]$. Observe that the pull-back measure $\mu^t$ satisfies the following substitution rule: $d\mu^t(z^t)=d\mu(z)$. Then it follows from \eqref{subs} that
\begin{eqnarray}
I[\mu^t]-I[\mu] &=& -\iint\log\left|1+t\frac{\delta(w)-\delta(z)}{w-z}\right|d\mu(z)d\mu(w) \nonumber \\
\label{diffenergy1}
{} &=& -\re\left[\iint\log\left(1+t\frac{\delta(w)-\delta(z)}{w-z}\right)d\mu(z)d\mu(w)\right]
\end{eqnarray}
for all $|t|\leq t_0$. Since the argument of the logarithm in \eqref{diffenergy1} is less than 2 in modulus, it holds that
\begin{equation}
\label{diffenergy2}
I[\mu^t]-I[\mu] = -\re\left[t\delta(\mu)+O(t^2)\right]
\end{equation}
for all $|t|\leq t_0$, where
\[
\delta(\mu) := \iint\frac{\delta(w)-\delta(z)}{w-z}d\mu(z)d\mu(w).
\]

Let now $\{\mu_t\}$ be a family of measures on $\Delta$ such that $(\mu_t)^t=\omega_{\Delta^t}$. Then
\begin{equation}
\label{convdeltas}
\delta(\mu_t)\to\delta(\omega_\Delta) \quad \mbox{as} \quad t\to0.
\end{equation}
Indeed, by the very definition of the equilibrium measure it holds that the differences $I[\mu_t]-I[\omega_\Delta]$ and $I[\omega_\Delta^t]-I[\omega_{\Delta^t}]$ are non-negative. Thus,
\begin{eqnarray}
0 \leq I[\mu_t]-I[\omega_\Delta] &=& I[\mu_t]-I[\omega_{\Delta^t}]+I[\omega_{\Delta^t}]-I[\omega_\Delta] \nonumber  \\
{} &\leq& I[\mu_t]-I[\omega_{\Delta^t}]+I[\omega_\Delta^t]-I[\omega_\Delta] \nonumber \\
\label{diffenergy3}
{} &=& \re\left[t(\delta(\mu_t)-\delta(\omega_\Delta))+O(t^2)\right]
\end{eqnarray}
for $|t|\leq t_0$ by \eqref{diffenergy2}. Clearly, $|\delta(\mu)|\leq M|\mu|$ by \eqref{M}, where $|\mu|$ is the total variation of $\mu$. Since $\mu_t$ and $\omega_\Delta$ are positive measures of unit mass, \eqref{diffenergy3} implies that $I[\mu_t]\to I[\omega_\Delta]$ as $t\to 0$. The latter yields that $\mu_t\cws\omega_\Delta$ by the uniqueness of the equilibrium measure,\footnote{The measure $\omega_\Delta$ is the unique probability measure that minimizes energy functional $I[\cdot]$ among all probability measures supported on $\Delta$. As any weak limit point of $\{\mu_t\}$ has the same energy as $\omega_\Delta$ by the Principle of Descent \cite[Theorem~I.6.8]{SaffTotik} and \eqref{diffenergy3}, the claim follows.} which immediately implies \eqref{convdeltas} by the very definition of weak$^*$ convergence.

Now, observe that $a_k^t=a_k$ for any $k\in\{1,\ldots,p\}$. Hence, $\Delta^t\in\mathcal{K}$ for all $|t|\leq t_0$. In particular, this means that $\cp(\Delta^t)\geq\cp(\Delta)$ and therefore $I[\omega_{\Delta^t}]\leq I[\omega_\Delta]$ as $\cp(K)=\exp\{-I[\omega_K]\}$. Thus, it holds that
\begin{equation}
\label{diffenergy4}
0 \leq I[\omega_\Delta] - I[\omega_{\Delta^t}] \leq I[\mu_t] - I[\omega_{\Delta^t}] = \re\left[t\delta(\mu_t)+O(t^2)\right] = \re\left[t\delta(\omega_\Delta)+o(t)\right]
\end{equation}
by \eqref{diffenergy2} and \eqref{convdeltas}. Clearly, \eqref{diffenergy4} is positive only if
\begin{equation}
\label{deltaeqzero}
\delta(\omega_\Delta)=0.
\end{equation}

In another connection, observe that there exists a polynomial in $u$, say
\[
B(u;z,w) = a_0(z,w) + a_1(z,w) u + \cdots + a_{p-3}(z,w) u^{p-3} + u^{p-2},
\]
where each $a_k(z,w)$ is a polynomial in $z$ and $w$, such that
\begin{equation}
\label{longpoly}
(w-u)A(z) - (z-u)A(w) + (z-w)A(u) = (z-w)(z-u)(w-u)B(u;z,w).
\end{equation}
Indeed, the left hand side of \eqref{longpoly} is a polynomial of degree $p$ in each of the variables $z,w,u$ that vanishes when $u=z$, $u=w$, and $z=w$. Then
\[
\frac{\delta(z)-\delta(w)}{z-w} =  \frac{A(z)}{(z-u)(z-w)} - \frac{A(w)}{(w-u)(z-w)} = B(u;z,w) - \frac{A(u)}{(z-u)(w-u)}
\]
by \eqref{longpoly}. So, we have by the definition of $\delta(\omega_\Delta)$ and \eqref{deltaeqzero} that
\[
B(u):=\int B(u;z,w)d\omega_\Delta(z)d\omega_\Delta(w) = A(u)\iint\frac{d\omega_\Delta(z)d\omega_\Delta(w)}{(z-u)(w-u)} = A(u)\left(\int\frac{d\omega_\Delta(\tau)}{\tau-u}\right)^2,
\]
which shows the validity of \eqref{BoverA} and respectively of \eqref{greendelta}.
\end{proof}

Having Proposition~\ref{prop:greenfun}, we can describe the structure of a set $\Delta$ as it was done in \cite{St85b} with the help of the critical trajectories of a quadratic differential \cite[Section~8]{Pommerenke2}. Recall that a \emph{quadratic differential} is the expression of the form $Q(z)dz^2$, where $Q$ is a meromorphic function in some domain. We are interested only in the case where $Q$ is a rational function.

A trajectory of the quadratic differential $Q(z)dz^2$ is a smooth (in fact, analytic) maximal Jordan arc or curve such that $Q(z(t))(z^\prime(t))^2>0$ for any parametrization. The zeros and poles of the differential are called \emph{critical points}.  The zeros and simple poles of the differential are called \emph{finite critical points} (the order of the point at infinity is equal to the order of $Q$ at infinity minus 4; for instance, if $Q$ has a double zero at infinity, then $Q(z)dz^2$ has a double pole there).

A trajectory is called \emph{critical} if it joins two not necessarily distinct critical points and at least one of them is finite. A trajectory is called \emph{closed} if it is a Jordan curve and is called \emph{recurrent} if its closure has non-trivial planar Lebesgue measure (such a trajectory is not a Jordan arc or a curve). A differential is called \emph{closed} if it has only critical and closed trajectories.

If $e$ is a finite critical point of order $k$, then there are $k+2$ critical trajectories emanating from $e$ under equally spaced angles. If $e$ is a double pole and the differential has a positive residue at $e$, then there are no trajectories emanating from $e$ and the trajectories around $e$ are closed, that is, they encircle $e$.

\begin{proposition}
\label{prop:structure}
Let $\Delta$ be as in Proposition~\ref{prop:existence} and the polynomials $A,B$ be as in Proposition~\ref{prop:greenfun}. Then \eqref{Delta} holds with $\bigcup \Delta_k$ being the union of the non-closed critical trajectories of the closed quadratic differential $(-B/A)(z)dz^2$ and $E$ being the symmetric difference of the set $\{a_1,\ldots,a_p\}$ and the set of those zeros of $B$ that belong to the closure of $\bigcup \Delta_k$. The remaining zeros of $B$, say $\{b_1,\ldots,b_q\}$, are of even order and of total multiplicity $2(m-1)$, where $m$ is the number of connected components of $\Delta$. Furthermore, \eqref{symmetry} holds.
\end{proposition}
\begin{proof}
In \cite[Lemma~5.2]{M-FRakh11} it is shown that $\Delta$ is a subset of the closure of the critical trajectories of $(-B/A)(z)dz^2$.  Since \eqref{greendelta} can be rewritten as
\[
g_\Delta(z) = \im\left(\int_{a_1}^z\sqrt{-\frac{B(\tau)}{A(\tau)}}d\tau\right), \quad z\in D,
\]
the critical trajectories of $(-B/A)(z)dz^2$ are the level lines of $g_\Delta$ and therefore $(-B/A)(z)dz^2$ is a closed differential. By its very nature, $\Delta$ has connected complement and therefore the closed critical trajectories do not belong to $\Delta$. Since $2\partial_z g_\Delta=\sqrt{B/A}$ is holomorphic in $D^*$, all the non-closed critical trajectories belong to $\Delta$ and all the zeros of $B$ that belong to the closed critical trajectories are of even order. Let us show that their total multiplicity is equal to $2(m-1)$. This will follow from the fact that the total multiplicity of the zeros of $B$ belonging to any connected component of $\Delta$ is equal to the the number of zeros of $A$ belonging to the same component minus~2.

To prove the claim, we introduce the following counting process. Given a connected compact set with connected complement consisting of open Jordan arcs and connecting isolated points, we call a connecting point outer if there is only one arc emanating from it, otherwise we call it inner. Assume further that there are at least 3 arcs emanating from each inner connecting point. We count inner connecting points according to their multiplicity which we define to be the number of arcs incident with the point minus 2. Suppose further that the number of outer points is $p^\prime$ and the number of inner connecting points is $p^\prime-2$ counting multiplicities. Now, form a new connected set in the following fashion. Fix $\tilde p$ of the previously outer points and link each of them by Jordan arcs to $\hat p$ chosen distinct points in the complex plane in such a fashion that the new set still has connected complement and each of the previously outer points is connected to at least 2 newly chosen points. Then the new set has $p^\prime-\tilde p+\hat p$ outer connecting points ($p^\prime-\tilde p$ of the old ones and $\hat p$ of the new ones) and $p^\prime-2+(\hat p+\tilde p-2\tilde p)=p^\prime-\tilde p+\hat p-2$ inner connecting points. That is, the difference between the outer and inner connecting points is again 2. Clearly, starting from any to points in $E$ connected by a Jordan arc, one can use the previous process to recover the whole connected component of $\Delta$ containing those two points, which proves the claim.

To prove \eqref{symmetry}, observe that by \eqref{partialzg} and \eqref{BoverA} we have that
\[
\frac{\partial g_\Delta}{\partial \n_\pm} = 2\re\left(n_\pm\partial_zg_\Delta\right) = \re\left[n_\pm\left(\sqrt\frac{B}{A}\right)_\pm\right],
\]
where $n_\pm$ are the unimodular complex numbers corresponding to $\n_\pm$. Since the tangential derivative of $g_\Delta$ is zero, so is the imaginary part of the product $n_\pm(\sqrt{B/A})_\pm$. Since, $n_+=-n_-$ and $(\sqrt{B/A})_+=-(\sqrt{B/A})_-$, \eqref{symmetry} follows.
\end{proof}

Propositions~\ref{prop:existence}--\ref{prop:structure} are sufficient to prove Theorem~\ref{thm:SA}. As an offshoot of Theorem~\ref{thm:SA} we get that the contour $\Delta$ is unique since \eqref{SA1} and \eqref{SA2} imply that all but finitely many zeros of $[n/n]_{\hat\rho}$ converge to $\Delta$. However, this fact can be proved directly \cite[Thm.~2]{St85}. Moreover, it can be shown that property \eqref{symmetry} uniquely characterizes $\Delta$ among smooth cuts making $f$ single-valued \cite[Thm.~6]{BStY12}.

\begin{figure}[!ht]
\centering
\subfloat[]{\includegraphics[scale=.5]{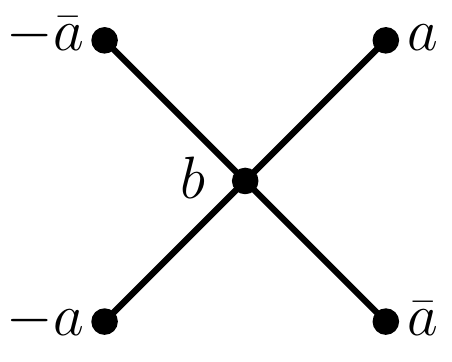}}\quad\quad\quad
\subfloat[]{\includegraphics[scale=.5]{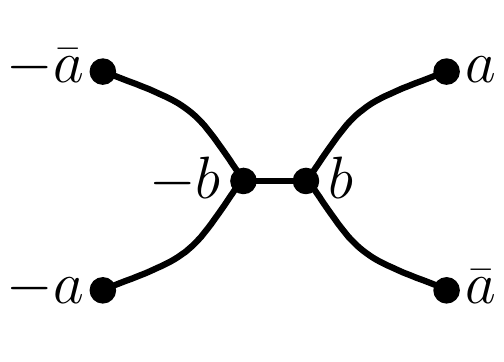}}
\caption{\small Both figures depict the set $\Delta$ for the quadratic differential $\displaystyle \frac{(z^2-b^2)dz^2}{z^4-(a^2+\bar a^2)z^2+1}$, where $|a|=1$ and $b$ is real and depends on $a$. On the left-hand figure $\Arg(a)=\pi/4$ which forces $b=0$. On the right-hand figure, $\Arg(a)<\pi/4$, in which case $b>0$.}
\label{fig:GP}
\end{figure}

While Propositions~\ref{prop:existence}--\ref{prop:structure} deal with the most general situation of an arbitrary finite set $A$, Condition~\hyperref[cond1]{GP} introduced in Section~\ref{sec:main} is designed to rule out some degenerate cases. Namely, it possible for some zeros of the polynomial $B$ to coincide with some zeros of the polynomial $A$. This happens, for example, when all the points in $A$ are collinear. In this case, the minimal capacity cut $\Delta$ is simply the smallest line segment containing all the points in $A$, and the zeros of $B$ are exactly the zeros of $A$  excluding two that are the end points of $\Delta$. It is also possible for the polynomial $B$ to have zeros of multiplicities greater than one that belong to $\Delta$. These zeros serve as endpoints to more than three arcs (multiplicity plus 2), see Figure~\ref{fig:GP}A. However, under small perturbations of the set $A$, these zeros separate to form a set $\Delta$ satisfying Condition~\hyperref[cond1]{GP}, see Figure~\ref{fig:GP}B.

\section{Riemann Surface}
\label{S:RS}

Let $\RS$ be the Riemann surface of $h$ described in Section~\ref{ss:RS}. That is, $\RS$ is a hyperelliptic Riemann surface of genus $g$ with $2g+2$ branch (ramification) points $E$ and the canonical projection (covering map) $\pi:\RS\to\overline\C$. We use bold letters $\z,\w,\tr$ to denote generic points on $\RS$ and designate the symbol $\cdot^*$ to stand for the conformal involution acting on the points of $\RS$ according to the rule
\[
\z^*=z^{(1-k)} \quad \mbox{for} \quad \z=z^{(k)}, \quad k\in\{0,1\}.
\]

\subsection{Abelian Differentials}
\label{ss:intro}

For a rational function on $\RS$, say $f$, we denote by $(f)$ the divisor of $f$, i.e., a formal symbol defined by
\[
(f) := \sum_{\z:~f(\z)=0}\z - \sum_{\w:~f(\w)=\infty}\w,
\]
where each zero $\z$ (resp. pole  $\w$) appears as many times as its multiplicity. A meromorphic differential on $\RS$ is a differential of the form $fdz$, where $f$ is a rational function on $\RS$. The divisor of $fdz$ is defined by
\[
(fdz) := (f) + (dz) = (f) + \sum_{e\in E}e - 2\infty^{(0)} - 2\infty^{(1)}.
\]

It is more convenient to write meromorphic differentials with the help of
\[
\hbar(z^{(k)}) := (-1)^k\bigg(\prod_{k\in E}(z-e_k)\bigg)^{1/2}, \quad k\in\{0,1\},
\]
where the square root is taken so $\hbar(z^{(0)})/z^{g+1}\to1$ as $z\to\infty$. Clearly, $\hbar$ is a rational function on $\RS$ with the divisor
\[
(\hbar) = \sum_{e\in E}e - (g+1)\infty^{(0)} - (g+1)\infty^{(1)}.
\]
Then arbitrary meromorphic differential can be written as $fd\widetilde\Omega$, $d\widetilde\Omega := dz/\hbar$, and respectively
\[
(fd\widetilde\Omega) = (f) + (g-1)\infty^{(0)} + (g-1)\infty^{(1)}.
\]

A meromorphic differential is called holomorphic if $(fd\widetilde\Omega)\geq0$ (its divisor is \emph{integral}). Since for any polynomial $\ell$ it holds that
\[
(\ell\circ\pi) = \sum_{z:~\ell(z)=0}\left(z^{(0)}+z^{(1)}\right) - \deg(\ell)\left(\infty^{(0)}+\infty^{(1)}\right),
\]
the holomorphic differentials are exactly those of the form $\ell d\widetilde\Omega$, $\deg(\ell)<g$. Thus, there are exactly $g$ linearly independent holomorphic differentials on $\RS$.  Under the normalization \eqref{periodsa}, these are exactly the differentials $d\Omega_k$.

Let now $\w_1,\w_2\in\widetilde\RS$, $\w_1\neq\w_2$. We denote by $d\Omega_{\w_1,\w_2}$ the abelian differential of the third kind having two simple poles at $\w_1$ and $\w_2$  with respective residues $1$ and $-1$ and normalized so
\begin{equation}
\label{thirdkinda}
\oint_{\mathbf{a}_j}d\Omega_{\w_1,\w_2}=0, \quad j\in\{1,\ldots,g\}.
\end{equation}
It is also known that
\begin{equation}
\label{thirdkindb}
\oint_{\mathbf{b}_j}d\Omega_{\w_1,\w_2} = -2\pi i\int_{\w_1}^{\w_2} d\Omega_j, \quad j\in\{1,\ldots,g\},
\end{equation}
where the path of integration lies entirely in $\widetilde\RS$.  

\subsection{Green Differential}
\label{ss:gd}
The Green differential $dG$ is the differential $ d\Omega_{\infty^{(1)},\infty^{(0)}}$ modified by a suitable holomorphic differential to have purely imaginary periods. In fact, it holds that
\begin{equation}
\label{GreenDiff}
dG(\z) =h(\z)dz, \quad \z\in\RS.
\end{equation}
Indeed, the value of the integral of $dG$ along any cycle in $\RS\setminus\left\{\infty^{(0)}\cup\infty^{(1)}\right\}$ is purely imaginary as it is a linear combination with integer coefficients of its periods on the $\mathbf{a}$- and $\mathbf{b}$-cycles and the residues at $\infty^{(0)}$ and $\infty^{(1)}$ with purely imaginary coefficients. Thus, $\re\left(\int_{a_1}^{z^{(0)}}dG\right)$ is the Green function for $D^{(0)}$ and therefore is equal to $g_D$ lifted to $D^{(0)}$. Hence, the claim follows from Proposition~\ref{prop:greenfun}.

For $\z\in\RS\setminus\left\{\infty^{(0)},\infty^{(1)}\right\}$, put
\begin{equation}
\label{ComplGreenFun}
G(\z) := \int_{a_1}^\z dG.
\end{equation}
Then $G$ is a multi-valued analytic function on $\RS\setminus\left\{\infty^{(0)},\infty^{(1)}\right\}$ which is single-valued in $\widetilde\RS$. Moreover, it easily follows from \eqref{GreenDiff} and the fact that $a_1$ is a branch point for $\RS$ that
\begin{equation}
\label{GreenFunPr1}
G(z^{(0)})+G(z^{(1)}) = 0 \quad (\mbox{mod } 2\pi i) \quad \mbox{in} \quad D_a.
\end{equation}
Furthermore, for any point $\z\in\bigcup_{k=1}^g(\mathbf{a}_k\cup\mathbf{b}_k)$ it holds that
\begin{equation}
\label{GreenFunPr2}
G^+(\z)-G^-(\z) = \left\{
\begin{array}{rcl}
\displaystyle -\oint_{\mathbf{b}_k}dG, & \mbox{if} & \z\in\mathbf{a}_k\setminus\mathbf{b}_k, \\
\displaystyle  \oint_{\mathbf{a}_k}dG, & \mbox{if} & \z\in\mathbf{b}_k\setminus\mathbf{a}_k,
\end{array}
\right.
=
\left\{
\begin{array}{rcl}
\displaystyle 2\pi i\omega_k, & \mbox{if} & \z\in\mathbf{a}_k\setminus\mathbf{b}_k, \bigskip \\
\displaystyle 2\pi i\tau_k, & \mbox{if} & \z\in\mathbf{b}_k\setminus\mathbf{a}_k,
\end{array}
\right.
\end{equation}
where the constants $\omega_k$ and $\tau_k$ were defined in \eqref{GreenPeriods} (clearly, the integrated differential in \eqref{GreenPeriods} is $dG$). As all the periods of $dG$ are purely imaginary, the constants $\omega_k$ and $\tau_k$ are real. With the above notation, we can write
\begin{equation}
\label{dG1}
dG = d\Omega_{\infty^{(1)},\infty^{(0)}}+2\pi i\sum_{j=1}^g\tau_jd\Omega_j.
\end{equation}
Indeed, the difference between $dG$ and the right-hand side of \eqref{dG1} is a holomorphic differential with zero periods on $\mathbf{a}$-cycles and therefore is identically zero since it should be a linear combination of differentials satisfying \eqref{periodsa}. In particular, it follows from \eqref{thirdkindb} and \eqref{dG1} that
\begin{equation}
\label{lateaddition}
\int_{\infty^{(1)}}^{\infty^{(0)}}d\vec\Omega = \vec\omega+\mathcal{B}_\Omega\vec\tau.
\end{equation}

Using the Green deferential $dG$, we can equivalently redefine $\map$ introduced in \eqref{map} by
\begin{equation}
\label{defmap}
\map := \exp\{G\}, \quad \map(z^{(0)})=\frac{z}{\cp(\Delta)}+\cdots.
\end{equation}
Then $\map$ is a meromorphic function on $\widetilde\RS$ with a simple pole at $\infty^{(0)}$, a simple zero at $\infty^{(1)}$, otherwise non-vanishing and finite. Moreover, $\map$ possesses continuous traces on both sides of each $\mathbf{a}_k$ and $\mathbf{b}_k$ that satisfy \eqref{jumpPhiperoids} by \eqref{GreenFunPr2}, and \eqref{Da} by \eqref{GreenFunPr1}. Let us also mention that the pull-back function of $\map$ from $\RS^{(0)}$ onto $D_a$, which we continue to denote by $\map$, is holomorphic and non-vanishing in $D_a$ except for a simple pole at infinity. It possesses continuous traces that satisfy
\begin{equation}
\label{jumponab}
\left\{
\begin{array}{lcll}
\map^+ /\map^- & = &\exp\left\{2\pi i\omega_k\right\} & \mbox{on} \quad \Delta_k^a \smallskip \\
\map^-\map^+ & = & \exp\{2\pi i \delta_k\} & \mbox{on} \quad \Delta_k,
\end{array}
\right.
\end{equation}
by \eqref{jumpPhiperoids}, \eqref{Da}, and the holomorphy of $\map$ across those cycles $L_k$ that are not the $\mathbf{b}$-cycles, where we set $\delta_k:=\tau_{j_k}$ if $\Delta_k=\pi(\mathbf{b}_{j_k})$, and $\delta_k:=0$ otherwise.

\subsection{Jacobi Inversion Problem}
\label{ss:jip}
Let $r$ be a rational function on $\overline\C$. Then $r\circ\pi$ is a rational function on $\RS$ with the involution-symmetric divisor, i.e.,
\[
(r\circ\pi) = \sum_j\tr_j-\sum_j\w_j = \sum_j\tr_j^*-\sum_j\w_j^*.
\]
As $\RS$ is hyperelliptic, any rational function over $\RS$ with fewer or equal to $g$ poles is necessarily of this form. Recall that a divisor is called \emph{principal} if it is a divisor of a rational function. Thus, the involution-symmetric divisors are always principal. By Abel's theorem, a divisor $\sum_{j=1}^k\tr_j-\sum_{j=1}^l\w_j$ is principal if and only if $k=l$ and
\[
\sum\int_{\w_j}^{\tr_j}d\vec\Omega \, \equiv \, \vec 0 \quad \left(\mdp d\vec\Omega\right).
\]
In fact, it is known that given an arbitrary integral divisor $\sum_{j=1}^g\w_j$, for any vector $\vec c$ there exists an integral divisor $\sum_{j=1}^g\tr_j$ such that
\begin{equation}
\label{general-jip}
\sum_{j=1}^g\int_{\w_j}^{\tr_j}d\vec\Omega \, \equiv \, \vec c \quad \left(\mdp d\vec\Omega\right).
\end{equation}
The problem of finding a divisor $\sum_{j=1}^g\tr_j$ for given $\vec c$ is called the Jacobi inversion problem. The solution of this problem is unique up to a principal divisor. That is, if
\begin{equation}
\label{equiv-divisor}
\sum_{j=1}^g\tr_j - \big\{\mbox{ principal divisor }\big\}
\end{equation}
is an integral divisor, then it also solves \eqref{general-jip}. Immediately one can see that the principle divisor in \eqref{equiv-divisor} should have at most $g$ poles. As discussed before, such divisors come only from rational functions over $\overline\C$. Hence, if $\sum_{j=1}^g\tr_j$, a solution of \eqref{general-jip}, is special, that is, contains at least one pair of involution-symmetric points, then replacing this pair by another such pair produces a different solution of the same Jacobi inversion problem. However, if $\sum_{j=1}^g\tr_j$ is not special, then it solves \eqref{general-jip} uniquely.

\subsection{Riemann Theta Function}
\label{ss:rtf}

\emph{Theta function} associated to $\mathcal{B}_\Omega$ is an entire transcendental function of $g$ complex variables defined by
\[
\theta\left(\vec u\right) := \sum_{\vec n\in\Z^g}\exp\bigg\{\pi i\vec n^T\mathcal{B}_\Omega\vec n + 2\pi i\vec n^T\vec u\bigg\}, \quad \vec u\in\C^g.
\]
As shown by Riemann, the symmetry of $\mathcal{B}_\Omega$ and positive definiteness of its imaginary part ensures the convergence of the series for any $\vec u$. It can be directly checked that $\theta$ enjoys the following periodicity properties:
\begin{equation}
\label{periodicity}
\theta\left(\vec u + \vec j + \mathcal{B}_\Omega\vec m\right) = \exp\bigg\{-\pi i\vec m^T\mathcal{B}_\Omega\vec m - 2\pi i\vec m^T\vec u\bigg\}\theta\big(\vec u\big), \quad \vec j,\vec m\in\Z^g.
\end{equation}

The theta function can be lifted to $\RS$ in the following manner. Define a vector  $\vec\Omega$ of holomorphic and single-valued functions in $\widetilde\RS$ by
\begin{equation}
\label{fun-1k}
\vec\Omega(\z) := \int_{a_1}^{\z}d\vec\Omega, \quad \z\in\widetilde\RS.
\end{equation}
This vector-function has continuous traces on each side of the $\mathbf{a}$- and $\mathbf{b}$-cycles that satisfy
\begin{equation}
\label{Omegapm}
\vec\Omega^{+}-\vec\Omega^- = \left\{
\begin{array}{rl}
-\mathcal{B}_\Omega\vec e_k & \mbox{on} \quad \mathbf{a}_k, \smallskip \\
\vec e_k & \mbox{on} \quad \mathbf{b}_k,
\end{array}
\right. \quad k\in\{1,\ldots,g\},
\end{equation}
by \eqref{periodsa} and \eqref{periodsb}. It readily follows from \eqref{Omegapm} that each $\Omega_k$ is, in fact, holomorphic in $\widehat\RS\setminus\mathbf{b}_k$. It is known that
\begin{equation}
\label{jac0}
\theta\left(\vec u\right)=0 \quad \Leftrightarrow \quad \vec u\equiv \sum_{j=1}^{g-1}\vec\Omega\left(\tr_{j}\right) + \vec K \quad \left(\mdp d\vec\Omega\right)
\end{equation}
for some divisor $\sum_{j=1}^{g-1}\tr_j$, where $\vec K$ is the vector of Riemann constants defined by $(\vec K)_j:=((\mathcal{B}_\Omega)_{jj}-1)/2-\sum_{k\neq j}\oint_{\mathbf{a}_k}\Omega_j^-d\Omega_k$, $j\in\{1,\ldots,g\}$.

Let $\sum_{j=1}^g\tr_j$ and $\sum_{j=1}^g\z_j$ be non-special divisors. Set
\begin{equation}
\label{theta-old}
\Theta\left(\z;\sum\tr_j,\sum\z_j\right) := \frac{\theta\left(\vec\Omega(\z) -\sum_{j=1}^g\vec\Omega\left(\tr_j\right)-\vec K\right)}{\theta\left(\vec\Omega(\z) - \sum_{j=1}^g\vec\Omega\left(\z_j\right)-\vec K\right)}.
\end{equation}
It follows from \eqref{Omegapm} that this is a meromorphic and single-valued function in $\widehat\RS$ (multiplicatively multi-valued in $\RS$). Furthermore, by \eqref{jac0} it has a pole at each $\z_j$ and a zero at each $\tr_j$ (coincidental points mean increased multiplicity),  and by \eqref{periodicity} it satisfies
\begin{equation}
\label{theta-combined-jump}
\Theta^+\left(\z;\sum\tr_j,\sum\z_j\right) = \Theta^-\left(\z;\sum\tr_j,\sum\z_j\right)\exp\left\{2\pi i\sum_{j=1}^g\big(\Omega_k(\z_j)-\Omega_k(\tr_j)\big)\right\}
\end{equation}
for $\z\in\mathbf{a}_k\setminus\left\{\cup\z_j\bigcup\cup\tr_j\right\}$. 

If the divisor $\sum_{j=1}^g\tr_j$ (resp. $\sum_{j=1}^g\z_j$) in \eqref{theta-old} is special, then the numerator (resp. denominator) is identically zero by \eqref{jac0}. This difficulty can be circumvented in the following way. Let $\w_1,\w_2\in\RS\setminus\left\{\w\right\}$ for some $\w\in\RS$. Set
\begin{equation}
\label{theta}
\Theta\left(\z;\w_1,\w_2\right) := \frac{\theta\left(\vec\Omega(\z) - \vec\Omega(\w_1) - (g-1)\vec\Omega\left(\w^*\right)-\vec K\right)}{\theta\left(\vec\Omega(\z) - \vec\Omega(\w_2) - (g-1)\vec\Omega\left(\w^*\right)-\vec K\right)}.
\end{equation}
Since the divisors $\w_j+(g-1)\w^*$ are non-special, $\Theta\left(\cdot;\w_1,\w_2\right)$ is a multiplicatively multi-valued meromorphic function on $\RS$ with a simple zero at $\w_1$, a simple pole at $\w_2$, and otherwise non-vanishing and finite. Moreover, it is meromorphic and single-valued in $\widehat\RS$ and
\begin{equation}
\label{theta-jumps}
\Theta^+\left(\z;\w_1,\w_2\right)  = \Theta^-\left(\z;\w_1,\w_2\right) \exp\left\{2\pi i\big(\Omega_k(\w_2)-\Omega_k(\w_1)\big)\right\}
\end{equation}
for $\z\in\mathbf{a}_k\setminus\{\w_1,\w_2\}$. Observe that the jump does not depend on $\vec\Omega(\w^*)$. Hence, analytic continuation argument and \eqref{theta-jumps} immediately show that $\Theta\left(\cdot;\w_1,\w_2\right)$ can be defined (up to a multiplicative constant) using any divisor $\sum_{j=1}^{g-1}\tr_{j}$ as long as $\w_i+\sum_{j=1}^{g-1}\tr_{j}$ is non-special and that
\[
\frac{\Theta\left(\z;\sum\tr_j,\sum\z_j\right)}{\Theta\left(\w;\sum\tr_j,\sum\z_j\right)} = \prod_{j=1}^g\frac{\Theta\left(\z;\tr_j,\z_j\right)}{\Theta\left(\w;\tr_j,\z_j\right)}
\]
for any $\w$ fixed and satisfying $\{\w\}\cap\left\{\cup\z_j\bigcup\cup\tr_j\right\}=\varnothing$. Let us point out that even though the construction \eqref{theta-old} is simpler, it requires only non-special divisors, while this restriction is not needed for \eqref{theta}.

\section{Boundary Value Problems on $L$}
\label{s:bvp}

This is a technical section needed to prove Proposition~\ref{prop:Sn}. The results of this sections will be applied to logarithm of $\rho\in\W$, which is holomorphic across each arc comprising $\Delta$. However, here we treat more general H\"older continuous densities as this generalization comes at no cost (analyticity of the weight $\rho$ will be essential for the Riemann-Hilbert analysis carried in Sections~\ref{s:5}--\ref{s:6}). In what follows, we describe properties of
\begin{equation}
\label{Psi}
\Psi(\z):=\frac{1}{4\pi i}\oint_L\psi d\Omega_{\z,\z^*}, \quad \z\in\widehat\RS\setminus L,
\end{equation}
for a given function $\psi$ on $L$. Before we proceed, let us derive an explicit expression for $d\Omega_{\z,\z^*}$. To this end, set
\begin{equation}
\label{Hj}
H_k(\z) := \oint_{\mathbf{a}_k}\frac{d\widetilde\Omega}{w-z} = 2\int_{\Delta^a_k}\frac{1}{w-z}\frac{dw}{\hbar(w)}, \quad k\in\{1,\ldots,g\}.
\end{equation}
Clearly, each $H_k$ is a holomorphic function on $\RS\setminus\mathbf{a}_k$ that satisfies
\begin{equation}
\label{jumpHj}
\hbar H_k^+ - \hbar H_k^- = 4\pi i \quad \mbox{on} \quad \mathbf{a}_k
\end{equation}
due to Sokhotski-Plemelj formulae \cite{Gakhov} as apparent from the second integral representation in \eqref{Hj}. Then, using functions $H_k$, we can write
\begin{equation}
\label{specialthirdkind}
d\Omega_{\z,\z^*}(\w) = \frac{\hbar(\z)}{w-z}d\widetilde\Omega(\w) - \sum_{k=1}^g(\hbar H_k)(\z)d\Omega_k(\w).
\end{equation}

\subsection{H\"older Continuous Densities}
\label{ss:sf}

Let $\psi$ be a function on $L\setminus E$ with H\"older continuous extension to each cycle $L_k$ and $\Psi$ be given by \eqref{Psi}. The differential $d\Omega_{\z,\z^*}$ plays a role of the Cauchy kernel on $\RS$ with a discontinuity. Indeed, it follows from \eqref{specialthirdkind} that
\begin{equation}
\label{PsiAgain}
\Psi(\z) = \frac{\hbar(\z)}{4\pi i}\sum_{j}\oint_{L_j}\frac{\psi(\tr)}{t-z}\frac{dt}{\hbar(\tr)} - \frac{\hbar(\z)}{4\pi i}\sum_{k=1}^gH_k(z)\oint_L\psi d\Omega_k =: \sum_j\Psi_{L_j}(\z)-\sum_{k=1}^g\Psi_{\mathbf{a}_k}(\z).
\end{equation}

Each function $\Psi_{L_j}$ is holomorphic in $\RS\setminus\left(L_j\cup\{\infty^{(0)},\infty^{(1)}\}\right)$ with H\"older continuous traces on $L_j$ that satisfy
\begin{equation}
\label{PsiPlemelj0}
\Psi_{L_j}^+ - \Psi_{L_j}^- = \psi.
\end{equation}
Clearly, $\Psi_{L_j}(e)=0$ for $e\in  E\setminus L_j$. Moreover, it holds by \eqref{PsiPlemelj0} and the identity $\Psi_{L_j}(\z)+\Psi_{L_j}(\z^*)\equiv0$ that
\begin{equation}
\label{Psi0E}
\Psi_{L_j}(z^{(k)}) \to \frac{(-1)^k}2\psi_{|L_j}(e) \quad \mbox{as} \quad z\to e
\end{equation}
for univalent ends $e\in E\cap L_j$. To describe the behavior of $\Psi_{L_j}$ near trivalent ends, recall that $\Delta$ splits any disk centered at $e$ of small enough radius into three sectors. Two of these sectors contain part of $\Delta_j$ in their boundary and one sector does not. Recall further that $\hbar(z)$ changes sign after crossing each of the subarcs of $\Delta$. Thus, it holds for trivalent ends $e\in E\cap L_j$ that
\begin{equation}
\label{Psi0Ea}
\Psi_{L_j}(z^{(k)}) \to \pm\frac{(-1)^k}2\psi_{|L_j}(e) \quad \mbox{as} \quad z\to e,
\end{equation}
where $+$ sign corresponds to the approach within the sectors partially bounded by $\Delta_j$ and the $-$ sign corresponds to the approach within the sector which does not contain $\Delta_j$ as part of its boundary.

Similarly, each $\Psi_{\mathbf{a}_k}$ is a holomorphic function in $\RS\setminus\left(\mathbf{a}_k\cup\{\infty^{(0)},\infty^{(1)}\}\right)$ with H\"older continuous traces on $\mathbf{a}_k$ that satisfy
\begin{equation}
\label{PsiPlemeljk}
\Psi_{\mathbf{a}_k}^+ - \Psi_{\mathbf{a}_k}^- = \oint_L\psi d\Omega_k
\end{equation}
by \eqref{jumpHj}. Analogously to \eqref{Psi0E}, one can verify that $\Psi_{\mathbf{a}_k}(e)=0$ for $e\in E\setminus\mathbf{a}_k$ and
\begin{equation}
\label{PsikE}
\Psi_{\mathbf{a}_k}(z^{(k)}) \to \frac{(-1)^k}2\oint_L\psi d\Omega_k \quad \mbox{as} \quad z\to e\in \mathbf{a}_k\cap E.
\end{equation}

Combining all the above, we get that $\Psi$ is a holomorphic function in $\widehat\RS\setminus L$ including at $\infty^{(0)}$ and $\infty^{(1)}$ where it holds that
\begin{equation}
\label{Psivalueatinfty}
\Psi(\infty^{(0)}) = -\Psi(\infty^{(1)}) = \frac12\vec\tau^T\oint_L\psi d\vec\Omega - \frac{1}{4\pi i}\oint_L\psi dG
\end{equation}
by \eqref{dG1}. Moreover, it has H\"older continuous traces on both sides of $(L\cup\bigcup \mathbf{a}_k)\setminus E$ that satisfy
\begin{equation}
\label{PsiPlemelj}
\Psi^+ - \Psi^- = \left\{
\begin{array}{ll}
\psi, & \mbox{on} \quad L\setminus E, \smallskip \\
-\oint_L\psi d\Omega_k, & \mbox{on} \quad \mathbf{a}_k\setminus E,
\end{array}
\right.
\end{equation}
according to \eqref{PsiPlemelj0} and \eqref{PsiPlemeljk}.  Finally, the behavior at $e\in E$ can be deduced from \eqref{Psi0E}, \eqref{Psi0Ea}, and \eqref{PsikE}.

\subsection{Logarithmic Discontinuities}
Assume now that $\psi$ has logarithmic singularities at $e\in E$, which, obviously, violates the condition of global H\"older continuity of $\psi$ on the cycles $L_k$. However, global H\"older continuity is not necessary for $\Psi$ to be well-defined. In fact, it is known that the traces $\Psi^\pm$ are H\"older continuous at $\tr\in L$ as long as $\psi$ is locally H\"older continuous around this point. Thus, we only need to describe the behavior of $\Psi$ near those $e\in E$ where $\psi$ has a singularity.

Let $a$ be a fixed univalent end of $\Delta$ and $\Delta_j$ be the arc incident with $a$. Further, let $\psi$ be a fixed determination of $\alpha\log(\cdot-a)$ holomorphic around each arc $\Delta_k$ (except at $a$ when $k=j$), where $\alpha$ is a constant. As before, define $\Psi$ by \eqref{Psi}. It clearly follows from \eqref{PsiAgain} that we only need to describe the behavior of $\Psi_{L_j}$ around $a$ as the behavior of the other terms is unchanged. To this end, it can be readily verified that
\begin{equation}
\label{Psia}
\Psi_{L_j}(\z) =\frac{\hbar(\z)}{2\pi i}\int_{\Delta_j}\frac{\psi(t)}{t-z}\frac{dt}{\hbar^+(t)}.
\end{equation}
Denote by $U_{a,\delta}$ a ball centered at $a$ of radius $\delta$ chosen small enough that the intersection $\Delta_j\cap U_{a,\delta}$ is an analytic arc. Denote also by $U_{a,\delta}^\pm$ the maximal open subset of $U_{a,\delta}\setminus\Delta_j$ in which $\psi$ is holomorphic and $\Delta_j^\pm\subset \partial U_{a,\delta}^\pm$ if $\Delta_j$ is orienter towards $a$  and $\Delta_j^\pm\subset \partial U_{a,\delta}^\mp$ if $\Delta_j$ is oriented away from $a$. Set
\begin{equation}
\label{argez}
\arg(a-z) := \arg(z-a) \pm \pi \quad \mbox{in} \quad U_{a,\delta}^\pm.
\end{equation}
It can be readily verified that thus defined $\log(a-z):=\log|a-z|+i\arg(a-z)$ is holomorphic in $U_{a,\delta}\setminus\Delta_j$. Then, arguing as in \cite[Equations (8.34)--(8.35)]{Gakhov}, that is, by identifying a function with the same jump across $\Delta_j$ as the one of integral in \eqref{Psia}, we get that 
\begin{equation}
\label{smth}
\frac{1}{2\pi i}\int_{\Delta_j}\frac{\psi(t)}{t-z}\frac{dt}{\hbar^+(t)} = \frac\alpha2\frac{\log(a-z)}{\hbar(z)} + \{\mbox{ terms that are holomorphic at $a$ }\}
\end{equation}
in $U_{a,\delta}\setminus\Delta_j$. Multiplying both sides of \eqref{smth} by $\hbar$, we get that \eqref{Psi0E} is replaced by
\begin{equation}
\label{Psi0a}
\Psi_{L_j}(z^{(k)}) -(-1)^k \frac\alpha2\log(a-z) \to 0 \quad \mbox{as} \quad z\to a.
\end{equation}

Let now $b$ be a trivalent end of $\Delta$ and $\psi$ be a fixed determination of $\alpha\log(\cdot-b)$ analytic across $\Delta$, where $\alpha$ as before is a constant. Further, let $\Delta_{b,j}$ be the arcs incident with $b$. Fix $l\in\{1,2,3\}$ and let $\log(b-z)$ be defined by \eqref{argez} with respect to $\Delta_{b,l}$. Then \eqref{smth} still takes place within the sectors delimited by $\Delta_{b,l}\cup\Delta_{b,l+1}$ and $\Delta_{b,l}\cup\Delta_{b,l-1}$, where $l\pm1$ is understood cyclicly within $\{1,2,3\}$. However, within the sector delimited by $\Delta_{b,l+1}\cup\Delta_{b,l-1}$, the right-hand side of \eqref{smth} has to be multiplied by $-1$ to ensure analyticity across $\Delta_{b,l+1}\cup\Delta_{b,l-1}$. Then, multiplying both sides of \eqref{smth} by $\hbar$, we get that
\begin{equation}
\label{Psinearb}
\Psi_{L_{b,l}}(z^{(k)}) \mp(-1)^k \frac\alpha2\log(b-z) \to 0 \quad \mbox{as} \quad z\to b,
\end{equation}
where $-$ sign corresponds to the approach within the sectors delimited by $\Delta_{b,l}\cup\Delta_{b,l+1}$ and $\Delta_{b,l}\cup\Delta_{b,l-1}$, and the $-$ sign corresponds to the approach within the sector delimited by the pair $\Delta_{b,l+1}\cup\Delta_{b,l-1}$. Hence, the behavior of $\Psi$ near $b$ is completely determined by the behavior of the sum $\sum_{l=1}^3\Psi_{L_{b,l}}$.

\subsection{Auxiliary Functions}
\label{ss:szego}
For an arbitrary $\vec x\in\R^g$ set $\psi_{\vec x}$ to be a function on $L$ such that
\[
\psi_{\vec x} : =\left\{
\begin{array}{ll}
2\pi i\big(\vec x\big)_k & \mbox{on} \quad \mathbf{b}_k, \smallskip \\
0 & \mbox{otherwise}.
\end{array}
\right.
\]
Define further
\begin{equation}
\label{defSx}
S_{\vec x}(\z) : = \exp\left\{\frac{1}{4\pi i}\oint_L\psi_{\vec x} d\Omega_{\z,\z^*}\right\}, \quad \z\in\widehat\RS\setminus L.
\end{equation}
Then $S_{\vec x}$ is a holomorphic and non-vanishing function in $\widetilde\RS$, with H\"older continuous non-vanishing traces on both sides of each $\mathbf{a}$- and $\mathbf{b}$-cycle that satisfy
\begin{equation}
\label{jumpSx}
S_{\vec x}^+ = S_{\vec x}^- \left\{
\begin{array}{lr}
\exp\big\{2\pi i \big(\vec x\big)_k\big\} & \mbox{on} \quad \mathbf{b}_k, \bigskip \\
\exp\big\{-2\pi i\left(\mathcal{B}_\Omega\vec x\right)_k\big\} & \mbox{on} \quad \mathbf{a}_k,
\end{array}
\right.
\end{equation}
for $k\in\{1,\ldots,g\}$ by \eqref{PsiPlemelj} and \eqref{periodsb}. Observe also that
\begin{equation}
\label{proddecomp}
S_{\vec x} = \prod_{k=1}^g S_{\vec e_k}^{(\vec x)_k},
\end{equation}
where, as before, $\left\{\vec e_k\right\}_{k=1}^g$ is the standard basis in $\R^g$.

Now, let $\rho\in\W$ and $\log\rho$ be a fixed branch holomorphic across each $\Delta_k$ in $\Delta$. Define
\begin{equation}
\label{defSrho}
S_{\log\rho}(\z) : = \exp\left\{\frac{1}{4\pi i}\oint_L\log(\rho\circ\pi) d\Omega_{\z,\z^*}\right\}, \quad \z\in\widehat\RS\setminus L.
\end{equation}
Then, as in the case of \eqref{defSx}, it follows from \eqref{PsiPlemelj} that $S_{\log\rho}$ is a holomorphic function in $\widehat\RS\setminus L$ and
\begin{equation}
\label{jumpSrho}
S_{\log\rho}^+ = S_{\log\rho}^- \left\{
\begin{array}{ll}
\rho\circ\pi, & \mbox{on} \quad L\setminus E, \smallskip \\
\exp\big\{-\oint_L\log(\rho\circ\pi)d\Omega_k\big\}, & \mbox{on} \quad \mathbf{a}_k\setminus E, \quad k\in\{1,\ldots,g\}.
\end{array}
\right.
\end{equation}
If $a$ is a univalent end of $\Delta$ and $\Delta_j$ is the arc incident with $a$, then $\rho_{|\Delta_j}(z)=w_j(z)(z-a)^{\alpha_a}$, where $w_j$ is holomorphic and non-vanishing in some neighborhood of $\Delta_j$ and $(z-a)^{\alpha_a}$ is a branch holomorphic around $\Delta_j\setminus\{a\}$.  Let $(a-z)^{\alpha_a}$ be the branch defined by \eqref{argez}. Then $(a-z)^{\alpha_a}=(z-a)^{\alpha_a}\exp\{\pm\alpha_a\pi i\}$ in $U_{a,\delta}^\pm$, where the latter were defined right before \eqref{argez}. Thus, it holds by \eqref{Psi0E} and \eqref{Psi0a} that
\begin{equation}
\label{SrhoUnivalent}
S_{\log\rho}^2(z)/\rho(z) \to \exp\left\{\pm\alpha_a\pi i-\sum_{k=1}^g\varepsilon_{\mathbf{a}_k}(a,z)\oint_L\log(\rho\circ\pi)d\Omega_k\right\} \quad \mbox{as} \quad U_{a,\delta}^\pm\ni z\to a,
\end{equation}
where $\varepsilon_{\mathbf{a}_k}(a,\z)\equiv0$  if $a\not\in\mathbf{a}_k$, and $\varepsilon_{\mathbf{a}_k}(a,\z)\equiv\pm1$ if $a\in\mathbf{a}_k$ and $\z\to a\in\mathbf{a}_k^\pm$. If $b$ is a trivalent end of $\Delta$, let $\Delta_{b,j}$, $j\in\{1,2,3\}$, be the arcs incident with $b$. Further, let $S_{b,j}$ be a sector delimited $\Delta_{b,j\pm1}$ within a disk centered at $b$ of small enough radius, where $j\pm1$ is understood cyclicly within $\{1,2,3\}$. Then it follows from \eqref{Psi0Ea} that
\begin{equation}
\label{SrhoTrivalent}
S_{\log\rho}^2(z) \to \frac{\rho_{|\Delta_{b,j-1}}(b)\rho_{|\Delta_{b,j+1}}(b)}{\rho_{|\Delta_{b,j}}(b)} \quad \mbox{as} \quad S_{b,j}\ni z\to b.
\end{equation}
Finally, we can deduce the behavior of $S_{\log\rho}$ at $\infty^{(0)}$ from \eqref{Psivalueatinfty} in a straightforward fashion.

Now, let $\log h^+$ be a fixed branch continuous on $\Delta\setminus E$. Define $S_{\log h}$ as in \eqref{defSrho} with $\log\rho$ replaced by $\log h^+$. Then $S_{\log h}$ enjoys the same properties $S_{\log\rho}$ does except for \eqref{SrhoUnivalent} and \eqref{SrhoTrivalent} as $h$ is not holomorphic across $\Delta\setminus E$ unlike $\rho$. In particular, it holds that
\begin{equation}
\label{jumpSh}
S_{\log h}^+ = S_{\log h}^- \left\{
\begin{array}{ll}
h^+\circ\pi, & \mbox{on} \quad L\setminus E, \smallskip \\
\exp\big\{-\oint_L\log(h^+\circ\pi)d\Omega_k\big\}, & \mbox{on} \quad \mathbf{a}_k\setminus E, \quad k\in\{1,\ldots,g\}.
\end{array}
\right.
\end{equation}
Moreover, it can be easily verified that \eqref{SrhoUnivalent} gets replaced by
\begin{equation}
\label{ShUnivalent}
S_{\log h}^2(z)/h(z) \to \exp\left\{\mp\frac{\pi i}2-\sum_{k=1}^g\varepsilon_{\mathbf{a}_k}(a,z)\oint_L\log(h^+\circ\pi)d\Omega_k\right\} \quad \mbox{as} \quad z\to a
\end{equation}
with $-$ sign used when $\Delta_j$ is oriented towards $a$ (one needs to take $U_{a,\delta}^+=U_{a,\delta}\setminus\Delta_j$ and $U_{a,\delta}^-=\varnothing$ in \eqref{Psi0a}) and  $+$ sign used when $\Delta_j$ is oriented away from $a$ ($U_{a,\delta}^-=U_{a,\delta}\setminus\Delta_j$ and $U_{a,\delta}^+=\varnothing$). Similarly, \eqref{SrhoTrivalent} is replaced by
\begin{equation}
\label{ShTrivalent}
S^2_{\log h}/h\to\exp\big\{\pm\pi i/2\big\} \quad \mbox{as} \quad z\to b,
\end{equation}
where the sign $+$ corresponds to the case when the arcs incident with $b$ are oriented towards $b$ and the sign $-$ corresponds to the other case (this conclusion is deduced from \eqref{Psi0Ea} and \eqref{Psinearb} applied to each arc incident with $b$).

\section{Szeg\H{o} Functions}
\label{s:s}

In this section we prove Propositions~\ref{prop:ni},~\ref{prop:sequence}, and~\ref{prop:Sn}.

\subsection{Proof of Proposition~\ref{prop:ni}}
\label{ss:ni}
It follows from the discussion in Section~\ref{ss:jip} that a solution of \eqref{main-jip} is either unique or special; and in the latter case any pair of involution-symmetric points can be replaced by any other such pair. According to the convention adopted in Definition~\ref{def:2}, we denote by $\sum_{j=1}^g\tr_{n,j}$ the divisor that either uniquely solves \eqref{main-jip} or solves \eqref{main-jip} and all the involution-symmetric pairs are taken to be $\infty^{(0)}+\infty^{(1)}$.

Let $\sum_{j=1}^gb_j^{(1)}$ be as in \eqref{bj1} and $\vec\omega,\vec\tau,\vec c_\rho$ be as in \eqref{vectors}. Notice that by \eqref{lateaddition} it holds that
\begin{equation}
\label{awesome}
\sum_{j=1}^g\int_{b_j^{(1)}}^{\tr_{n,j}}d\vec\Omega + i\int_{\infty^{(1)}}^{\infty^{(0)}}d\vec\Omega  \, \equiv \, \vec c_\rho + (n+i)\big(\vec\omega+\mathcal{B}_\Omega\vec\tau\big) \quad \left(\mdp d\vec\Omega\right).
\end{equation}
Then if
\[
\sum_{j=1}^g\tr_{n,j} = \sum_{j=1}^{g-l}\tr_j + k\infty^{(0)} + (l-k)\infty^{(1)},
\]
where $\big\{\tr_j\big\}_{j=1}^{g-l}\subset\RS\setminus\big\{\infty^{(0)},\infty^{(1)}\big\}$, it holds by \eqref{awesome} that
\[
\sum_{j=1}^g\tr_{n+i,j} = \sum_{j=1}^{g-l}\tr_j + (k+i)\infty^{(0)} + (l-k-i)\infty^{(1)}
\]
for each $i\in\{-k,\ldots,l-k\}$. The uniqueness of the solutions for $i=-k,l-k$ immediately follows from the fact that $\sum_{j=1}^g\tr_{n+i,j}$ is not special for these indices.

Now, let $\sum_{j=1}^g\tr_{n,j}$ be the unique solution of \eqref{main-jip} that does not contain $\infty^{(k)}$, $k\in\{0,1\}$. If $\sum_{j=1}^g\tr_{n-(-1)^k,j}$ were not the unique solution, it would contain at least one pair $\infty^{(1)}+\infty^{(0)}$ and therefore $\sum_{j=1}^g\tr_{n,j}$ would contain $\infty^{(k)}$ by the first part of the proof. Thus, $\sum_{j=1}^g\tr_{n-(-1)^k,j}$ solves  \eqref{main-jip} uniquely, and it only remains to show that
\begin{equation}
\label{disjointness}
\big\{\tr_{n,j}\big\}_{j=1}^g \cap \big\{\tr_{n-(-1)^k,j}\big\}_{j=1}^g = \varnothing.
\end{equation}
Assume the contrary. For definiteness, let $g^\prime<g$ be the number of distinct points in the divisors $\sum_{j=1}^g\tr_{n,j}$ and $\sum_{j=1}^g\tr_{n-(-1)^k,j}$, and label the common points by indices ranging from $g^\prime+1$ to $g$. If \eqref{disjointness} were false, then it would follow from \eqref{main-jip} and \eqref{lateaddition} that
\[
\sum_{j=1}^{g^\prime}\int_{\tr_{n-(-1)^k,j}}^{\tr_{n,j}}d\vec\Omega -(-1)^k \int_{\infty^{(1)}}^{\infty^{(0)}}d\vec\Omega  \, \equiv \, \vec 0 \quad \left(\mdp d\vec\Omega\right).
\]
That is, the divisor
\[
\sum_{j=1}^{g^\prime} \tr_{n,j} - \sum_{j=1}^{g^\prime} \tr_{n-(-1)^k,j} -(-1)^k \infty^{(0)} +(-1)^k \infty^{(1)}
\]
would be principal. However, since $g^\prime+1\leq g$, such divisors come solely from rational functions over $\overline\C$ and their zeros as well as poles appear in involution-symmetric pairs. Hence, the divisor $\sum_{j=1}^{g^\prime} \tr_{n,j}$ would contain an involution-symmetric pair or $\infty^{(k)}$. As both conclusions are impossible, \eqref{disjointness} indeed takes place. This completes the proof of Proposition~\ref{prop:ni}.

\subsection{Proof of Proposition~\ref{prop:sequence}}
\label{ss:sequence}

Let $\N^{\prime\prime}\subseteq\N^\prime$ be such a subsequence that the divisors $\sum_{i=1}^g\tr_{n+j,i}$ converge to a divisor $\sum_{i=1}^g\w_i$ as $\N^{\prime\prime}\ni n\to\infty$ for a fixed index $j\in\{-l_0-k,\ldots,l_1+k\}$. Then the continuity of $\vec\Omega$ implies that
\[
\lim_{\N^{\prime\prime}\ni n\to\infty} \sum_{i=1}^g\int_{b_i^{(1)}}^{\tr_{n+j,i}}d\vec\Omega = \sum_{i=1}^g\int_{b_i^{(1)}}^{\w_i}d\vec\Omega = \sum_{i=1}^g\vec\Omega\left(\w_i\right)- \vec v_b,
\]
where $\vec v_b:=\sum_{i=1}^g\vec\Omega\big(b_i^{(1)}\big)$ (recall also the convention that all the paths of integration belong to $\widetilde\RS$ and therefore the right-hand side of the equality above does not depend on the labeling of $\sum\tr_{n+j,i}$ and $\sum b_i^{(1)}$). Hence, it holds that
\[
\lim_{\N^{\prime\prime}\ni n\to\infty} \left(\vec c_\rho+(n+j)\big(\vec\omega+\mathcal{B}_\Omega\vec\tau\big)\right) \equiv \sum_{i=1}^g\vec\Omega\left(\w_i\right) - \vec v_b,
\]
where, from now on, all the equivalences are understood $\mdp d\vec\Omega$. Set $l:=l_0-l_1$. Assume first that $l\geq0$. Then, analogously to the previous computation, we have that
\[
\lim_{\N^\prime\ni n\to\infty} \left(\vec c_\rho+n\big(\vec\omega+\mathcal{B}_\Omega\vec\tau\big)\right) \equiv \sum_{i=1}^{g-2k-l_0-l_1}\vec\Omega\left(\tr_i\right) + l\vec\Omega\left(\infty^{(0)}\right) - \vec v_b
\]
since $\vec\Omega\left(z^{(0)}\right)=-\vec\Omega\left(z^{(1)}\right)$. 

In what follows, we assume that $l+2j\geq0$, otherwise, if $l+2j<0$, each occurrence of $\infty^{(0)}$ and $l+2j$ needs to be replaced by $\infty^{(1)}$ and $-(l+2j)$, respectively. Then \eqref{lateaddition} and the just mentioned anti-symmetry of $\vec\Omega$ yield that
\[
\lim_{\N^\prime\ni n\to\infty} \left(\vec c_\rho+(n+j)\big(\vec\omega+\mathcal{B}_\Omega\vec\tau\big)\right) \equiv \sum_{i=1}^{g-2k-l_0-l_1}\vec\Omega\left(\tr_i\right) + (l+2j)\vec\Omega\left(\infty^{(0)}\right) - \vec v_b.
\]
Hence, it is true that
\[
\sum_{i=1}^g\vec\Omega\left(\w_i\right) \equiv \sum_{i=1}^{g-2k-l_0-l_1}\vec\Omega\left(\tr_i\right) + (l+2j)\vec\Omega\left(\infty^{(0)}\right).
\]
Therefore, for any collection $\{u_i\}_{i=1}^{l_1+k-j}\subset\overline\C$ it holds by Abel's theorem that the divisor
\[
\sum_{i=1}^{g-2k-l_0-l_1}\tr_i+\sum_{i=1}^{l_1+k-j}\left(u_i^{(0)}+u_i^{(1)}\right)+(l+2j)\infty^{(0)} - \sum_{i=1}^g\w_i
\]
is principal ($l_1+k-j$ needs to be replaced by $l_0+k+j$ when $l+2j<0$). As the integral part of this divisor has at most $g$ elements, the divisor should be involution-symmetric. However, if $\sum_{i=1}^{g-2k-l_0-l_1}\tr_j+(l+2j)\infty^{(0)}$ is non-void, it is non-special, and therefore $\sum_{i=1}^g\w_i$ is equal to $\sum_{i=1}^{g-2k-l_0-l_1}\tr_i+\sum_{i=1}^{l_1+k-j}\big(u_i^{(0)}+u_i^{(1)}\big)+(l+2j)\infty^{(0)}$; if it is void, $\sum_{i=1}^g\w_i$ is an arbitrary involution-symmetric divisor. In any case, this is exactly what is claimed by the proposition. Clearly, the case $l<0$ can be treated similarly.

To prove the last assertion of the proposition, observe that the divisors $\sum_{j=1}^g\tr_j$ and $\sum_{j=1}^g\w_j$ are connected by the relation
\[
\sum_{j=1}^g\vec\Omega(\tr_j) -(-1)^k\int_{\infty^{(1)}}^{\infty^{(0)}}d\vec\Omega \equiv \sum_{j=1}^g\vec\Omega(\w_j).
\]
Hence, by Abel's theorem the divisor $\sum_{j=1}^g\tr_j-\sum_{j=1}^g\w_j-(-1)^k\infty^{(0)}+(-1)^k\infty^{(1)}$ is principal. Since $\sum_{j=1}^g\tr_j+\infty^{(1-k)}$ is non-special, the claim follows as in the end of the proof of Proposition~\ref{prop:ni}.

\subsection{Proof of Proposition~\ref{prop:Sn}}
\label{ss:Sn}

Any vector $\vec u\in\C^g$ can be uniquely and continuously written as $\vec x+\mathcal{B}_\Omega\vec y$, $\vec x,\vec y\in\R^g$, since the imaginary part of $\mathcal{B}_\Omega$ is positive definite. Hence, we can define
\begin{equation}
\label{xnyndef}
\vec x_n+\mathcal{B}_\Omega\vec y_n\; := \; \sum_{j=1}^g\int_{b_j^{(1)}}^{\tr_{n,j}}d\vec\Omega = \sum_{j=1}^g\left(\vec\Omega\left(\tr_{n,j}\right)-\vec\Omega\left(b_j^{(1)}\right)\right).
\end{equation}
As the image of the closure of $\widetilde\RS$ under $\vec\Omega$ is bounded in $\C^g$, it holds that
\begin{equation}
\label{xnynbound}
|\vec x_n|,|\vec y_n| \leq \const
\end{equation}
independently of $n$, where $|\vec c|^2:=\sum_{k=1}^g|(\vec c)_k|^2$. Set further
\[
\vec x_\rho+\mathcal{B}_\Omega\vec y_\rho := \vec c_\rho.
\]
Then it follows from the very choice of $\sum_{j=1}^g\tr_{n,j}$, see \eqref{main-jip}, that there exist unique vectors $\vec j_n,\vec m_n\in\Z^g$ such that
\begin{equation}
\label{connectitall}
\vec x_\rho + n\vec\omega = \vec x_n + \vec j_n \quad \mbox{and} \quad \vec y_\rho + n\vec \tau = \vec y_n + \vec m_n.
\end{equation}
Therefore, we immediately deduce from \eqref{xnynbound} that
\begin{equation}
\label{mn}
|\vec m_n-\vec m_{n-1}|,~|(2n-1)\vec\tau-\vec m_n-\vec m_{n-1}|,~|n\vec\tau-\vec m_n| \leq \const
\end{equation}
independently of $n$. 

Let now $S_{\vec\tau}$ and $S_{\vec m_n}$ be defined by  \eqref{defSx}. Then it is an easy consequence of \eqref{mn} and \eqref{proddecomp} that
\begin{equation}
\label{ratioSx}
0<\const\leq\left|S_{\vec m_n}/S_{\vec m_{n-1}}\right|,~\left|S_{\vec\tau}^{2n-1}/S_{\vec m_n}S_{\vec m_{n-1}}\right|,~\left|S_{\vec m_n}/S_{\vec\tau}^n\right|\leq \const<\infty
\end{equation}
uniformly in $\widetilde\RS$. Notice also that
\begin{equation}
\label{buildblock1}
\left(S_{\vec m_n}/S_{\vec\tau}^n\right)^+ = \left(S_{\vec m_n}/S_{\vec\tau}^n\right)^-
\left\{
\begin{array}{lr}
\exp\big\{-2\pi i n\tau_k\big\} & \mbox{on} \quad \mathbf{b}_k, \bigskip \\
\exp\big\{-2\pi i\left(\mathcal{B}_\Omega\big(\vec m_n-n\vec\tau\big) \right)_k\big\} & \mbox{on} \quad \mathbf{a}_k.
\end{array}
\right.
\end{equation}

Using definitions \eqref{theta-old} and \eqref{theta}, set
\[
\Theta_n(\z) := \Theta\left(\z;\sum\tr_{n,j},\sum b_j^{(1)}\right) \quad \mbox{and} \quad \Theta_n(\z) := \prod_{j=1}^g\Theta\left(\z;\tr_{n,j},b_j^{(1)}\right),
\]
where the first formula is used for non-special divisors $\sum\tr_{n,j}$ and the second one otherwise. Then $\Theta_n$ is a meromorphic function in $\widehat\RS$ with poles at $b_j^{(1)}$, zeros at $\tr_{n,j}$ (as usual, coincidental points mean increased multiplicity), and otherwise non-vanishing and finite. It also follows from \eqref{theta-jumps}, \eqref{theta-combined-jump}, and \eqref{xnyndef} that
\begin{equation}
\label{buildblock2}
\Theta_n^+ = \Theta_n^-\exp\left\{-2\pi i\left(\vec x_n+\mathcal{B}_\Omega\vec y_n\right)_k\right\} \quad \mbox{on} \quad \mathbf{a}_k.
\end{equation}

Finally, let $S_{\log h}$ and $S_{\log\rho}$ be defined as in Section~\ref{ss:szego} with the branch of the difference $\log\rho-\log h^+$ chosen to match the one used in \eqref{vectors} to define $\vec c_\rho$. Set
\begin{equation}
\label{defofSn}
S_n := \frac{S_{\log h}}{S_{\log\rho}}\frac{S_{\vec m_n}}{S_{\vec\tau}^n}\Theta_n.
\end{equation}
Then $S_n$ is a meromorphic function in $\widehat\RS\setminus L$ and $S_n\map^n$ is meromorphic in $\RS\setminus L$ by \eqref{jumpPhiperoids}, \eqref{jumpSrho}, \eqref{jumpSh}, \eqref{buildblock1}, \eqref{buildblock2}, and \eqref{connectitall}. Clearly, the same equations also yield that $S_n\map^n$ satisfies \eqref{jumpSn}. Finally,  \eqref{SnE} follows from \eqref{SrhoUnivalent} and \eqref{SrhoTrivalent}, \eqref{ShUnivalent} and \eqref{ShTrivalent}, reciprocal symmetry of $S_{\log\rho}$ and $S_{\log h}$ on different sheets of $\RS$, and the properties of $\Theta_n$.

Now, let $S$ be as described in the statement of Proposition~\ref{prop:Sn}. Then by the principle of analytic continuation $S/S_n$ is a rational function over $\RS$ with the divisor $\sum\tr_j-\sum_{j=1}^g\tr_{n,j}$. Since rational functions have as many zeros as poles, the divisor $\sum\tr_j$ has exactly $g$ elements. Further, as explained in Section~\ref{ss:jip}, the principal divisors with strictly fewer than $g+1$ poles are necessarily involution-symmetric; that is, they come from the lifts of rational functions on $\overline\C$ to $\RS$. It also follows from Proposition~\ref{prop:ni} that $\sum_{j=1}^g\tr_{n,j}$ consists of a non-special part and a number of pairs $\infty^{(1)}+\infty^{(0)}$. Hence, $\sum\tr_j$ has the same non-special part as $\sum_{j=1}^g\tr_{n,j}$ and the same number of involution-symmetric pairs of elements. Due to Proposition~\ref{prop:ni}, the latter means that $\sum\tr_j$ solves \eqref{main-jip}.  Lastly, as all the poles of the rational function $S/S_n$ are equally split between $\infty^{(0)}$ and $\infty^{(1)}$, this is a polynomial.

It remains to show the validity of \eqref{SnNormalized}. It follows from the definition of $S_n$ and \eqref{ratioSx} that we only need to estimate
\[
\frac{\Theta_{n-1}(\z)}{\Theta_n(\z)}\frac{\Theta_n(\infty^{(0)})}{\Theta_{n-1}(\infty^{(1)})}.
\]

To this end, denote by $\mathfrak{C}_\varepsilon^0$ and $\mathfrak{C}_\varepsilon^1$ the closures of $\big\{\sum\tr_{n,j}\big\}_{n\in\N_\varepsilon}$ and $\big\{\sum\tr_{n-1,j}\big\}_{n\in\N_\varepsilon}$ in the $\RS^g/\Sigma_g$-topology. Neither of these sets contains special divisors. Indeed, both sequences consists of non-special divisors and therefore we need to consider only the limiting ones. The limit points belonging to $\mathfrak{C}_\varepsilon^0$ are necessarily of the form
\[
\sum_{i=1}^{g-2k-l}\tr_i+\sum_{i=1}^k\left(z_i^{(0)}+z_i^{(1)}\right) + l\infty^{(1)},
\]
where $\sum_{i=1}^{g-2k-l}\tr_i$, $|\pi(\tr_i)|<\infty$, is non-special and $\{z_i\}_{i=1}^k\subset\C$. If $k>0$, Proposition~\ref{prop:sequence}, applied with $l_0=0$, $l_1=1$, and $j=-1$, would imply that $\mathfrak{C}_\varepsilon^1$ contains divisors of the form
\[
\sum_{i=1}^{g-2k-l}\tr_i+\sum_{i=1}^{k^\prime}\left(w_i^{(0)}+w_i^{(1)}\right) + (k-k^\prime-1)\infty^{(0)} + (l+1+k-k^\prime)\infty^{(1)}
\]
$0\leq k^\prime\leq k-1$. In particular, it would be true that $l+1+k-k^\prime\geq2$, which is impossible by the very definition of $\N_\varepsilon$. Since the set $\mathfrak{C}_\varepsilon^1$ can be examined similarly, the claim follows.

Hence, given $\sum\tr_j\in\mathfrak{C}_\varepsilon^k$, we can define $\Theta(\z;\sum\tr_j,\sum b_j^{(1)})$ via \eqref{theta-old}. By the very definition of $\mathfrak{C}_\varepsilon^k$, it holds that
\[
0< \left|\Theta\left(\infty^{(k)};\sum\tr_j,\sum b_j^{(1)}\right)\right| < \infty.
\]
Moreover, compactness of $\mathfrak{C}_\varepsilon^k$ and the continuity of $\sum\vec\Omega(\tr_j)$ with respect to $\sum\tr_j$ imply that there are uniform constants $c(\mathfrak{C}_\varepsilon^k)$ and $C(\mathfrak{C}_\varepsilon^k)$ such that
\[
0< c(\mathfrak{C}_\varepsilon^k) \leq \left|\Theta\left(\infty^{(k)};\sum\tr_j,\sum b_j^{(1)}\right)\right| \leq C(\mathfrak{C}_\varepsilon^k) < \infty
\]
for any $\sum\tr_j\in \mathfrak{C}_\varepsilon^k$. Analogously, observe that the absolute value of $\Theta_{n-1}/\Theta_n$ is bounded above in $\RS_{n,\epsilon}$ as it is a meromorphic function in $\widehat\RS$ with poles given by the divisor $\sum\tr_{n,j}$. The fact that this bound is uniform follows again from continuity of $\vec\Omega$ and compactness of  $\mathfrak{C}_\varepsilon^k$.

For future reference, let us point out that a slight modification of the above considerations and \eqref{ratioSx} lead to the estimates
\begin{equation}
\label{SnNormalized1}
\left|S_n/S_n(\infty)\right|,\left|hS_{n-1}^*/S_{n-1}^*(\infty)\right| \leq C_{\epsilon,\varepsilon,\rho}<\infty
\end{equation}
that holds uniformly in $D^*\setminus\cup_{e\in E}\{z:|z-e|<\epsilon\}$ for all $n\in\N_\varepsilon$.

\section{Riemann-Hilbert Problem}
\label{s:5}

In what follows, we adopt the notation $\phi^{m\sigma_3}$ for the diagonal matrix $\left(\begin{array}{cc} \phi^m & 0 \\ 0 & \phi^{-m} \end{array}\right)$, where $\sigma_3$ is the Pauli matrix $\displaystyle \sigma_3 = \left(\begin{array}{cc} 1&0\\0&-1 \end{array}\right)$. Moreover, for brevity, we put $\gamma_\Delta:=\cp(\Delta)$.

\subsection{Initial Riemann-Hilbert Problem}

Let $\mathcal{Y}$ be a $2\times2$ matrix function. Consider the following Riemann-Hilbert problem for $\mathcal{Y}$ (\rhy):
\begin{itemize}
\item[(a)] $\mathcal{Y}$ is analytic in $\C\setminus\Delta$ and $\displaystyle \lim_{z\to\infty} \mathcal{Y}(z)z^{-n\sigma_3} = \mathcal{I}$, where $\mathcal{I}$ is the identity matrix;
\item[(b)] $\mathcal{Y}$ has continuous traces on each $\Delta_k$ that satisfy $\displaystyle \mathcal{Y}_+ = \mathcal{Y}_- \left(\begin{array}{cc}1&\rho\\0&1\end{array}\right);$
\item[(c)] $\mathcal{Y}$ is bounded near each $e\in E\setminus A$ and the behavior of $\mathcal{Y}$ near each $e\in A$ is described by
\[
\left\{
\begin{array}{ll}
\displaystyle \mathcal{O}\left(\begin{array}{cc}1&|z-e|^{\alpha_e}\\1&|z-e|^{\alpha_e}\end{array}\right), & \mbox{if} \quad \alpha_e<0, \smallskip \\
\displaystyle \mathcal{O}\left(\begin{array}{cc}1&\log|z-e|\\1&\log|z-e|\end{array}\right), & \mbox{if} \quad \alpha_e=0,\smallskip \\
\displaystyle \mathcal{O}\left(\begin{array}{cc}1&1\\1&1\end{array}\right), & \mbox{if} \quad \alpha_e>0,
\end{array}
\right. \quad \mbox{as} \quad D^*\ni z\to e.
\]
\end{itemize}

The connection between \rhy~ and polynomials orthogonal with respect to $\rho$ was first realized by Fokas, Its, and Kitaev \cite{FIK91,FIK92} and lies in the following.
\begin{lemma}
\label{lem:rhy}
If a solution of \rhy~ exists then it is unique. Moreover, in this case $\deg(q_n)=n$, $R_{n-1}(z)\sim z^{-n}$ as $z\to\infty$, and the solution of \rhy~ is given by
\begin{equation}
\label{eq:y}
\mathcal{Y} = \left(\begin{array}{cc}
q_n & R_n \\
m_{n-1}q_{n-1} & m_{n-1}R_{n-1}
\end{array}\right),
\end{equation}
where $m_n$ is a constant such that $m_{n-1}R_{n-1}(z)=z^{-n}[1+o(1)]$ near infinity. Conversely, if $\deg(q_n)=n$ and $R_{n-1}(z)\sim z^{-n}$ as $z\to\infty$, then $\mathcal{Y}$ defined in \eqref{eq:y} solves \rhy~.
\end{lemma}
\begin{proof}
In the case when $\Delta=[-1,1]$ and $\rho>0$ on $\Delta$ this lemma has been proven in \cite[Lemma~2.3]{KMcLVAV04}. It has been explained in \cite{BY10} that the lemma translates without change to the case of a general closed analytic arc and a general analytic non-vanishing weight $\rho$, and yields the uniqueness of the solution of \rhy~ whenever the latter exists. For a general contour $\Delta$ the claim follows from the fact that $R_n=\sum_k R_{nk}$, where
\[
R_{nk}(z) := \int_{\Delta_k}\frac{q_n(t)\rho(t)}{t-z}\frac{dt}{2\pi i} = \int_{\Delta_k}\frac{q_n(t) w_k(t)(t-a)^{\alpha_a}(t-b)^{\alpha_b}}{t-z}\frac{dt}{2\pi i}
\]
and therefore the behavior of $R_n$ near $e\in A$ is deduced from the behavior $R_{nk}$ there. On the other hand, for each arc $\Delta_{e,j}$ incident with $e\in E\setminus A$ (see notation in \eqref{fSum}), the respective function $R_{nk}$ behaves as \cite[Section~8.1]{Gakhov}
\[
\frac{\rho_{e,j}(e)}{2\pi i}\log(z-e)+R_{e,j}^*(z),
\]
where the function $R_{e,j}^*$ has a definite limit at $e$ and the logarithm is holomorphic outside of $\Delta_{e,j}$. Using \eqref{rhoSum}, we get that
\[
R(z) = \frac{\rho_{e,1}(e)}{2\pi}\arg_{e,1}(z-e) + \frac{\rho_{e,2}(e)}{2\pi}\arg_{e,2}(z-e) + \frac{\rho_{e,3}(e)}{2\pi}\arg_{e,3}(z-e) + R_e^*(z),
\]
where $R_e^*$ has a definite limit at $e$ and $\arg_{e,j}(z-e)$ has the branch cut along $\Delta_{e,j}$. Thus, $\mathcal{Y}$ is bounded in the vicinity of each $e\in E\setminus A$.

Suppose now that the solution, say $\mathcal{Y}=[\mathcal{Y}_{jk}]_{j,k=1}^2$, of \rhy~ exists. Then $\mathcal{Y}_{11}=z^n+${ lower order terms} by the normalization in \rhy(a). Moreover, by \rhy(b), $\mathcal{Y}_{11}$ has no jump on $\Delta$ and hence is holomorphic in the whole complex plane. Thus, $\mathcal{Y}_{11}$ is necessarily a polynomial of degree $n$ by Liouville's theorem. Further, since $\mathcal{Y}_{12}=\mathcal{O}(z^{-n-1})$ and satisfies \rhy(b), it holds that $\mathcal{Y}_{12}$ is the Cauchy transform of $\mathcal{Y}_{11}\rho$. From the latter, we easily deduce that $\mathcal{Y}_{11}$ satisfies orthogonality relations \eqref{ortho}. Applying the same arguments to the second row of $\mathcal{Y}$, we obtain that $\mathcal{Y}_{21}=q_{n-1}$ and $\mathcal{Y}_{22}=m_{n-1}R_{n-1}$ with $m_{n-1}$ well-defined.

Conversely, let $\deg(q_n)=n$ and $R_{n-1}(z)=\mathcal{O}(z^{-n})$ as $z\to\infty$. Then it can be easily checked by the direct examination  of \rhy(a)--(c) that $\mathcal{Y}$, given by \eqref{eq:y}, solves \rhy.
\end{proof}

\subsection{Renormalized Riemann-Hilbert Problem}

Suppose now that \rhy~ is solvable and $\mathcal{Y}$ is the solution. Define
\begin{equation}
\label{eq:t}
\mathcal{T} := \gamma_\Delta^{-n\sigma_3}\mathcal{Y}\map^{-n\sigma_3},
\end{equation}
where, as before, we use the same symbol $\map$ for the pull-back function of $\map$ from $D^{(0)}$. By \eqref{defmap} it holds that $\lim_{z\to\infty}z/\map(z)=\gamma_\Delta$ and therefore
\begin{equation}
\label{t1}
\lim_{z\to\infty} \mathcal{T}(z) = \lim_{z\to\infty} \gamma_\Delta^{-n\sigma_3}\mathcal{Y}(z)z^{-n\sigma_3}\left(z/\map(z)\right)^{n\sigma_3} = \mathcal{I}.
\end{equation}
Moreover, it holds by \eqref{jumponab} that
\begin{equation}
\label{t2}
(\map^+)^{-n\sigma_3} = (\map^-)^{-n\sigma_3}e^{-2\pi in\omega_k\sigma_3}
\end{equation}
on each $\Delta_k^a$. Finally, on each $\Delta_k$ we have that
\begin{eqnarray}
(\map^-)^{-n\sigma_3}\left(\begin{array}{cc}1&\rho\\0&1\end{array}\right)(\map^+)^{-n\sigma_3} &=& \left(\begin{array}{cc}(\map^-/\map^+)^n & (\map^-\map^+)^n\rho \smallskip \\ 0 & (\map^+/\map^-)^n \end{array}\right) \nonumber \\
\label{t3}
&=& \left(\begin{array}{cc} e^{2\pi in\delta_k}(\map^+)^{-2n} & e^{2\pi in\delta_k}\rho \smallskip \\ 0 & e^{2\pi in\delta_k}(\map^-)^{-2n} \end{array}\right),
\end{eqnarray}
where the second equality holds again by \eqref{jumponab} and $\delta_k$ are defined right after \eqref{jumponab}. Combining \eqref{t1}---\eqref{t3}, we see that $\mathcal{T}$ solves the following Riemann-Hilbert problem (\rht):
\begin{itemize}
\item[(a)] $\mathcal{T}$ is analytic in $D_a$ and $\mathcal{T}(\infty)=\mathcal{I}$;
\item[(b)] $\mathcal{T}$ has continuous traces on $\bigcup \Delta_k\cup\bigcup \Delta_k^a$ that satisfy
\[
\displaystyle \mathcal{T}_+ = \mathcal{T}_- \left\{
\begin{array}{ll}
e^{-2\pi in\omega_k\sigma_3} & \mbox{on each} \quad \Delta_k^a, \bigskip \\
\left(\begin{array}{cc} e^{2\pi in\delta_k}(\map^+)^{-2n} & e^{2\pi in\delta_k}\rho \smallskip \\ 0 & e^{2\pi in\delta_k}(\map^-)^{-2n} \end{array}\right) & \mbox{on each} \quad \Delta_k;
\end{array}
\right.
\]
\item[(c)] $\mathcal{T}$ has the behavior near each $e\in E$ as described in \rhy(c) only with $D^*$ replaced by $D_a$.
\end{itemize}

Trivially, the following lemma holds.

\begin{lemma}
\label{lem:rht}
\rht~ is solvable if and only if \rhy~ is solvable. When solutions of \rht~ and \rhy~ exist, they are unique and connected by \eqref{eq:t}.
\end{lemma}

\subsection{Opening of Lenses}
\label{ss:ol}

As is standard in the Riemann-Hilbert approach, the second transformation of ~\rhy~ is based on the following factorization of the jump matrix \eqref{t3} in \rht(b):
\[
\left(\begin{array}{cc} 1 & 0 \\ (\map^-)^{-2n}/\rho & 1 \end{array}\right) \left(\begin{array}{cc} 0 & e^{2\pi in\delta_k}\rho \\ -e^{-2\pi in\delta_k}/\rho & 0 \end{array}\right) \left(\begin{array}{cc} 1 & 0 \\ (\map^+)^{-2n}/\rho & 1 \end{array}\right),
\]
where we used \eqref{jumponab}. This factorization allows us to consider a Riemann-Hilbert problem with jumps on a lens-shaped contour $\Sigma$ (see the right-hand part of Figure~\ref{fig:2}), which is defined as follows.
\begin{figure}[!ht]
\centering
\includegraphics[scale=.5]{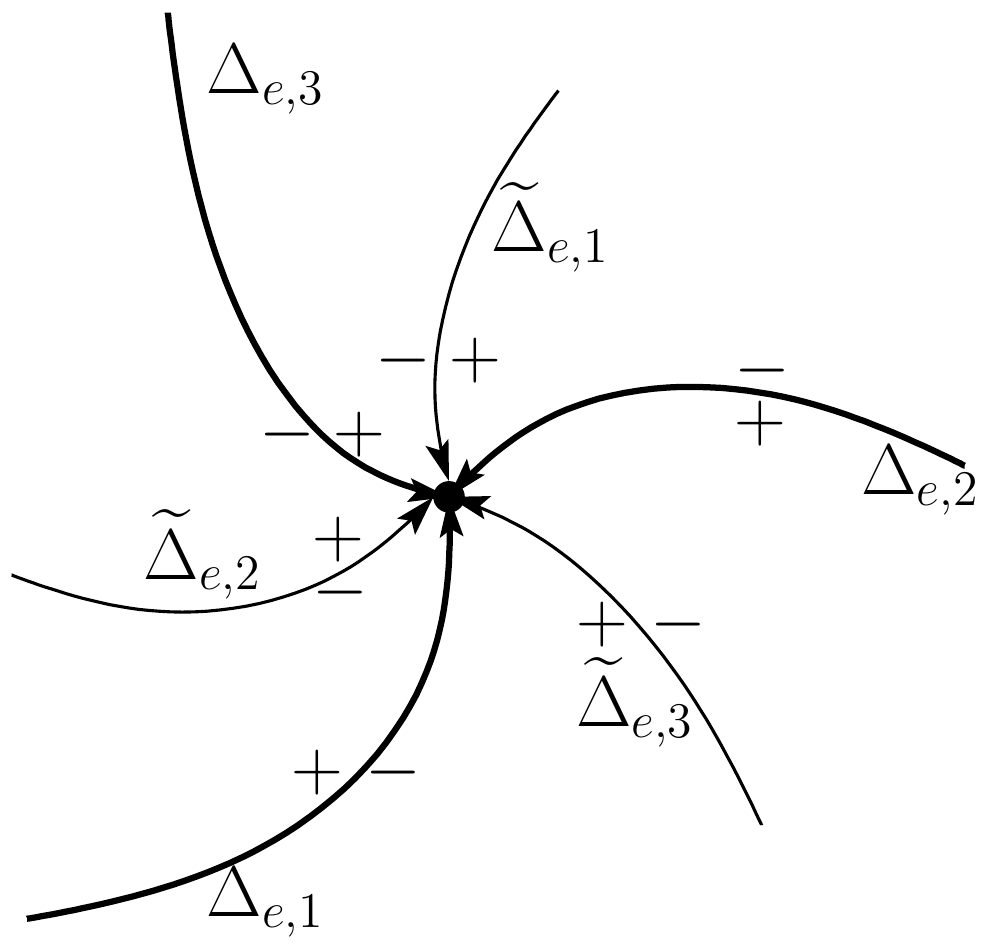}
\hspace{.5in}
\includegraphics[scale=.6]{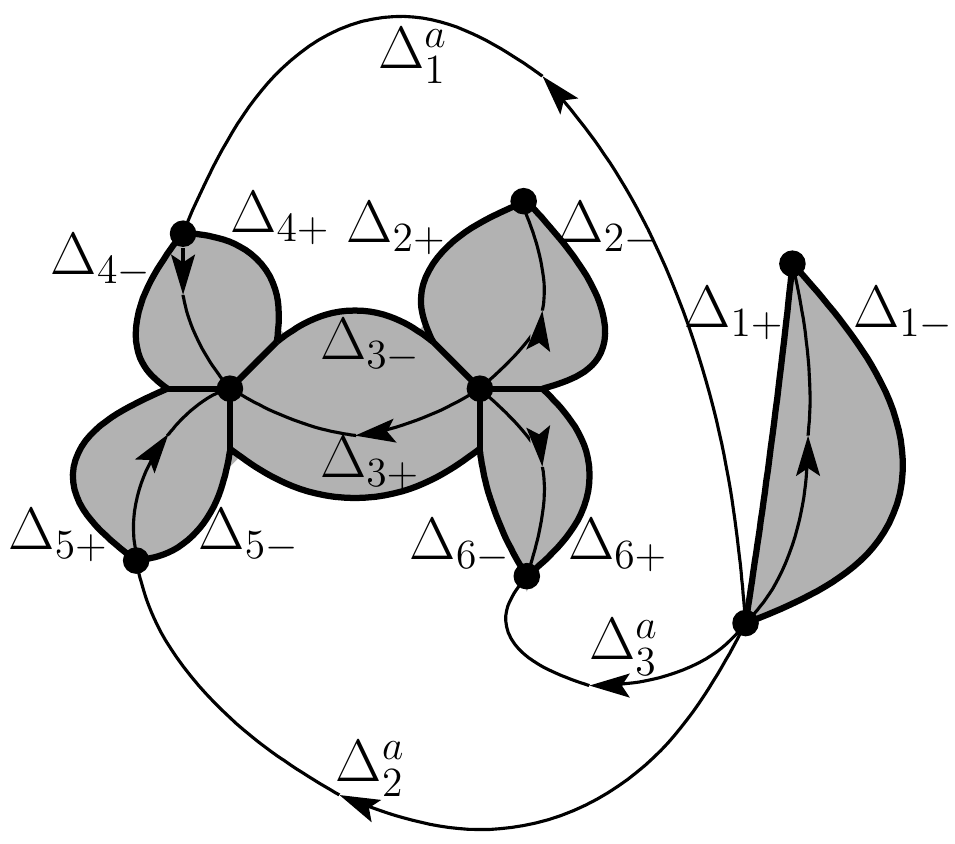}
\caption{\small The left figure: the arcs $\widetilde\Delta_{e,j}$ introduced in the construction of the lens $\Sigma$ near a trivalent end. The right figure: the full lens $\Sigma$ consisting of the arcs $\Delta_k$(not labeled), the arcs $\Delta_k^a$, and the outer arcs $\Delta_{k\pm}$ (the choice of $\pm$ is determined by the chosen orientation of the corresponding arc $\Delta_k$), and the domains $\Omega_{k\pm}$ (shaded areas, not labeled).}
\label{fig:2}
\end{figure}
For each trivalent end $e$, let $\Delta_{e,j}$, $j\in\{1,2,3\}$, be the arcs in $\Delta$ incident with $e$. For definiteness, assume that they are ordered counter-clockwise; that is, when encircling $e$ in the counter-clockwise direction we first encounter $\Delta_{e,1}$, then $\Delta_{e,2}$, and then $\Delta_{e,3}$. Assume also that $\epsilon>0$ is small enough so that the intersection of the disk centered at $e$ of radius $\epsilon$, say $U_{e,\epsilon}$, with any $\Delta_{e,j}$ is a Jordan arc and the disk itself is contained in the domain of holomorphy of each $w_{e,j}$ (see \eqref{rho}). Firstly, let $\widetilde\Delta_{e,j}$, $j\in\{1,2,3\}$, be three open analytic arcs incident with $e$  and some points on the circumference of $U_{e,\epsilon}$ placed so that the arc $\widetilde\Delta_{e,j}$ splits the sector formed by $\Delta_{e,j-1}$ and $\Delta_{e,j+1}$, where we understand $j\pm1$ cyclicly within the set $\{1,2,3\}$. We orient the arcs $\widetilde\Delta_{e,j}$ so that all the arcs including $\Delta_{e,j}$ are simultaneously oriented either towards $e$ or away from $e$ (see the left-hand side of Figure~\ref{fig:2}). Secondly, let $\Delta_k$ be an arc with one univalent and one trivalent endpoint, say $e$. Then we chose open analytic arcs $\Delta_{k\pm}\subset D_a$ so that $\Delta_k\cup\Delta_{k+}\cup\widetilde\Delta_{e,j}$ and $\Delta_k\cup\Delta_{k-}\cup\widetilde\Delta_{e,l}$ ($l=j+1$ if $\Delta_k$ is oriented towards $e$ and $l=j-1$ otherwise) delimit two simply connected domains, say $\Omega_{k+}$ and $\Omega_{k-}$, that lie to the left and right of $\Delta_k$ (see the right-hand side of Figure~\ref{fig:2}). We oriented $\Delta_{k\pm}$ the way $\Delta_k$ is oriented and assume that they lie within the domain of holomorphy of $w_k$. The cases where $\Delta_k$ is incident with two univalent ends or two trivalent ends, we treat similarly with the obvious modifications. Finally, we require all the arcs $\Delta_k$, $\Delta_{k\pm}$, $\Delta_k^a$, and $\widetilde\Delta_{e,j}$ to be mutually disjoint, in particular, we have that $\Delta_k^a\cap\Omega_{j\pm}=\varnothing$ for all possible pairs of $k$ and~$j$.

Suppose now that \rht~ is solvable and $\mathcal{T}$ is the solution. Define $\mathcal{S}$ on $\overline\C\setminus\Sigma$ by
\begin{equation}
\label{eq:s}
\mathcal{S}:= \left\{
\begin{array}{ll}
\mathcal{T} \left(\begin{array}{cc}1&0 \\ \mp \map^{-2n}/\rho & 1 \end{array}\right), & \mbox{in each} \quad \Omega_{k\pm}, \smallskip \\
\mathcal{T}, & \mbox{outside the lens} \ \Sigma.
\end{array}
\right.
\end{equation}
Then $\mathscr{S}$ solves the following Riemann-Hilbert problem (\rhs):
\begin{itemize}
\item[(a)] $\mathcal{S}$ is analytic in $\overline\C\setminus\Sigma$ and $\mathcal{S}(\infty)=\mathcal{I}$;
\item[(b)] $\mathcal{S}$ has continuous traces on $\Sigma$ that satisfy \smallskip
\begin{itemize}
\item[(1)] $\displaystyle \mathcal{S}_+=\mathcal{S}_-e^{-2\pi in\omega_k\sigma_3}$ on each $\Delta_k^a$; \smallskip
\item[(2)] $\displaystyle \mathcal{S}_+=\mathcal{S}_-\left(\begin{array}{cc} 0 & e^{2\pi in\delta_k}\rho \\ -e^{-2\pi in\delta_k}/\rho & 0 \end{array}\right)$ on each $\Delta_k$; \smallskip
\item[(3)] $\displaystyle \mathcal{S}_+=\mathcal{S}_-\left(\begin{array}{cc} 1 & 0 \\ \map^{-2n}/\rho & 1 \end{array}\right)$ on each $\Delta_{k\pm}$; \smallskip
\item[(4)] $\displaystyle \mathcal{S}_+=\mathcal{S}_-\left(\begin{array}{cc} 1 & 0 \\ \map^{-2n}(1/\rho_{e,j-1}+1/\rho_{e,j+1}) & 1 \end{array}\right)$ on each $\widetilde\Delta_{e,j}$; \smallskip
\end{itemize}
\item[(c)] $\mathcal{S}$ is bounded near each $e\in E\setminus A$ and the behavior of $\mathcal{S}$ near each $e\in A$ is described by
\[
\left\{
\begin{array}{lll}
\displaystyle \mathcal{O}\left(\begin{array}{cc}1&|z-e|^{\alpha_e}\\1&|z-e|^{\alpha_e}\end{array}\right), & \mbox{if} \quad \alpha_e<0, & \mbox{as} \quad \C\setminus\Sigma\ni z\to e, \smallskip \\
\displaystyle \mathcal{O}\left(\begin{array}{cc}\log|z-e|&\log|z-e|\\\log|z-e|&\log|z-e|\end{array}\right), & \mbox{if} \quad \alpha_e=0,  & \mbox{as} \quad \C\setminus\Sigma\ni z\to e, \smallskip \\
\displaystyle \mathcal{O}\left(\begin{array}{cc}1&1\\1&1\end{array}\right), & \mbox{if} \quad \alpha_e>0, & \mbox{as} \ z\to e \ \mbox{outside the lens} \ \Sigma, \smallskip \\
\displaystyle \mathcal{O}\left(\begin{array}{cc}|z-e|^{-\alpha_e}&1\\|z-e|^{-\alpha_e}&1\end{array}\right), & \mbox{if} \quad \alpha_e>0, & \mbox{as} \ z\to e \ \mbox{inside the lens} \ \Sigma.
\end{array}
\right.
\]
\end{itemize}

Then the following lemma holds.

\begin{lemma}
\label{lem:rhs}
\rhs~ is solvable if and only if \rht~ is solvable. When solutions of \rhs~ and \rht~ exist, they are unique and connected by \eqref{eq:s}.
\end{lemma}
\begin{proof}
By construction, the solution of \rht~ yields a solution of \rhs. Conversely, let $\mathscr{S}^*$ be a solution of \rhs. It can be readily verified that $\mathscr{T}^*$, obtained from $\mathscr{S}^*$ by inverting \eqref{eq:s}, satisfies \rht(a)-(b). Denote by $\mathscr{T}^*_{jk}$ the $jk$-entry of $\mathscr{T}^*$, $j,k\in\{1,2\}$. The appropriate behavior of $\mathscr{T}_{j2}^*$ near the points of $E$ follows immediately from \rhs(c) and \eqref{eq:s}. Thus, we only need to show that $\mathscr{T}^*_{j1}=\mathcal{O}(1)$ in the vicinity of $E$ and only for $e\in A$. Observe that by simply inverting transformation \eqref{eq:s}, we get that
\begin{equation}
\label{eq:neartheendentires}
\mathscr{T}^*_{j1}(z) = \left\{
\begin{array}{ll}
\mathcal{O}(1), & \mbox{if} \quad \alpha_e<0 \smallskip \\
\mathcal{O}(\log|z-e|), & \mbox{if} \quad \alpha_e=0, \smallskip \\
\mathcal{O}(|z-e|^{-\alpha_e}), & \mbox{if} \quad \alpha_e>0 \ \mbox{and} \ z \ \mbox{is inside the lens}, \smallskip \\
\mathcal{O}(1), & \mbox{if} \quad \alpha_e>0 \ \mbox{and} \ z \ \mbox{is outside the lens},
\end{array}
\right.
\end{equation}
for $j=1,2$. However, each $\mathscr{T}^*_{j1}$ solves the following scalar boundary value problem:
\begin{equation}
\label{eq:bvprn}
\phi^+ = \phi^-
\left\{
\begin{array}{ll}
e^{-2\pi in\omega_k} & \mbox{on} \quad\bigcup \Delta_k^a, \smallskip \\
e^{2\pi in\delta_k}(\map^+)^{-2n} & \mbox{on} \quad\bigcup \Delta_k,
\end{array}
\right.
\end{equation}
where $\phi$ is a holomorphic function in $D_a$. It can be easily checked using \eqref{jumponab} that $\map^{-n}$ is the canonical solution of \eqref{eq:bvprn}. Hence, the functions $\phi_j:=\mathscr{T}^*_{j1}\map^n$, $j=1,2$, are analytic in $\C\setminus E$. Moreover, according to (\ref{eq:neartheendentires}), the singularities of these functions at the points $e\in E$ cannot be essential, thus, they are either removable or polar. In fact, since $\phi_j(z)=\mathcal{O}(1)$ or $\phi_j(z)=\mathcal{O}(\log|z-e|)$ when $z$ approaches $e$ outside of the lens $\Sigma$, $\phi_j$ can have only removable singularities at these points. Hence, $\phi_j(z)=\mathcal{O}(1)$ and subsequently $\mathscr{T}^*_{j1}=\mathcal{O}(1)$  near each $e\in E$ (clearly, these functions have the form $q\map^{-n}$, where $q$ is any polynomials of degree at most $n$).
\end{proof}

\section{Asymptotic Analysis}
\label{s:6}

\subsection{Analysis in the Bulk}
\label{ss:6.1}

As $\Phi^{-2n}$ converges to zero geometrically fast away from $\Delta$, the second jump matrix in \rhs(b)~ is close to the identity on $\bigcup \Delta_{k\pm}$. Thus, the main term of the asymptotics for $\mathscr{S}$ in $D_a$ is determined by the following Riemann-Hilbert problem (\rhn):
\begin{itemize}
\item[(a)] $\mathcal{N}$ is analytic in $D_a$ and $\mathcal{N}(\infty)=\mathcal{I}$;
\item[(b)] $\mathcal{N}$ has continuous traces on $\bigcup \Delta_k\cup\bigcup \Delta_k^a$ that satisfy \rhs(b1)--(b2).
\end{itemize}

As usual, we denote by $S_n$ and $S_n^*$ the pull-back functions of $S_n$ on $\RS$ defined in Proposition~\ref{prop:Sn}.

\begin{lemma}
\label{lem:rhn}
If $n\in\N_\varepsilon$, then \rhn~ is solvable and the solution is given by
\begin{equation}
\label{eq:n}
\mathcal{N} = \left(\begin{array}{cc} 1/S_n(\infty) & 0 \smallskip \\ 0 & \gamma_\Delta/S_{n-1}^*(\infty) \end{array}\right) \left(\begin{array}{cc} S_n & hS_n^* \smallskip \\ S_{n-1}/\map & h\map S_{n-1}^* \end{array}\right).
\end{equation}
Moreover, $\det(\mathcal{N})\equiv1$ on $\C$, and it holds that $\mathcal{N}$ behaves like
\begin{equation}
\label{n-ends}
\left(\begin{array}{cc}\mathcal{O}\left(|z-e|^{-(2\alpha_e+1)/4}\right)&\mathcal{O}\left(|z-e|^{(2\alpha_e-1)/4}\right)\smallskip\\\mathcal{O}\left(|z-e|^{-(2\alpha_e+1)/4}\right)&\mathcal{O}\left(|z-e|^{(2\alpha_e-1)/4}\right)\end{array}\right)
\quad \mbox{and} \quad
\left[\mathcal{O}\left(|z-e|^{-1/4}\right)\right]_{j,k=1}^2
\end{equation}
$D_a\ni z\to e$ for univalent and trivalent ends $e\in E$, respectively.
\end{lemma}
\begin{proof}
Observe that whenever $n\in\N_\varepsilon$ it holds that $S_n(\infty)S_{n-1}^*(\infty)\neq0$ by the construction and therefore $\mathcal{N}$ is well-defined for such indices. Since $S_n$ and $h \map S_n^*$ are holomorphic function in $D_a$ by \eqref{SnE}, $\mathcal{N}$ is an analytic matrix function there. The normalization $\mathcal{N}(\infty)=\mathcal{I}$ follows from the analyticity of $S_n$ and $S_n^*$ at infinity and the fact that $(h\map)(z)=1/\gamma_\Delta+\mathcal{O}(1/z)$. Further, for any $\Delta_k^a$ we have that
\begin{eqnarray}
\mathcal{N}_+ &=&\left(\begin{array}{cc} 1/S_n(\infty) & 0 \smallskip \\ 0 & \gamma_\Delta/S_{n-1}^*(\infty) \end{array}\right) \left(\begin{array}{cc} S_n^-e^{-2\pi in\omega_k} & h(S_n^*)^-e^{2\pi in\omega_k} \smallskip \\ (S_{n-1}/\map)^-e^{-2\pi in\omega_k} & h(\map S_{n-1}^* )^-e^{2\pi in\omega_k}\end{array}\right) \nonumber \\
&=& \mathcal{N}_-e^{-2\pi in\omega_k\sigma_3}, \nonumber
\end{eqnarray}
where we used \eqref{jumpPhiperoids}, analyticity of $S_n\map^n$ across the $\mathbf{a}$-cycles, and the fact that $S_n^+/S_n^-=(S_n^*)^-/(S_n^*)^+$. Moreover, for each $\Delta_k$ it holds that
\[
S_n^\pm = (S_n^*)^\mp\exp\big\{-2\pi in\delta_k\big\}\big(h^+/\rho\big)
\]
by \eqref{jumponab}, \eqref{jumpSn}, and \eqref{Da}. Then
\begin{eqnarray}
\mathcal{N}_+  &=& \left(\begin{array}{cc} 1/S_n(\infty) & 0 \smallskip \\ 0 & \gamma_\Delta/S_{n-1}^*(\infty) \end{array}\right) \left(\begin{array}{cc} -(hS_n^*)^-e^{-2\pi in\delta_k}/\rho & S_n^-e^{2\pi in\delta_k}\rho \smallskip \\ -(hS_{n-1}^*\map)^-e^{-2\pi in\delta_k}/\rho & (S_{n-1}/\map)^-e^{2\pi in\delta_k}\rho \end{array}\right) \nonumber \\
{} &=& \mathcal{N}_-\left(\begin{array}{cc} 0 & e^{2\pi in\delta_k}\rho \smallskip\\ -e^{-2\pi in\delta_k}/\rho & 0 \end{array}\right) \nonumber
\end{eqnarray}
on $\Delta_k$, again by \eqref{jumponab}, \eqref{jumpSn}, \eqref{Da}, and since $h^-=-h^+$ there. Thus, $\mathcal{N}$ as defined in \eqref{eq:n} does solve \rhn. Equations \eqref{n-ends} readily follow from \eqref{SnE}. Finally, as the determinants of the jump matrices in \rhn(b) are equal to 1, $\det(\mathcal{N})$ is a holomorphic function in $\overline\C\setminus E$. However, it follows from \eqref{n-ends} that
\[
\det(\mathcal{N})(z)\leq\const|z-e|^{-1/2} \quad \mbox{as} \quad z\to e\in E.
\]
Thus, $\det(\mathcal{N})$ is a function holomorphic in the entire extended complex plane and therefore is a constant. From the normalization at infinity, we get that $\det(\mathcal{N})\equiv1$.
\end{proof}

\subsection{Local Analysis Near Univalent Ends}
\label{ss:uni}

In the previous section we described the main term of the asymptotics of $\mathcal{S}$ away from $\Delta$. In this section we shall do the same near the points in $A$. Recall that there exists exactly one $k=k(a)$ such that the arc $\Delta_k$ is incident with $a$. Until the end of this section, we understand that $k$ is this fixed integer. Moreover, we let $J_a$ to be the possibly empty index set such that $\Delta_j^a$ has $a$ as its endpoint for each $j\in J_a$.

\subsubsection{Riemann-Hilbert Problem for Local Parametrix}
Let $\delta>0$ be small enough so that the intersection of the ball of radius $\delta$ centered at $a$, say $U_{a,\delta}$, with each of the arcs comprising $\Sigma$ and incident with $a$ is again a Jordan arc. We are seeking the solution of the following \rhpa:
\begin{itemize}
\item[(a)] $\mathcal{P}_a$ is analytic in $U_{a,\delta}\setminus\Sigma$;
\item[(b)] $\mathcal{P}_a$ has continuous traces on each side of $U_{a,\delta}\cap\Sigma$ that satisfy \rhs(b1)--(b3);
\item[(c)] $\mathcal{P}_a$ has the behavior near $a$ within $U_{a,\delta}$ described by \rhs(c);
\item[(d)] $\mathcal{P}_a\mathcal{N}^{-1}=\mathcal{I}+\mathcal{O}(1/n)$ uniformly on $\partial U_{a,\delta}\setminus\Sigma$, where $\mathcal{N}$ is given by \eqref{eq:n}.
\end{itemize}
We solve \rhpa~ only for $n\in\N_\varepsilon$. For these indices the above problem is well-posed as $\det(\mathcal{N})\equiv1$ by Lemma~\ref{lem:rhn} and therefore $\mathcal{N}^{-1}$ is an analytic matrix function in $D_a$. In fact, the solution does not depend on the actual value of $\varepsilon$, however, the term $\mathcal{O}(1/n)$ in \rhpa(d) does depend on $\varepsilon$ as well as $\delta$. That is, this estimate is uniform with $n$ for each fixed $\varepsilon$ and $\delta$, but is not uniform with respect to $\varepsilon$ or $\delta$ approaching zero.

To describe the solution of \rhpa, we need to define three special objects. The first one is the so-called $G$-function whose square conformally maps $U_{a,\delta}$ into some neighborhood of the origin in such a fashion that $\Delta_k$ is mapped into negative reals. The second one is a holomorphic matrix function needed to satisfy \rhpa(d). The third is a holomorphic matrix function that solves auxiliary Riemann-Hilbert problem with constant jumps.

\subsubsection{$G$-Function}
\label{sss:gfun}
Set
\[
G_a(z) := \int_a^zh(t)dt, \quad z\in U_{a,\delta}\setminus\Delta_k.
\]
Then $G_a$ is a holomorphic function in $U_{a,\delta}\setminus\Delta_k$ such that
\begin{equation}
\label{GeMap}
\left|\map e^{-G_a}\right| \equiv1 \quad \mbox{in} \quad U_{a,\delta}.
\end{equation}
Indeed, since both $a_1$ and $a$ belong to $\Delta$ and the Green differential $dG$ has purely imaginary periods, the integral $\int_{a_1}^adG$ is purely imaginary itself. It is also true that
\begin{equation}
\label{Gepm}
G_a^++G_a^- \equiv 0 \quad \mbox{on} \quad \Delta_k\cap U_{a,\delta}
\end{equation}
since $h^++h^-\equiv0$ on $\Delta$. Moreover, it holds that the traces $G_a^\pm$ have purely imaginary values on $\Delta_k$ as the same is true for $h^\pm(t)dt$ (recall that the quadratic differential $h^2(z)dz^2$ is negative on $\Delta_k$). The last observation and \eqref{Gepm} imply that $G_a^2$ is a holomorphic function in $U_{a,\delta}$ that assumes negative values on $\Delta_k\cap U_{a,\delta}$. Furthermore,
\begin{equation}
\label{Geat0}
|(G_a^2)^\prime(a)| = 2 \lim_{t\to a}|h(t)||t-a|^{1/2}  \neq 0.
\end{equation}
Property \eqref{Geat0} implies that $G_a^2$ is univalent in some neighborhood of $a$. Without loss of generality, we can assume that $\delta$ is small enough for $G_a^2$ to be univalent in $U_{a,\delta}$. Hence, $G_a^2$ maps $U_{a,\delta}$ conformally onto some neighborhood of the origin. In particular, this means that $\Delta_k$ can be extended as an analytic arc beyond $a$ by the preimage of $[0,\infty)$ under $G_a^2$ and we denote by $\widetilde\Delta_k$ this extension. 

Let $I_+ := \{z:~\Arg(z)=2\pi/3\}$, $I:= \{z:~\Arg(z)=\pi\}$, and $I_- := \{z:~\Arg(z)=-2\pi/3\}$ be three semi-infinite rays oriented towards the origin. Since we had some freedom in choosing the arcs $\Delta_{k\pm}$, we require that
\[
G_a^2\left((\Delta_{k+}\cup\Delta_{k-})\cap U_{a,\delta}\right)\subset I_+\cup I_-.
\]
The latter is possible as $G_a^2$ is conformal around $a$. We denote by $U_{a,\delta}^+$ (resp. $U_{a,\delta}^-$) the open subset of $U_{a,\delta}$ that is mapped by $G_a^2$ into the upper (resp. lower) half-plane. Clearly, there are two possibilities, either $\Delta_{k+}\subset U_{a,\delta}^+$ and therefore $\Delta_k$ is oriented towards $a$, or $\Delta_{k+}\subset U_{a,\delta}^-$ and respectively $\Delta_k$ is oriented away from $a$ (see Figure~\ref{fig:3}).
\begin{figure}[!ht]
\centering
\includegraphics[scale=.5]{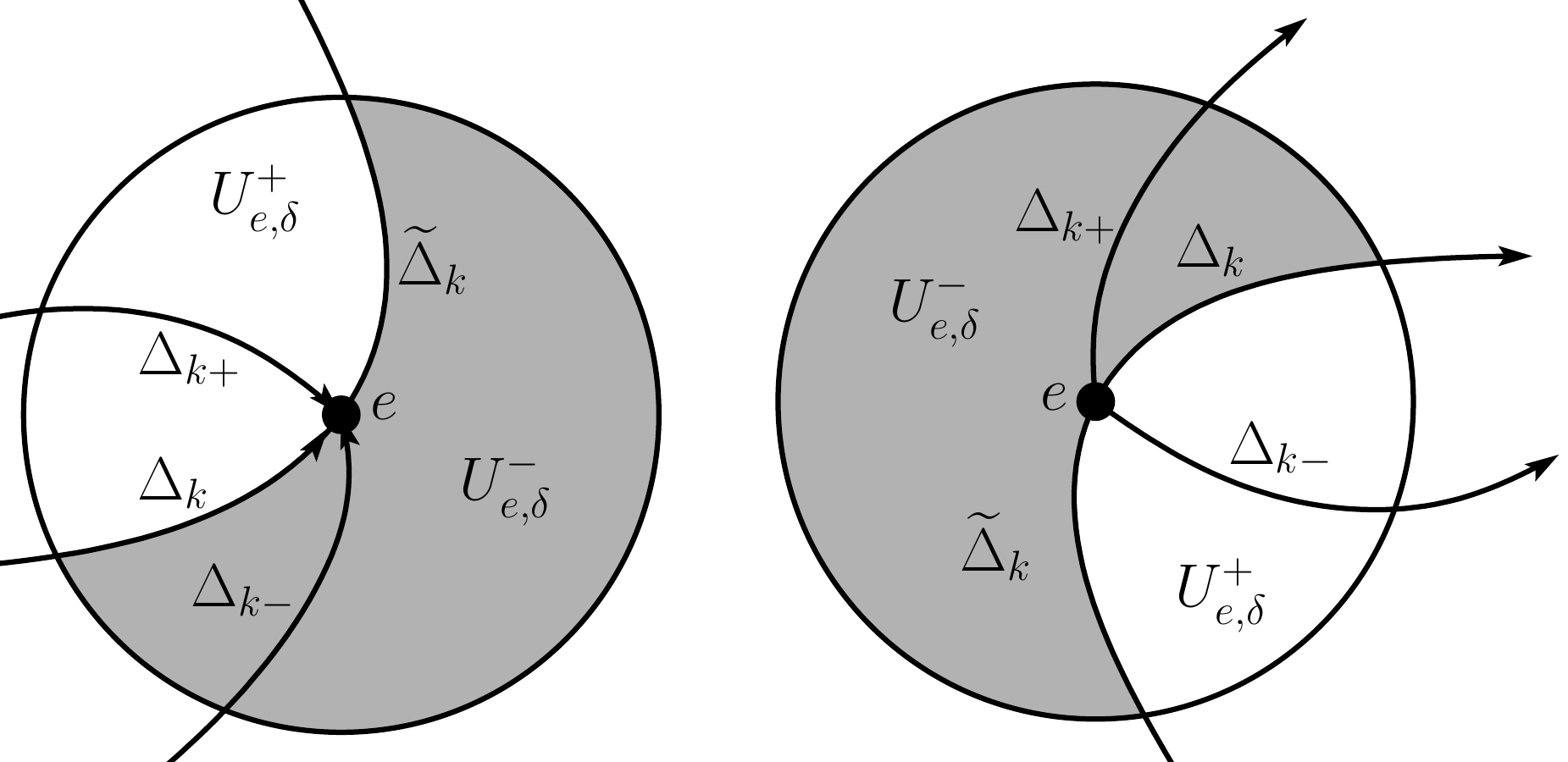}
\caption{\small Schematic representation of the arcs $\Delta_k$, $\widetilde\Delta_k$, $\Delta_{k\pm}$, the domains $U_{a,\delta}^+$ (shaded part of the disk) and $U_{a,\delta}^-$ (unshaded part of the disk), and two possible cases: $\Delta_{k\pm}\subset U_{a,\delta}^\pm$ ($\Delta_k$ oriented towards $a$) and $\Delta_{k\pm}\subset U_{a,\delta}^\mp$ ($\Delta_k$ oriented away from $a$).}
\label{fig:3}
\end{figure}

Finally, since the traces $G_a^\pm$ are purely imaginary on $\Delta_k\cap U_{a,\delta}$, satisfy \eqref{Gepm} there, and the increment of the argument of $G_a$ is $\pi$ when $a$ is encircled in the clockwise direction from a point on $\Delta_k\cap U_{a,\delta}$ back to itself, we can define the square root of $G_a$ that satisfies
\begin{equation}
\label{secondroot}
\left(G_a^{1/2}\right)^+ =\pm i\left(G_a^{1/2}\right)^- \quad \mbox{on} \quad \Delta_k\cap U_{a,\delta},
\end{equation}
where the sign $+$ must be used when $\Delta_k$ is oriented towards $a$ and the sign $-$ otherwise.

\subsubsection{Matrix Function $\mathcal{E}$}

Let $\arg(z-a)$ be the branch of the argument of $(z-a)$ that was used in the definition of $\rho$ in \eqref{rho}.  Without loss of generality we assume that its branch cut is $\widetilde\Delta_k$.  Put
\begin{equation}
\label{matrixW}
\mathcal{W} := \left(\map^nw_k^{1/2}(a-\cdot)^{\alpha_a/2}\right)^{\sigma_3} \quad \mbox{in} \quad U_{a,\delta}\setminus\Delta_k,
\end{equation}
where we take the principal value of the square root of $w_k$ (we assume that $\delta$ is small enough so $w_k$ is holomorphic and non-vanishing in $U_{a,\delta}$) and use the branch \eqref{argez} to define $(a-\cdot)^{\alpha_a/2}$. Then it holds that
\begin{equation}
\label{we2}
\left(w_k^{1/2}(a-\cdot)^{\alpha_a/2}\right)^2 = e^{\pm\alpha_a\pi i}\rho \quad \mbox{in} \quad U_{a,\delta}^\pm
\end{equation}
and
\begin{equation}
\label{webvp}
\left(w_k^{1/2}(a-\cdot)^{\alpha_a/2}\right)^+\left(w_k^{1/2}(a-\cdot)^{\alpha_a/2}\right)^- = \rho \quad \mbox{on} \quad \Delta_k\cap U_{a,\delta}.
\end{equation}
So the matrix function $\mathcal{NW}$ is holomorphic in $U_{a,\delta}\setminus \left(\Delta_k\cup\bigcup_{J_a}\Delta_j^a\right)$ and
\[
\mathcal{N}_+\mathcal{W}_+ = \mathcal{N}_-\mathcal{W}_- \left(\begin{array}{rr} 0 & 1 \\ -1 & 0 \end{array}\right) \quad \mbox{on} \quad  \Delta_k\cap U_{a,\delta}.
\]
Moreover, it is, in fact, holomorphic across each $\Delta_j^a$, $j\in J_a$, as
\[
\mathcal{N}_+\mathcal{W}_+ = \mathcal{N}_-e^{-2\pi in\omega_j\sigma_3}\mathcal{W}_-e^{2\pi in\omega_j\sigma_3} = \mathcal{N}_-\mathcal{W}_-
\]
by \eqref{t2} and since $\mathcal{W}$ is diagonal. Hence, we deduce from \eqref{Gepm} that $\mathcal{NW}\exp\{-nG_a\sigma_3\}$ is holomorphic in $U_{a,\delta}\setminus\Delta_k$ and has the same jump across $\Delta_k$ as $\mathcal{NW}$. Define
\[
\mathcal{E} := \mathcal{NW}\exp\big\{-nG_a\sigma_3\big\}\frac{1}{\sqrt2} \left(\begin{array}{cc} 1 & \mp i \\ \mp i & 1 \end{array}\right) \left(\pi nG_a\right)^{\sigma_3/2}
\]
where the sign $-$ must be used when $\Delta_k$ is oriented towards $a$ and the sign $+$ otherwise. Since the product
\[
\left(\left(G_a^{1/2}\right)^-\right)^{-\sigma_3} \left(\begin{array}{rr} 1/2 & \pm i/2 \\ \pm i/2 & 1/2 \end{array}\right) \left(\begin{array}{rr} 0 & 1 \\ -1 & 0 \end{array}\right) \left(\begin{array}{rr} 1 & \mp i \\ \mp i & 1 \end{array}\right) \left(\left(G_a^{1/2}\right)^+\right)^{\sigma_3}
\]
is equal to $\mathcal{I}$ by \eqref{secondroot}, the matrix function $\mathcal{E}$ is holomorphic in $U_{a,\delta}\setminus\{a\}$.  Now, the second part of \eqref{n-ends} and \eqref{matrixW} yield that all the entries of $\mathcal{NW}$ behave like $\mathcal{O}\big(|z-a|^{-1/4}\big)$ as $z\to a$. Hence, it follows from \eqref{Geat0} that the entries of $\mathcal{E}$ can have at most square-root singularity there, which is possible only if $\mathcal{E}$ is analytic in the whole disk $U_{a,\delta}$.

\subsubsection{Matrix Functions $\Psi$ and $\widetilde\Psi$}
The following construction was introduced in \cite[Theorem~6.3]{KMcLVAV04}. Let $I_\alpha$ and $K_\alpha$ be the modified Bessel functions and $H_\alpha^{(1)}$ and $H_\alpha^{(2)}$ be the Hankel functions \cite[Ch. 9]{AbramowitzStegun}. Set $\Psi$ to be the following sectionally holomorphic matrix function:
\[
\Psi(\zeta) = \Psi(\zeta;\alpha) := \left(
\begin{array}{cc}
I_\alpha\left(2\zeta^{1/2}\right) & \frac i\pi K_\alpha\left(2\zeta^{1/2}\right) \smallskip \\
2\pi i\zeta^{1/2} I_\alpha^\prime\left(2\zeta^{1/2}\right) & -2\zeta^{1/2} K_\alpha^\prime\left(2\zeta^{1/2}\right)
\end{array}
\right)
\]
for $|\Arg(\zeta)|<2\pi/3$;
\[
\Psi(\zeta) := \left(\begin{array}{cc}
\frac12 H_\alpha^{(1)}\left(2(-\zeta)^{1/2}\right) & \frac12 H_\alpha^{(2)}\left(2(-\zeta)^{1/2}\right) \smallskip \\
\pi\zeta^{1/2}\left(H_\alpha^{(1)}\right)^\prime\left(2(-\zeta)^{1/2}\right) & \pi\zeta^{1/2}\left( H_\alpha^{(2)}\right)^\prime\left(2(-\zeta)^{1/2}\right)
\end{array}
\right)e^{\frac12 \alpha\pi i\sigma_3}
\]
for $2\pi/3<\Arg(\zeta)<\pi$;
\[
\Psi(\zeta) := \left(\begin{array}{cc}
\frac12 H_\alpha^{(2)}\left(2(-\zeta)^{1/2}\right) & -\frac12 H_\alpha^{(1)}\left(2(-\zeta)^{1/2}\right) \smallskip \\
-\pi\zeta^{1/2}\left(H_\alpha^{(2)}\right)^\prime\left(2(-\zeta)^{1/2}\right) & \pi\zeta^{1/2}\left( H_\alpha^{(1)}\right)^\prime\left(2(-\zeta)^{1/2}\right)
\end{array}
\right)e^{-\frac12 \alpha\pi i\sigma_3}
\]
for $-\pi<\Arg(\zeta)<-2\pi/3$, where $\Arg(\zeta)\in(-\pi,\pi]$ is the principal determination of the argument of $\zeta$. Assume that the rays $I$, $I_+$, and $I_-$ defined in Section~\ref{sss:gfun} are oriented towards the origin. Using known properties of $I_\alpha$, $K_\alpha$, $H_\alpha^{(1)}$, $H_\alpha^{(2)}$, and their derivatives, it can be checked that $\Psi$ is the solution of the following Riemann-Hilbert problem \rhpsi:
\begin{itemize}
\item[(a)] $\Psi$ is a holomorphic matrix function in $\C\setminus(I\cup I_+\cup I_-)$;
\item[(b)] $\Psi$ has continuous traces on $I_+\cup I_-\cup I$ that satisfy
\[
\Psi_+ = \Psi_- \left\{
\begin{array}{ll}
\left(\begin{array}{cc} 1 & 0 \\ e^{\pm\alpha\pi i} & 1 \end{array}\right) & \mbox{on} \quad I_\pm \smallskip \\
\left(\begin{array}{cc} 0 & 1 \\ -1 & 0 \end{array}\right) & \mbox{on} \quad  I;
\end{array}
\right.
\]
\item[(c)] $\Psi$ has the following behavior near $0$:
\[
\left\{
\begin{array}{lll}
\displaystyle \mathcal{O}\left(\begin{array}{cc} |\zeta|^{\alpha/2} & |\zeta|^{\alpha/2} \\ |\zeta|^{\alpha/2} & |\zeta|^{\alpha/2} \end{array}\right)  & \mbox{if} \quad \alpha<0, & \mbox{as} \quad \zeta\to0, \smallskip \\
\displaystyle \mathcal{O}\left(\begin{array}{cc} \log|\zeta| & \log|\zeta| \\ \log|\zeta| & \log|\zeta| \end{array}\right) & \mbox{if} \quad \alpha=0, & \mbox{as} \quad \zeta\to0, \smallskip \\
\displaystyle \mathcal{O}\left(\begin{array}{cc} |\zeta|^{\alpha/2} & |\zeta|^{-\alpha/2} \\ |\zeta|^{\alpha/2} & |\zeta|^{-\alpha/2} \end{array}\right) & \mbox{if} \quad \alpha>0, & \mbox{as} \quad \zeta\to0 \quad \mbox{in} \quad |\Arg(\zeta)|<2\pi/3, \smallskip \\
\displaystyle \mathcal{O}\left(\begin{array}{cc} |\zeta|^{-\alpha/2} & |\zeta|^{-\alpha/2} \\ |\zeta|^{-\alpha/2} & |\zeta|^{-\alpha/2} \end{array}\right) & \mbox{if} \quad \alpha>0, & \mbox{as} \quad \zeta\to0 \quad \mbox{in} \quad 2\pi/3<|\Arg(\zeta)|<\pi.
\end{array}
\right.
\]
\item[(d)] $\Psi$ has the following behavior near $\infty$:
\[
\Psi(\zeta) = \left(2\pi\zeta^{1/2}\right)^{-\sigma_3/2} \frac{1}{\sqrt2}\left(
\begin{array}{cc}
1 & i \smallskip \\
i & 1
\end{array}
\right)
\left(\mathcal{I} + \mathcal{O}\left(\zeta^{-1/2}\right)\right)\exp\left\{2\zeta^{1/2}\sigma_3\right\}
\]
uniformly in $\C\setminus(I\cup I_+\cup I_-)$.
\end{itemize}

Finally, if we set $\widetilde \Psi := \sigma_3\Psi\sigma_3$. It can be readily checked that this matrix function satisfies \rhpsi~ with the orientations of the rays $I$, $I_+$, and $I_-$ reversed.

\subsubsection{Solution of \rhpa.}
With the notation introduced above, the following lemma holds.
\begin{lemma}
\label{lem:rhp}
For $n\in\N_\varepsilon$, a solution of \rhpa~ is given by
\begin{equation}
\label{eq:pe}
\mathcal{P}_a = \mathcal{E}\Psi\mathcal{W}^{-1}, \quad \zeta = (n/2)^2G_a^2,
\end{equation}
if $\Delta_k$ is oriented towards $a$ and with $\Psi$ replaced by $\widetilde\Psi$ otherwise, where $\Psi=\Psi(\cdot;\alpha_a)$.
\end{lemma}
\begin{proof}
Assume that $\Delta_k$, and respectively $\Delta_{k\pm}$, is oriented towards $a$. In this case $G_a^2$ preserves the orientation of these arcs and we use \eqref{eq:pe} with $\Psi$. The analyticity of $\mathcal{E}$ implies that the jumps of $\mathcal{P}_a$ are those of $\Psi\mathcal{W}^{-1}$. By the very definition of $G_a^2$ and $\Psi$, the latter has jumps only on $\Sigma\cap U_{a,\delta}$ and otherwise is holomorphic. This shows the validity of \rhpa(a). It also can be readily verified that \rhpa(b) is fulfilled by using \eqref{t2}, \eqref{we2}, and \eqref{webvp}. Next, observe that \rhpa(c) follows from \rhpsi(c) upon recalling that $|G_a^2(z)|\sim|z-a|$ and $|\mathcal{W}(z)|\sim|z-a|^{(\alpha_a/2)\sigma_3}$ as $z\to a$. Observe now that with $\zeta$ defined as in \eqref{eq:pe}, it holds by the definition of $\mathcal{E}$ and \rhpsi(d) that
\[
\mathcal{P}_a\mathcal{N}^{-1}-\mathcal{I} = \mathcal{NW}\exp\left\{-nG_a\sigma_3\right\}\mathcal{O}\left(\frac1n\right)\exp\left\{nG_a\sigma_3\right\}\mathcal{W}^{-1}\mathcal{N}^{-1} = \mathcal{N}\mathcal{O}\left(\frac1n\right)\mathcal{N}^{-1}
\]
on $\partial U_{a,\delta}$, where we also used \eqref{GeMap}. Multiplying the last three matrices out we get that the entires of thus obtained matrix contain all possible products of $S_n/S_n(\infty)$, $hS_n^*/S_n(\infty)$, $S_{n-1}/S_{n-1}^*(\infty)$, and $hS_{n-1}^*/S_{n-1}^*(\infty)$. Then it follows from \eqref{SnNormalized1} used with $\epsilon<\delta$ that the moduli of the entires of $\mathcal{P}_a\mathcal{N}^{-1}-\mathcal{I}$ are of order $\mathcal{O}\left(1/n\right)$ uniformly for $n\in\N_\varepsilon$. This finishes the proof of the lemma since the case where $\Delta_k$ is oriented away from $a$ can be examined analogously.
\end{proof}

\subsection{Local Analysis Near Trivalent Ends}

In this section we continue to investigate the behavior of $\mathcal{S}$ near the points in $E$. However, now we concentrate on the zeros of $h$, that is, the trivalent ends of $\Delta$. As in the construction of the lens $\Sigma$, let $\Delta_{b,k}$ be the arcs comprising $\Delta$ incident with $b$ which are numbered in the counter-clockwise fashion.

\subsubsection{Riemann-Hilbert Problem for Local Parametrix}
As before, we denote by $U_{b,\delta}$ a disk centered at $b$ of small enough radius $\delta$ (``small enough'' is specified as we proceed with the solution of \rhpb). We are seeking the solution of the following \rhpb:
\begin{itemize}
\item[(a)] $\mathcal{P}_b$ is analytic in $U_{b,\delta}\setminus\Sigma$;
\item[(b)] $\mathcal{P}_b$ has continuous traces on each side of $U_{b,\delta}\cap\Sigma$ that satisfy \rhs(b2) and (b4);
\item[(c)] $\mathcal{P}_b$ is bounded in the vicinity of $b$;
\item[(d)] $\mathcal{P}_b\mathcal{N}^{-1}=\mathcal{I}+\mathcal{O}(1/n)$ uniformly on $\partial U_{b,\delta}\setminus\Sigma$, where $\mathcal{N}$ is given by \eqref{eq:n}.
\end{itemize}
As in the case of \rhpa, we consider only the indices $n\in\N_\varepsilon$ and the estimate $\mathcal{O}(1/n)$ in \rhpb(d) is not uniform with respect to $\varepsilon$ or $\delta$ approaching zero.

\subsubsection{$G$-Function}
Set, for convenience, $S_{b,k}$ to be the sectorial subset of $U_{b,\delta}$ bounded by $\Delta_{b,k+1}$, $\Delta_{b,k-1}$, and $\partial U_{b,\delta}$. Define
\[
G_b(z) := (-1)^k\int_b^zh(t)dt, \quad z\in S_{b,k}.
\]
Thus defined, the function $G_b$ satisfies
\begin{equation}
\label{gemapsectors}
\left|\map b^{G_b}\right| \equiv 1 \quad \mbox{in} \quad S_{b,1}\cup S_{b,3} \quad \mbox{and} \quad \left|\map b^{-G_b}\right| \equiv 1 \quad \mbox{in} \quad S_{b,2}.
\end{equation}
The same reasoning as in \eqref{Gepm} yields that $G_b$ is a holomorphic function in $U_{b,\delta}\setminus\Delta_{b,2}$ whose traces on $\Delta_{b,1}\cap U_{b,\delta}$ and $\Delta_{b,3}\cap U_{b,\delta}$ as well as on both side of $\Delta_{b,2}\cap U_{b,\delta}$ are purely imaginary and $G_b^++G_b^-\equiv0$ on $\Delta_{b,2}\cap U_\delta^b$.  The last observation implies that $G_b^2$ is a holomorphic function in $U_{b,\delta}$ for $\delta$ small enough that assumes negative values on each $\Delta_{b,k}\cap U_{b,\delta}$.  

Recall that $h^2$ has a simple zero at $b$ and therefore $|h(z)|\sim|z-b|^{1/2}$ as $z\to b$. This, in turn, implies we can holomorphically define a cubic root of $G_b^2$ in $U_{b,\delta}$. In what follows, we set $G_b^{2/3}$ to be a conformal map of $U_{b,\delta}$ onto some neighborhood of the origin such that
\begin{equation}
\label{ik}
G_b^{2/3}\left(\Delta_{b,k}\cap U_{b,\delta}\right)\subset I_k,
\end{equation}
where we set $I_k:=\left\{z:\arg(z)=\pi (2k-1)/3\right\}$, $k\in\{1,2,3\}$, and these rays are oriented towards the origin. Moreover, since we had some freedom in choosing the arcs $\widetilde\Delta_{b,k}$, we require that
\begin{equation}
\label{tildeik}
G_b^{2/3}\left(\widetilde\Delta_{b,k}\cap U_{b,\delta}\right)\subset \widetilde I_k,
\end{equation}
where $\widetilde I_k:=\left\{z:\Arg(z)=2\pi(k-2)/3\right\}$, $k\in\{1,2,3\}$, and the rays are once again oriented towards the origin. Such a choice is always possible as $G_b^2$ maps the sector $S_{b,k}\supset\widetilde\Delta_{b,k}$ onto a neighborhood of the origin cut along $I_2$.

Finally, since the traces $G_b^\pm$ are purely imaginary on $\Delta_{b,2}\cap U_{b,\delta}$, satisfy $G_b^+=-G_b^-$ there, and the increment of the argument of $G_b$ is $3\pi$ when $b$ is encircled in the clockwise direction from a point on $\Delta_{b,2}\cap U_{b,\delta}$ back to itself, we can define the sixth root of $G_b$, which is equivalent to the fourth root of $G_b^{2/3}$, that satisfies
\begin{equation}
\label{sixthroot}
\left(G_b^{1/6}\right)^+ = \pm i\left(G_b^{1/6}\right)^- \quad \mbox{on} \quad \Delta_{b,2}\cap U_{b,\delta},
\end{equation}
where the $+$ sign must be used when $\Delta_{b,k}$ are oriented towards $b$ and the $-$ sign otherwise.

\subsubsection{Matrix Function $\mathcal{E}$} Recall our notation, $\rho_{b,k}=\rho_{|\Delta_{b,k}}$. For each function $\rho_{b,k}$, fix a continuous branch of the square root $\sqrt{\rho_{b,k}}$. Let $\mathcal{N}$ be the solution of \rhn~presented in Section~\ref{ss:6.1}. Set
\[
\mathcal{W} := \left\{
\begin{array}{ll}
\displaystyle \left(-i\map^n\frac{\sqrt{\rho_{b,3}}\sqrt{\rho_{b,2}}}{\sqrt{\rho_{b,1}}}\right)^{\sigma_3}, & \mbox{in} \quad S_{b,1}, \bigskip \\
\displaystyle \left(i\map^n\frac{\sqrt{\rho_{b,k-1}}\sqrt{\rho_{b,k+1}}}{\sqrt{\rho_{b,k}}}\right)^{\sigma_3}, & \mbox{in} \quad S_{b,2}\cup S_{b,3}.
\end{array}
\right.
\]
Then the matrix function $\mathcal{NW}$ is holomorphic in $U_{b,\delta}\setminus\cup_k\Delta_{b,k}$ and satisfies
\[
\mathcal{N_+}\mathcal{W_+} = \mathcal{N_-}\mathcal{W_-} \left\{
\begin{array}{ll}
\left(\begin{array}{rr} 0 & -1 \\ 1 & 0 \end{array}\right)  & \mbox{on} \quad \Delta_{b,1}\cap U_{b,\delta}, \bigskip \\
\left(\begin{array}{rr} 0 & 1 \\ -1 & 0 \end{array}\right)  & \mbox{on} \quad \left(\Delta_{b,2}\cup\Delta_{b,3}\right)\cap U_{b,\delta}.
\end{array}
\right.
\]
Further, put
\[
\mathcal{E^*} := \left\{
\begin{array}{ll}
\displaystyle \mathcal{NW}\exp\left\{nG_b\sigma_3\right\}, & \mbox{in} \quad S_{b,1}\cup S_{b,3}, \bigskip \\
\displaystyle \mathcal{NW}\left(\begin{array}{rr} 0 & \pm 1 \\ \mp 1 & 0 \end{array}\right)\exp\left\{nG_b\sigma_3\right\}, & \mbox{in} \quad S_{b,2},
\end{array}
\right.
\]
where we use the upper signs when the arcs $\Delta_k$ are oriented towards $b$ and the lower ones otherwise. Then $\mathcal{E^*}$ is a holomorphic matrix function in $U_{b,\delta}\setminus\Delta_{b,2}$ and
\[
\mathcal{E^*_+} = \mathcal{E^*_-}\left(\begin{array}{rr} 0 & 1 \\ -1 & 0 \end{array}\right) \quad \mbox{on} \quad \Delta_{b,2}\cap U_{b,\delta}
\]
since $G_b^+=-G_b^-$ there. Finally, define
\[
\mathcal{E} := \mathcal{E^*}\left(\begin{array}{rr} 1/2 & \mp 1/2 \\ \mp i/2 & -i/2 \end{array}\right)G_b^{\sigma_3/6}(3n/2)^{\sigma_3/6}
\]
where we use the $-$ sign when the arcs $\Delta_k$ are oriented towards $b$ and the $+$ sign otherwise. Since the product
\[
\left(\left(G_b^{1/6}\right)^-\right)^{-\sigma_3} \left(\begin{array}{rr} 1 & \pm i \\ \mp 1 & i \end{array}\right) \left(\begin{array}{rr} 0 & 1 \\ -1 & 0 \end{array}\right) \left(\begin{array}{rr} 1/2 & \mp1/2 \\ \mp i/2 & -i/2 \end{array}\right) \left(\left(G_b^{1/6}\right)^+\right)^{\sigma_3}
\]
is equal to $\mathcal{I}$ by \eqref{sixthroot} where we use the upper signs when the arcs $\Delta_k$ are oriented towards $b$ and the lower ones otherwise, the matrix function $\mathcal{E}$ is holomorphic in $U_{b,\delta}\setminus\{b\}$. Since $|G_b(z)|^{1/6}\sim|z-b|^{1/4}$ as $z\to b$ and by the first part of \eqref{n-ends}, the entries of $\mathcal{E}$ can have at most square-root singularity there. Therefore $\mathcal{E}$ is analytic in the whole disk $U_{b,\delta}$.

\subsubsection{Matrix Function $\Upsilon$}
The following construction is a modification of the one introduced in \cite[Section~7]{DKMLVZ99b}. Let $\Ai$ be the Airy function. Set
\[
\Upsilon_0(\zeta) := \left(
\begin{array}{cc}
\Ai(\zeta) & \Ai\left(e^{\frac{4\pi i}{3}}\zeta\right) \smallskip \\
\Ai^\prime(\zeta) & e^{\frac{4\pi i}{3}}\Ai^\prime\left(e^{\frac{4\pi i}{3}}\zeta\right)
\end{array}
\right)e^{-\frac{\pi i}{6}\sigma_3}
\]
and
\[
\Upsilon_1(\zeta) := \left(
\begin{array}{cc}
\Ai(\zeta) & -e^{\frac{4\pi i}{3}}\Ai\left(e^{\frac{2\pi i}{3}}\zeta\right) \smallskip \\
\Ai^\prime(\zeta) & -\Ai^\prime\left(e^{\frac{2\pi i}{3}}\zeta\right)\end{array}
\right)e^{-\frac{\pi i}{6}\sigma_3}.
\]
Further, put
\[
\Upsilon := \left\{
\begin{array}{ll}
\Upsilon_0\left(\begin{array}{rr}
0 & -1 \smallskip \\
1 & 0
\end{array}\right),  & \Arg(\zeta)\in\left(0,\frac\pi3\right), \smallskip \\
\Upsilon_0, & \Arg(\zeta)\in\left(\frac\pi3,\frac{2\pi}{3}\right), \smallskip \\
\Upsilon_0\left(\begin{array}{rr}
1 & 0 \smallskip \\
-1 & 1
\end{array}\right), & \Arg(\zeta)\in\left(\frac{2\pi}{3},\pi\right),
\end{array}\right. 
\;
:= \left\{
\begin{array}{ll}
\Upsilon_1\left(\begin{array}{rr}
0 & -1 \smallskip \\
1 & 0
\end{array}\right),  & \Arg(\zeta)\in\left(-\frac\pi3,0\right), \smallskip \\
\Upsilon_1, & \Arg(\zeta)\in\left(-\frac{2\pi}{3},-\frac\pi3\right), \smallskip \\
\Upsilon_1\left(\begin{array}{cc}
1 & 0 \smallskip \\
1 & 1
\end{array}\right), & \Arg(\zeta)\in\left(-\pi,-\frac{2\pi}{3}\right).
\end{array}\right. 
\]

It is known that
\[
\left\{
\begin{array}{lll}
\Ai(\zeta) &=& \frac{1}{2\sqrt\pi}\zeta^{-1/4}\exp\left\{-\frac23\zeta^{3/2}\right\}\left(1+\mathcal{O}(\zeta^{-3/2})\right) \bigskip \\
\Ai^\prime(\zeta) &=&-\frac{1}{2\sqrt\pi}\zeta^{1/4}\exp\left\{-\frac23\zeta^{3/2}\right\}\left(1+\mathcal{O}(\zeta^{-3/2})\right)
\end{array}
\right.
\]
as $\zeta\to\infty$ in the angle $|\Arg(\zeta)|<\pi$, from which it was deduced in \cite[Lemma 7.4]{DKMLVZ99b} that
\begin{equation}
\label{upsinfty}
\Upsilon(\zeta) = \frac{e^{-\frac{\pi i}{6}}}{2\sqrt\pi}\zeta^{-\sigma_3/4}\left(
\begin{array}{cc}
1  & i  \smallskip \\
-1 & i
\end{array}
\right)\left(
\mathcal{I} + \mathcal{O}\left(\zeta^{-3/2}\right)\right)\exp\left\{-\frac23\zeta^{3/2}\sigma_3\right\}
\end{equation}
as $\zeta\to\infty$ for $|\Arg(\zeta)|\in\left(\frac\pi3,\frac{2\pi}{3}\right)\cup\left(\frac{2\pi}{3},\pi\right)$. Asymptotics for $|\Arg(\zeta)|\in\left(0,\frac\pi3\right)$ can be obtained by multiplying the left-hand side of \eqref{upsinfty} by the matrix $\left(\begin{array}{rr}0 & -1 \smallskip \\1 & 0\end{array}\right)$ from the right.
\begin{figure}[!ht]
\centering
\includegraphics[scale=.5]{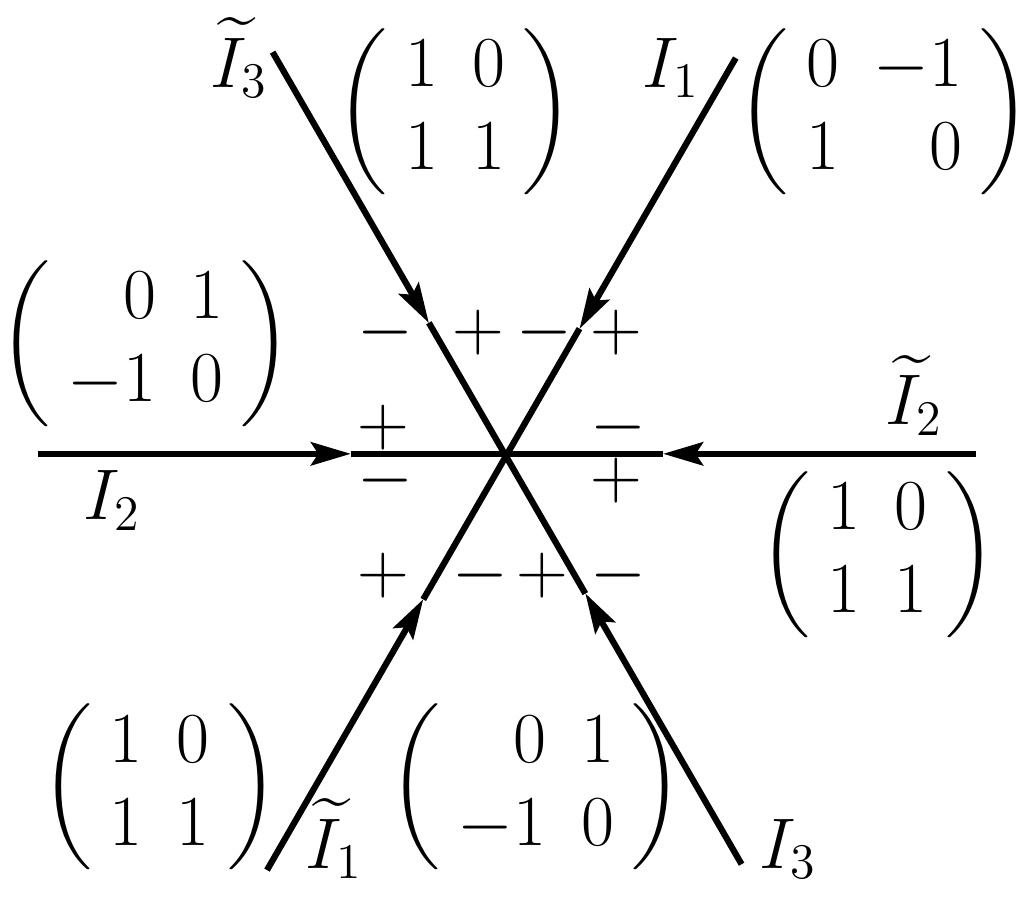}
\caption{\small The jump matrices that describe the relations between the traces of $\Upsilon$ on $\bigcup_k\left(I_k\cup \widetilde I_k\right)$.}
\label{fig:5}
\end{figure}
Altogether, it can be checked as in \cite[Section~7]{DKMLVZ99b} that $\Upsilon$ is the solution of the following Riemann-Hilbert problem \rhups:
\begin{itemize}
\item[(a)] $\Upsilon$ is a holomorphic matrix function in $\C\setminus\bigcup_k\left(I_k\cup \widetilde I_k\right)$;
\item[(b)] $\Upsilon$ has continuous traces on $\bigcup_k\left(I_k\cup \widetilde I_k\right)$ that satisfy the jump relations described by Figure~\ref{fig:5};
\item[(c)] each entry of $\Upsilon$ has a finite nonzero limit at the origin from within each sector;
\item[(d)] the behavior of $\Upsilon$ near $\infty$ is governed by \eqref{upsinfty}.
\end{itemize}

Finally, if we set $\widetilde \Upsilon := \sigma_3\Upsilon\sigma_3$. It can be readily checked that this matrix function satisfies \rhups~ with the orientations of the rays $I_k$ and $\widetilde I_k$ reversed.

\subsubsection{Solution of \rhpb}
With the notation introduced above, the following lemma holds.
\begin{lemma}
\label{lem:rhpet}
For $n\in\N_\varepsilon$, a solution of \rhpb~ is given by
\begin{equation}
\label{eq:pet}
\mathcal{P}_b = 2\sqrt\pi e^\frac{\pi i}{6}\mathcal{E}\Upsilon\mathcal{W}^{-1},  \quad \zeta= (3n/2)^{2/3}G_b^{2/3}, 
\end{equation}
if the arcs $\Delta_{b,k}$ are oriented towards $b$ and with $\Upsilon$ replaced by $\widetilde\Upsilon$ otherwise.
\end{lemma}
\begin{proof}
Assume that the arcs $\Delta_{b,k}$ are oriented towards $b$. As $\mathcal{E}$ is holomorphic in $U_{b,\delta}$, it can be readily verified using \eqref{ik} and \eqref{tildeik} that $\Upsilon\mathcal{W}^{-1}$ satisfies \rhpb(b). It is also evident that $\Upsilon\mathcal{W}^{-1}$ has no other jumps and therefore \rhpb(a) is fulfilled. Since all the matrices are bounded in the vicinity of $b$, so is \rhpb(c). Observe now that with $\zeta$ defined as in \eqref{eq:pet}, it holds by the definition of $\mathcal{E}$ and \eqref{upsinfty} that
\[
\mathcal{P}_b\mathcal{N}^{-1} - \mathcal{I} = \mathcal{NW}\exp\left\{nG_b\sigma_3\right\}\mathcal{O}\left(\frac1n\right)\exp\left\{-nG_b\sigma_3\right\}\mathcal{W}^{-1}\mathcal{N}^{-1}
\]
on $\partial U_{b,\delta}\cap\partial(S_{b,1}\cup S_{b,3})$ and
\[
\mathcal{P}_b\mathcal{N}^{-1} - \mathcal{I} = \mathcal{NW}\left(\begin{array}{rr}0 & 1 \smallskip \\-1 & 0\end{array}\right)\exp\left\{nG_b\sigma_3\right\}\mathcal{O}\left(\frac1n\right)\exp\left\{-nG_b\sigma_3\right\}\left(\begin{array}{rr}0 & -1 \smallskip \\1 & 0\end{array}\right)\mathcal{W}^{-1}\mathcal{N}^{-1}
\]
on $\partial U_{b,\delta}\cap\partial S_{b,2}$. As in the case of \rhpa(d), these representations yield \rhpb(d) on account of \eqref{gemapsectors} and \eqref{SnNormalized1} used with $\epsilon<\delta$. This finishes the proof of the lemma since the case where the arcs $\Delta_{b,k}$ are oriented away from $b$ can be examined analogously.
\end{proof}

\subsection{Final Transformation}

Denote by $\widetilde\Sigma$ the reduced system of contours that we define as
\[
\widetilde\Sigma := \left(\Sigma \setminus \left[\Delta\cup\bigcup_{e\in E} U_{e,\delta}\cup\bigcup_{k=1}^g\Delta_k^a\right]\right)\cup\bigcup_{e\in E} \partial U_{e,\delta}
\]
(see Figure~\ref{fig:6}).
\begin{figure}[!ht]
\centering
\includegraphics[scale=.75]{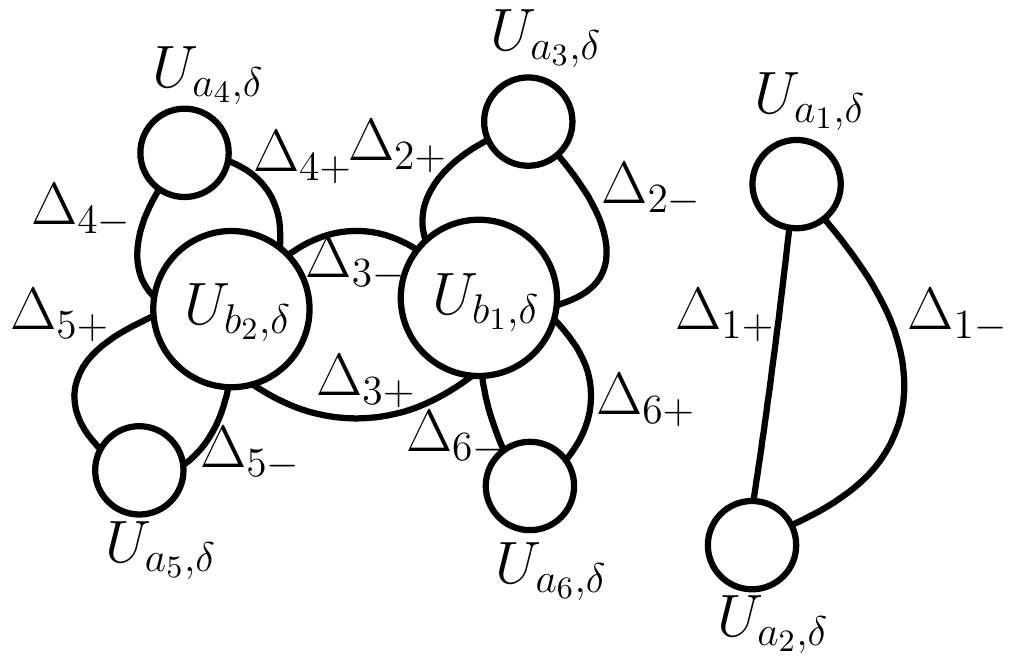}
\caption{\small Contour $\widetilde\Sigma$ for $\Sigma$ from Figure~\ref{fig:2}.}
\label{fig:6}
\end{figure}
For this new system we consider the following Riemann-Hilbert problem (\rhr):
\begin{itemize}
\item[(a)] $\mathcal{R}$ is a holomorphic matrix function in $\overline\C\setminus\widetilde\Sigma$ and $\mathcal{R}(\infty)=\mathcal{I}$;
\item[(b)] the traces of $\mathcal{R}$ on each side of $\widetilde\Sigma$ are continuous except for the branching points of $\widetilde\Sigma$, where they have definite limits from each sector and along each Jordan arc in $\widetilde\Sigma$. Moreover, they satisfy
\[
\mathcal{R}_+  = \mathcal{R}_- \left\{
\begin{array}{ll}
\mathcal{P}_e\mathcal{N}^{-1} & \mbox{on} \quad \partial U_{e,\delta} \quad \mbox{for each} \quad e\in E, \smallskip \\
\mathcal{N} \left(\begin{array}{cc} 1 & 0 \\ \map^{-2n}/\rho & 1 \end{array}\right) \mathcal{N}^{-1} & \mbox{on} \quad \widetilde\Sigma\setminus \bigcup_{e\in E}\partial U_{e,\delta}.
\end{array}
\right.
\]
\end{itemize}
Then the following lemma takes place.
\begin{lemma}
\label{lem:rhr}
The solution of \rhr~ exists for all $n\in\N_\varepsilon$ large enough and satisfies
\begin{equation}
\label{eq:r}
\mathcal{R}=\mathcal{I}+\mathcal{O}\left(1/n\right),
\end{equation}
where $\mathcal{O}(1/n)$ holds uniformly in $\overline\C$ (but not uniformly with respect to $\varepsilon$).
\end{lemma}
\begin{proof}
By \rhpa(d) and \rhpb(d), we have that \rhr(b) can be written as
\begin{equation}
\label{eq:smalljump}
\mathcal{R}_+ = \mathcal{R_-}\left(\mathcal{I}+\mathcal{O}\left(1/n\right)\right)
\end{equation}
uniformly on $\bigcup_{e\in E}\partial U_{e,\delta}$. Further, as $\map^{-2n}$ converges to zero geometrically fast away from $\Delta$, the jump of $\mathcal{R}$ on $\widetilde\Sigma\setminus\bigcup_{e\in E}\partial U_{e,\delta}$ is geometrically uniformly close to $\mathcal{I}$. Hence, \eqref{eq:smalljump} holds uniformly on $\widetilde\Sigma$. Thus, by \cite[Corollary~7.108]{Deift}, \rhr~ is solvable for all $n$ large enough and $\mathcal{R}_\pm$ converge to zero on $\widetilde\Sigma$ in $L^2$-sense as fast as $1/n$. The latter yields \eqref{eq:r} locally uniformly in $\overline\C\setminus\widetilde\Sigma$. To show that \eqref{eq:r} holds at $z\in\widetilde\Sigma$, deform $\widetilde\Sigma$ to a  new contour that avoids $z$. As the jump in \rhr~ is given by analytic matrix functions, one can state an equivalent problem on this new contour, the solution to which is an analytic continuation of $\mathcal{R}$. However, now we have that \eqref{eq:r} holds locally around $z$. Compactness of $\widetilde\Sigma$ finishes the proof of \eqref{eq:r}.
\end{proof}

Now, it can be verified directly from Lemmas \ref{lem:rhn}, \ref{lem:rhp}, and \ref{lem:rhr} that the following lemma holds.
\begin{lemma}
\label{lem:rhar}
The solution of \rhs~  exists for all $n\in\N_\varepsilon$ large enough and is given by
\begin{equation}
\label{eq:exps}
\mathcal{S} := \left\{
\begin{array}{ll}
\mathcal{R}\mathcal{N}, & \mbox{in} \quad \overline\C\setminus\left(\widetilde\Sigma\cup\bigcup_{e\in E} U_{e,\delta}\right), \bigskip \\
\mathcal{R}\mathcal{P}_e, & \mbox{in} \quad  U_{e,\delta} \quad \mbox{for each} \quad e\in E,
\end{array}
\right.
\end{equation}
where $\mathcal{R}$ is the solution of \rhr.
\end{lemma}

It is an immediate consequence of Lemmas \ref{lem:rht}, \ref{lem:rhs}, and \ref{lem:rhar} that the following result holds.

\begin{lemma}
\label{lem:solutionY}
If Condition~\hyperref[cond1]{GP} is fulfilled, then the solution of \rhy~ uniquely exists for all $n\in\N_\varepsilon$ large enough and can be expressed by reversing the transformations $\mathcal{Y}\to\mathcal{T}\to\mathcal{S}$ using \eqref{eq:t} and \eqref{eq:s} with $\mathcal{S}$ given by \eqref{eq:exps}.
\end{lemma}

\section{Asymptotics of Nuttall-Stahl Polynomials}
\label{s:7}

\subsection{Proof of Theorem~\ref{thm:SA}}
Assume that $n\in\N_\varepsilon$. For any given closed set in $D^*$, it can be easily arranged that this set lies exterior to the lens $\widetilde\Sigma$. Thus, the matrix $\mathcal{Y}$ on this closed set is given by
\[
\mathcal{Y} = \gamma_\Delta^{n\sigma_3}\mathcal{RN}\map^{n\sigma_3},
\]
where $\mathcal{R}$ is the solution of \rhr~ given by Lemma \ref{lem:rhar} and $\mathcal{N}$ is the solution of \rhn~ given by \eqref{eq:n}. Then
\begin{equation}
\label{eq:rnd}
\mathcal{RN} = \left(
\begin{array}{cc}
\displaystyle\left(1+\upsilon_{n1}\right)\frac{S_n}{S_n(\infty)}+\upsilon_{n2}\frac{\gamma_\Delta S_{n-1}}{\map S_{n-1}^*(\infty)} & \displaystyle \left(1+\upsilon_{n1}\right)\frac{hS_n^*}{S_n(\infty)}+\upsilon_{n2}\frac{\gamma_\Delta h\map S_{n-1}^*}{S_{n-1}^*(\infty)} \bigskip \\
\displaystyle \upsilon_{n3}\frac{S_n}{S_n(\infty)}+\left(1+\upsilon_{n4}\right)\frac{\gamma_\Delta S_{n-1}}{\map S_{n-1}^*(\infty)} & \displaystyle \upsilon_{n3}\frac{hS_n^*}{S_n(\infty)}+\left(1+\upsilon_{n4}\right)\frac{\gamma_\Delta h\map S_{n-1}^*}{S_{n-1}^*(\infty)}
\end{array}
\right)
\end{equation}
with $|\upsilon_{nk}|\leq c(\varepsilon)/n$ uniformly in $\overline\C$ by \eqref{eq:r}; and therefore \eqref{SA1} follow from \eqref{eq:y}.

To derive asymptotic behavior of $q_n$ and $R_n$ on $\Delta \setminus E$, we need to consider what happens within the lens $\Sigma$ and outside the disks $U_{e,\delta}$. We shall consider the asymptotics of $\mathcal{Y}$ from within $\bigcup_k\Omega_{k+}$, the ``upper'' part of the lens $\Sigma$, the behavior of $\mathcal{Y}$ in $\bigcup_k\Omega_{k-}$ can be derived in a similar fashion. We deduce from Lemma \ref{lem:solutionY} that
\[
\mathcal{Y}_+ = \gamma_\Delta^{n\sigma_3}\mathcal{(RN)}_+\left(\begin{array}{cc} 1 & 0 \\ (\map^+)^{-2n}/\rho & 1 \end{array}\right)(\map^+)^{n\sigma_3} = \gamma_\Delta^{n\sigma_3}\mathcal{(RN)}_+\left(\begin{array}{cc} (\map^+)^n  & 0 \smallskip \\ (\map^+)^{-n}/\rho & (\map^+)^{-n}  \end{array}\right)
\]
locally uniformly on $\Delta\setminus E$. Therefore, it holds that
\[
\left\{
\begin{array}{lcl}
\mathcal{Y}_{11} & = &  \displaystyle \mathcal{(RN)}_{11}^+(\gamma_\Delta\map^+)^n + \mathcal{(RN)}_{12}^+\gamma_\Delta^n(\map^+)^{-n}/\rho,  \bigskip \\
\mathcal{Y}_{12}^+ & =& \displaystyle \mathcal{(RN)}_{12}^+(\gamma_\Delta/\map^+)^n.
\end{array}
\right.
\]
Hence, we get \eqref{SA2} from \eqref{eq:rnd}, \eqref{eq:y}, and \eqref{jumpSn}.
\qed

\subsection{Proof of Corollary~\ref{cor:SA}}
Since $\widehat\rho-[n/n]_{\widehat\rho}=R_n/q_n$, it follows from \eqref{SA1} that
\[
\widehat\rho-[n/n]_{\widehat\rho} = \frac{S_n^*}{S_n}\frac{h}{\map^{2n}}\frac{1+\upsilon_{n1}+(\upsilon_{n2}\map)(\gamma_n^*/\gamma_n)(S_{n-1}^*/S_n^*)}{1+\upsilon_{n1}+(\upsilon_{n2}/\map)(\gamma_n^*/\gamma_n)(S_{n-1}/S_n)}.
\]
Since $\upsilon_{n2}$ is vanishing at infinity, it follows from \eqref{eq:r} that $|\upsilon_{n2}\map|,|\upsilon_{n2}/\map| \leq c(\varepsilon)/n$ uniformly in $D^*$. Thus, \eqref{errorasyspun} is the consequence of \eqref{SnNormalized}.
\qed

\bibliographystyle{plain}
\bibliography{../bibliography}

\end{document}